\documentclass[11pt,a4paper,english,reqno]{amsart}
\usepackage{lmodern}

\usepackage[T1]{fontenc}
\usepackage[latin9]{inputenc}
\usepackage{array}
\usepackage{verbatim}
\usepackage{longtable}
\usepackage{float}
\usepackage{booktabs}
\usepackage{multirow}
\usepackage{amstext}
\usepackage{amsthm}
\usepackage{amssymb}
\usepackage{stmaryrd}
\usepackage{graphicx}

\makeatletter

\pdfpageheight\paperheight
\pdfpagewidth\paperwidth

\providecommand{\tabularnewline}{\\}

\numberwithin{equation}{section}
\numberwithin{figure}{section}
\theoremstyle{plain}
\ifx\thechapter\undefined
\newtheorem{thm}{\protect\theoremname}
\else
\newtheorem{thm}{\protect\theoremname}[chapter]
\fi
  \theoremstyle{remark}
  \newtheorem{rem}[thm]{\protect\remarkname}
  \theoremstyle{definition}
  \newtheorem{defn}[thm]{\protect\definitionname}
  \theoremstyle{plain}
  \newtheorem{prop}[thm]{\protect\propositionname}
  \theoremstyle{remark}
  \newtheorem{notation}[thm]{\protect\notationname}
  \theoremstyle{plain}
  \newtheorem{fact}[thm]{\protect\factname}
  \theoremstyle{plain}
  \newtheorem{lem}[thm]{\protect\lemmaname}
  \theoremstyle{plain}
  \newtheorem{cor}[thm]{\protect\corollaryname}
  \theoremstyle{plain}
  \newtheorem{conjecture}[thm]{\protect\conjecturename}
  \theoremstyle{definition}
  \newtheorem{example}[thm]{\protect\examplename}

\usepackage{amscd}
\usepackage{graphics}
\usepackage{rotating,array}
\usepackage[all]{xy}
\usepackage[hmargin=2.5cm, vmargin=2.5cm]{geometry}
\usepackage[toc,page]{appendix}

\usepackage{bbding}
\usepackage[linkbordercolor={0 1 1}]{hyperref}

\usepackage{listings}

\@ifundefined{definecolor}
 {\usepackage{color}}{}
\usepackage{alltt}
\usepackage{times}

\allowdisplaybreaks

\makeatletter
\newcommand\appendix@section[1]{%
  \refstepcounter{section}%
  \orig@section*{Appendix \@Alph\c@section #1}%
  \addcontentsline{toc}{section}{Appendix \@Alph\c@section #1}%
}
\let\orig@section\section
\g@addto@macro\appendix{\let\section\appendix@section}

\def\th@plain{%
  \thm@notefont{}% same as heading font
  \thm@headfont{\normalfont\bfseries}
  \normalfont % body font
}
\def\th@definition{%
  \thm@notefont{}% same as heading font
  \thm@headfont{\normalfont\bfseries}
  \normalfont % body font
}

\renewenvironment{proof}[1][\textit{\proofname}]
{\par
  \pushQED{\qed}%
  \normalfont \topsep6\p@\@plus6\p@\relax
  \trivlist
  \item[\hskip\labelsep
        \normalfont
    #1\@addpunct{.}]\ignorespaces
}{%
  \popQED\endtrivlist\@endpefalse
}

\makeatother

\usepackage{perpage} % auf jeder Seite neu beginnen mit FN-Nummerierung:
\MakePerPage{footnote}

\usepackage[sf,bf,compact,topmarks,calcwidth,pagestyles]{titlesec}
\usepackage{titletoc}
\def\gobble#1{}
\def\cs#1{\expandafter\gobble\string\\#1}

\makeatother

\usepackage{textcomp,pslatex}
\widenhead{2.1pc}{0pc}
\titlelabel{\thetitle.\quad}
\contentsmargin{0pt}
\titlecontents{section}[1.8pc]
  {\addvspace{3pt}\bfseries}
  {\contentslabel[\thecontentslabel.]{1.8pc}}
  {}
  {}

\titlecontents*{subsection}[1.8pc]
  {\small}
  {\thecontentslabel. }
  {}
  {}
  [.---][.]

\AtBeginDocument{

}

\makeatother

\usepackage{babel}
  \providecommand{\conjecturename}{Conjecture}
  \providecommand{\corollaryname}{Corollary}
  \providecommand{\definitionname}{Definition}
  \providecommand{\examplename}{Example}
  \providecommand{\factname}{Fact}
  \providecommand{\lemmaname}{Lemma}
  \providecommand{\notationname}{Notation}
  \providecommand{\propositionname}{Proposition}
  \providecommand{\remarkname}{Remark}
\providecommand{\theoremname}{Theorem}

\begin{document}

\title{On the Support for Weight Modules of Affine Kac-Moody-Algebras}

\author{Thomas Bunke}

\address{Department Mathematik, FAU Erlangen-Nürnberg, Cauerstraße 11, 91058
Erlangen}

\email{thomas.bunke@gmail.com}
\begin{abstract}
\thanks{
programming)} An irreducible weight module of an affine Kac-Moody algebra $\mathfrak{g}$
is called dense if its support is equal to a coset in $\mathfrak{h}^{*}/Q$.
Following a conjecture of V. Futorny about affine Kac-Moody algebras
$\mathfrak{g}$, an irreducible weight $\mathfrak{g}$-module is dense
if and only if it is cuspidal (i.e. not a quotient of an induced module).
The conjecture is confirmed for $\mathfrak{g}=A_{2}^{\left(1\right)}$,
$A_{3}^{\left(1\right)}$ and$A_{4}^{\left(1\right)}$ and a classification
of the supports of the irreducible weight $\mathfrak{g}$-modules
obtained. For all $A_{n}^{\left(1\right)}$ the problem is reduced
to finding primitive elements for only finitely many cases, all lying
below a certain bound. For the left-over finitely many cases an algorithm
is proposed, which leads to the solution of Futorny's conjecture for
the cases $A_{2}^{\left(1\right)}$ and $A_{3}^{\left(1\right)}$.
Yet, the solution of the case $A_{4}^{\left(1\right)}$ required additional
combinatorics. 

A new category of hypoabelian subalgebras, pre-prosolvable subalgebras,
and a subclass thereof, quasicone subalgebras, is introduced and its
objects classified.
\end{abstract}

\maketitle

\section{Introduction}

\textbf{Survey.} Let $\mathfrak{g}$ be an affine Kac-Moody algebra
with Cartan subalgebra $\mathfrak{h}$, root system $\Delta$ and
center $\mathbb{C}c$. A $\mathfrak{g}$-module $V$ is called a \emph{weight
module }if $V=\bigoplus_{\lambda\in\mathfrak{h}^{*}}V_{\lambda}$,
where $V_{\lambda}$ denote the weight subspaces (also called weight
components) $V_{\lambda}=\left\{ v\in V\mid hv=\lambda\left(h\right)v\mbox{ for all }h\in\mathfrak{h}\right\} $.
We will also assume $V$ to have countable dimension. If $V$ is an
irreducible weight $\mathfrak{g}$-module then $c$ acts on $V$ as
a scalar, called\emph{ central charge} of $V$. For a weight $\mathfrak{g}$-module
$V$, the support is the set $\mathrm{supp}\left(V\right)=\left\{ \lambda\in\mathfrak{h}^{*}\mid V_{\lambda}\ne0\right\} $.
The root lattice $Q$ is the subgroup of $\mathfrak{h}^{*}$ generated
by $\Delta$. If $V$ is irreducible then $\mathrm{supp}\left(V\right)\subset\lambda+Q$
for some $\lambda\in\mathfrak{h}^{*}$. An irreducible weight $\mathfrak{g}$-module
$V$ is called \emph{dense }if its support is equal to a coset in
$\mathfrak{h}^{*}/Q$ and \emph{non-dense} if $\mathrm{supp}\left(V\right)\subsetneq\lambda+Q$;
a point of the set $\lambda+Q\smallsetminus\mathrm{supp}\left(V\right)$
will also be called a \emph{hole}.

Another criterion to classify modules is according to the way they
are constructed. There are two classes of irreducible weight $\mathfrak{g}$-modules,
those parabolically induced from other modules and those which are
not; we call the latter \emph{cuspidal} modules\index{cuspidal module}.
A result of V. Futorny and A. Tsylke \cite{FutornyTsylke} reduces
the classification of irreducible weight $\mathfrak{g}$-modules with
finite-dimensional weight spaces to the classification of irreducible
cuspidal modules over Levi subalgebras. Any such module is a quotient
of a module induced from an irreducible cuspidal module over a finite-dimensional
reductive Lie subalgebra. The pending conjecture that connects the
two approaches is as follows.

\textbf{Conjecture 1.} An irreducible weight $\mathfrak{g}$-module
is dense if and only if it is cuspidal \cite{Futorny1997}.

The property of a weight module $V$ with finite-dimensional weight
spaces being cuspidal is, for a reductive Lie algebra $\mathfrak{g}$,
equivalent to the statement that all root operators act injectively
on $V$ (cf. \cite[Corollary 3.7]{DimitrovMathieuPenkov}). If $\mathfrak{g}$
admits only cuspidal modules with finite-dimensional weight spaces
then all simple components of $\mathfrak{g}$ are of type $A$ and
$C$ \cite{Fernando1990}.

The classification of all simple weight modules with finite-dimensional
weight spaces over affine Lie algebras is work-in-progress \cite{DimitrovGrantscharov}.
If we omit the requirement of finite-dimensional weight subspaces,
the achievement of a complete classification is much more difficult,
because the method of (twisted) localization, developed by V. V. Deodhar
and T. Enright \cite[and references therein]{Deodhar1980} and successfully
applied by O. Mathieu for the simple Lie algebra case \cite{Mathieu2000}
and by I. Dimitrov and D. Grantcharov in the affine case \cite{DimitrovGrantscharov},
is not applicable. 

The classification for non-dense irreducible $A_{1}^{\left(1\right)}$-modules
with a finite-dimensional weight subspace has been completed by V.
Futorny \cite{Futorny1996}. The classification problem of non-zero
central charge modules with all finite-dimensional weight subspaces
is solved for all affine Kac-Moody algebras \cite{FutornyTsylke}.
In these cases, an irreducible module is either a quotient of a classical
Verma module, or of a generalized Verma module, or of a loop module
(induced from a Heisenberg subalgebra). An important tool is the concept
of primitive vectors. 

A \emph{primitive} \emph{vector} is a vector $v$ of a weight $\mathfrak{g}$-module
with the following property: there exists a parabolic subalgebra $\mathfrak{p}$
with Levi decomposition $\mathfrak{p}=\mathcal{L}\oplus\mathcal{N}$
such that $\mathcal{N}$ acts trivially on $v$. This primitive vector
generates an irreducible quotient of a classical Verma module, a generalized
Verma type module or a generalized loop module, depending on the the
type of $\mathfrak{p}$ \cite{Futorny1994}. In the well-studied case
of a classical Verma module, $\mathfrak{p}$ is just a Borel subalgebra
\cite{Kac1983}. Equivalent to Conjecture 1, there is

\textbf{Conjecture 1'.} Every non-dense weight $\mathfrak{g}$-module
$V$ contains a primitive vector. 

A proof of the conjectures is an important step towards the classification
of irreducible weight $\mathfrak{g}$-modules.

\textbf{Conventions.} Denote by $\mathbb{C}$ the complex numbers
and by $\mathbb{Z}_{\ge k}$ the set $\left\{ k,k+1,\dots\right\} $,
by $\mathbb{Z}_{+}=\mathbb{Z}_{\ge1}$ and $\mathbb{N}_{0}$ for $\mathbb{Z}_{\ge0}$.
The difference of sets is written $A\smallsetminus B=\left\{ x\in A\mid x\notin B\right\} $.

\section*{\textbf{Acknowledgments}}

I am very grateful to Slava Futorny for his encouragement and discussions.

\section{Affine Lie Algebras}

\subsection{Prerequisites}

Let $\mathfrak{g}^{\circ}$ be a simple finite dimensional complex
Lie algebra over $\mathbb{C}$ with a non-degenerate invariant symmetric
bilinear form $\kappa:\mathfrak{g}^{\circ}\times\mathfrak{g}^{\circ}\rightarrow\mathbb{C}$.
Let $\mathfrak{h}^{\circ}$ be Cartan subalgebra and $\Delta^{\circ}$
the root system with respect to $\mathfrak{h}^{\circ}$. The loop
algebra $\mathcal{L}\left(\mathfrak{g}^{\circ}\right)=\mathfrak{g}^{\circ}\otimes\Bbbk\left[t^{-1},t\right]$
has a $\mathrm{span}_{\mathbb{Z}}\Delta^{\circ}\times\mathbb{Z}$-grading
and a double extension 
\[
\mathfrak{g}=\widehat{\mathcal{L}}\left(\mathfrak{g}^{\circ}\right)=\left(\mathcal{L}\left(\mathfrak{g}^{\circ}\right)\oplus_{\omega_{D}}\mathbb{C}\right)\rtimes_{\widetilde{D}}\mathbb{C}=\mathcal{L}\left(\mathfrak{g}^{\circ}\right)\oplus\mathbb{C}c\oplus\mathbb{C}d,
\]
 for $\omega_{D}\left(x,y\right)=\tilde{\kappa}\left(Dx,y\right)$,
$\tilde{\kappa}\left(Dx\otimes t^{n},y\otimes t^{m}\right)=\delta_{n+m,0}\kappa^{\circ}\left(x,y\right)$
and $D\left(x\otimes t^{n}\right):=nx\otimes t^{n}$. The Lie bracket
in $\mathfrak{g}$ is given by
\begin{eqnarray*}
 &  & \left[x_{1}\otimes t^{n_{1}}+a_{1}c+\xi_{1}d,x_{2}\otimes t^{n_{2}}+a_{2}c+\xi_{2}d\right]=\\
 &  & =\left[x_{1},x_{2}\right]\otimes t^{n_{1}+n_{2}}+\xi_{1}n_{2}\left(x_{2}\otimes t^{n_{2}}\right)-\xi_{2}n_{1}\left(x_{1}\otimes t^{n_{1}}\right)+n_{1}\kappa^{\circ}\left(x_{1},x_{2}\right)\delta_{n_{1}+n_{2},0}c.
\end{eqnarray*}
 and the linear functionals $n\delta\:\left(n\in\mathbb{Z}\right)$,
with $\delta\left(d\right)=1$ and $\delta\mid_{\mathfrak{h}^{\circ}\oplus\mathbb{C}c}=\left\{ 0\right\} $
are roots \cite{Kac1983,MoodyPianzola}. The Cartan algebra of $\mathfrak{g}$
is $\mathfrak{h}=\mathfrak{h}^{\circ}\oplus\mathbb{C}d\oplus\mathbb{C}d$.

The invariant form restricted to $\mathfrak{h}$ is non-degenerate.
There is thus an injective map $\flat:\mathfrak{h}\to\mathfrak{h}^{*},h\mapsto h^{\flat},h^{\flat}\left(x\right)=\kappa\left(x,h\right)$.
For $\alpha\in\mathfrak{h}^{\flat}=\flat\left(\mathfrak{h}\right)$
we put $\alpha^{\sharp}=\flat^{-1}\left(\alpha\right)$ and define
a symmetric bilinear form on $\mathfrak{h}^{\flat}$ by $\left(\alpha,\beta\right)=\kappa\left(\alpha^{\sharp},\beta^{\sharp}\right)$.

The affine Weyl group $\mathcal{W}$ is generated by the set of reflections
\[
r_{\alpha}\left(\lambda\right)=\lambda-\cfrac{2\left(\lambda^{\circ},\alpha\right)}{\left(\alpha,\alpha\right)}\alpha,\:\alpha\in\Delta,\:\lambda=\lambda^{\circ}+zc+td\in\mathfrak{h}^{*},\:z,t\in\mathbb{C}.
\]
If $r_{\alpha}$ is the identity on $\left(\mathfrak{h}^{\circ}\right)^{*}$,
then $\alpha$ is an imaginary root, otherwise it is a real root.
The set of roots is the disjoint unit of imaginary and real roots,
i.e. $\Delta=\Delta_{\mathrm{im}}\dot{\cup}\Delta_{\mathrm{re}}$.
For $\alpha\in\left(\mathfrak{h}^{\circ}\right)^{*}$ the translation
with respect to $\alpha$ is the operator $t_{\alpha}$ acting on
$\mathfrak{h}^{*}$ by 
\begin{equation}
t_{\alpha}\left(\lambda\right)=\lambda+\lambda\left(c\right)\alpha-\left(\left(\lambda,\alpha\right)+\frac{1}{2}\left(\alpha,\alpha\right)\lambda\left(c\right)\right)\delta\quad\left(\lambda\in\mathfrak{h}^{*}\right).\label{eq:translation}
\end{equation}

 A subset $P\subset\Delta$ is called\label{def:partition}\emph{
additively closed,} if $\left(P+P\right)\cap\Delta\subset P$. An
additively closed subset is a \emph{parabolic system,} if $P\ne\Delta$
and $P\cup-P=\Delta$. A subset $P\subset\Delta$ is a\emph{ positive
system,} if $\mathrm{span}_{\mathbb{N}_{0}}P\cap-\mathrm{span}_{\mathbb{N}_{0}}P=\left\{ 0\right\} $
and $P\cup-P=\Delta$. Two subsets are called \emph{equivalent }if
they lie in the same $\mathcal{W}\times\left\{ \pm1\right\} $-orbit. 

From \cite[Ch. 2]{Futorny1997} we know that there exists a finite
number of pairwise non-equivalent positive systems of the root system
of $\mathfrak{g}$. The positive systems $P=\left(\left\{ \alpha+\mathbb{Z}\delta\mid\alpha\in\Delta_{+}^{\circ}\right\} \cap\Delta\right)\cup\mathbb{Z}_{+}\delta$
and $\Delta^{+}=\left(\left\{ \alpha+\mathbb{Z}_{\ge0}\delta\mid\alpha\in\Delta_{+}^{\circ}\right\} \cup\left\{ -\alpha+\mathbb{Z}_{+}\delta\mid\alpha\in\Delta_{+}^{\circ}\right\} \cup\mathbb{Z}_{+}\delta\right)\cap\Delta$
are non-equivalent. 
\begin{rem}
In the literature $P$ is also called \emph{imaginary parabolic} \emph{partition}
of $\Delta$, as related to the natural Borel subalgebra and imaginary
Verma modules both introduced later. The set $\Delta^{+}$ is called
\emph{standard} (or \emph{classical}) \emph{parabolic} \emph{partition}.
Any other positive system that is not equivalent to $P$ or $\Delta^{+}$
will be labeled \emph{mixed type}. 
\end{rem}

Kac and Jacobson \cite{JacobsonKac1985}, and independently the exposition
of V. Futorny \cite{Futorny1992} have determined a positive system
of a finite-rank root system $\Delta$ uniquely by means of characteristic
functionals. The latter exposition calls a positive system a parabolic
partition.

Let $\Pi$ be a basis for the root system and $\Pi^{*}$ be the dual
basis, defined by $\alpha^{*}\left(\beta\right)=\delta_{\alpha,\beta}$
for all $\alpha^{*}\in\Pi^{*}$ and all $\beta\in\Pi$. There exists
a pair of basis $\left(\Pi_{\delta},\Pi_{\delta}^{*}\right)$, where
$\Pi_{\delta}^{*}$ contains $\delta^{*}$ and $\Pi_{\delta}$ contains
$\delta$. We denote the coefficients of $\delta=\sum_{\alpha\in\Pi}k_{\alpha}\alpha$
with respect to such a base change by $k_{\alpha}\in\mathbb{N}_{0}$,
$\alpha\in\Pi$. Define furthermore weights $\omega_{\alpha}:\mathfrak{h}\rightarrow\Bbbk$
by $\omega_{\alpha}\left(\check{\beta}\right)=\delta_{\alpha,\beta}$
for $\alpha,\beta\in\Pi$. Then the set $\left(\omega_{\alpha}\right)_{\alpha\in\Pi}$
is the set of \emph{fundamental weight}s. 

From the above it follows, that $2\alpha^{\sharp}=\left(\alpha,\alpha\right)\check{\alpha}$.
A weight $\lambda\in\mathfrak{h}^{*}$ is positive, if it is a positive
linear combination of fundamental weights. Thus if $\lambda$ is positive
with respect to $\Pi$, then $\lambda\left(\alpha^{\sharp}\right)$
is also positive for all $\alpha\in\Pi$, unless $\alpha$ is an isotropic
root. For any weight, $\ker\left(\lambda\circ\sharp\right)=\left(\ker\lambda\right)^{\flat}$. 

Then we may define the weights $\lambda_{\pm X}=\pm\sum_{\alpha\in X}\omega_{\alpha}$
for all $X\subset\Pi\smallsetminus\left\{ \alpha_{0}\right\} =:\Pi^{\circ}$
and 
\begin{equation}
\phi_{X}=\begin{cases}
\sum_{\alpha\in\Pi^{\circ}\smallsetminus X}\omega_{\alpha}-\left(\sum_{\alpha\in\Pi^{\circ}\smallsetminus X}k_{\alpha}\right)\omega_{\alpha_{0}} & \mbox{ if }X\ne\Pi^{\circ}\\
\lambda_{\Pi^{\circ}} & \mbox{ if }X=\Pi^{\circ}.
\end{cases}\label{eq:hyperplane defining weights}
\end{equation}

For a pair of weights $\left(\lambda_{1},\lambda_{2}\right)$, define
\begin{equation}
\Delta_{+}\left(\lambda_{1},\lambda_{2}\right)=\Delta_{+}^{\flat}\left(\lambda_{1}\right)\cup\left(\Delta_{0}^{\flat}\left(\lambda_{1}\right)\cap\Delta_{+}^{\flat}\left(\lambda_{2}\right)\right)\label{eq:positive system by tuple}
\end{equation}
 and for $X\subset\Pi^{\circ}$, set $\Delta_{+}\left(X\right)=\Delta_{+}\left(\phi_{X},\phi_{\Pi}\right).$

Note that this is consistent with the definition of $\Delta_{+}\left(\Pi^{\circ}\right)$
as $\Delta_{+}\left(\Pi^{\circ}\right)=\mathrm{span}_{\mathbb{N}}\Pi^{\circ}\cap\Delta$.
The following theorem tells us that the equivalence classes of positive
systems are parametrized by the sets $X\subset\Pi^{\circ}$.
\begin{thm}
\label{thm:positive system by tuple}\cite{Futorny1997} If $P$ is
a positive system of an affine root system $\Delta$, then there is
a set $X\subset\Pi^{\circ}$ such that $P$ is $\mathcal{W}\times\left\{ \pm1\right\} $-equivalent
to \textup{$\Delta_{+}\left(X\right)$}.
\end{thm}

\subsection{\label{sec:2.2}Triangular decompositions and parabolic systems}

If $P$ is a parabolic system for $\mathfrak{g}$ with Cartan algebra
$\mathfrak{h}$ and root system $\Delta$, then we can define $\mathfrak{p}\left(P\right)=\mathfrak{h}+\sum_{\alpha\in P}\mathfrak{g}_{\alpha}=\mathfrak{g}_{0}+\mathfrak{g}_{P}$
which is a subalgebra of $\mathfrak{g}$. It turns out to be a parabolic
subalgebra in the commonly defined sense. 
\begin{defn}
\noindent A triple $\left(\mathfrak{g}_{+},\mathfrak{g}_{0},\mathfrak{g}_{-}\right)$
of subalgebras of $\mathfrak{g}$ defines a\emph{ split triangular
decomposition}, if there exist subsets $\Delta_{\pm},\Delta_{0}\subseteq\Delta$
such that (i) $\Delta=\Delta_{+}\dot{\cup}\Delta{}_{0}\dot{\cup}\Delta_{-}$
is a partition of $\Delta$, (ii) $\mathfrak{g}_{\pm}=\sum_{\alpha\in\Delta_{\pm}}\mathfrak{g}_{\alpha}$
and $\mathfrak{g}_{0}=\mathfrak{h}+\sum_{\alpha\in\Delta_{0}}\mathfrak{g}_{\alpha}$,
(iii) $\left[\mathfrak{g}_{0},\mathfrak{g}_{\pm}\right]\subset\mathfrak{g}_{\pm}$,
and (iv): if $\alpha_{1},\dots,\alpha_{n}\in\Delta_{+}$ or $\alpha_{1},\dots,\alpha_{n}\in\Delta_{-}$,
then $\sum_{i=1}^{n}\alpha_{i}\ne0$ for $n>0$.

A parabolic system $P$ is called \emph{principal}, if there exists
a split triangular decomposition such that $P=\Delta_{+}\dot{\cup}\Delta{}_{0}$
\cite{DimitrovGrantscharov}.
\end{defn}

Every linear functional $\lambda\in\mathfrak{h}^{*}$ determines a
split triangular decomposition by putting 
\begin{equation}
\Delta_{\pm}^{\flat}\left(\lambda\right)=\left\{ \pm\alpha\in\Delta\mid\lambda\left(\alpha^{\sharp}\right)>0\right\} \quad\mbox{and}\quad\Delta_{0}^{\flat}\left(\lambda\right)=\left\{ \alpha\in\Delta\mid\lambda\left(\alpha^{\sharp}\right)=0\right\} .\label{eq: triangular weight decomposition}
\end{equation}
Clearly $\Delta_{\pm}^{\flat}\left(\lambda\right)\dot{\cup}\Delta_{0}^{\flat}\left(\lambda\right)$
is a parabolic system, which are all principal. Recall that
\[
\Delta_{\pm}^{\flat}\left(\lambda\right)=\left(\lambda^{-1}\left(\mathbb{Z}_{\pm}\right)\right)^{\flat}\cap\Delta\quad\mbox{and}\quad\Delta_{0}^{\flat}\left(\lambda\right)=\left(\ker\lambda\right)^{\flat}\cap\Delta.
\]
The scalar $\lambda\left(\delta^{\sharp}\right)=\lambda\left(c\right)$
is called the central charge of $\lambda$.

The theory of parabolic systems of finite rank root systems is completely
governed by a pair of subsets, $S$ and $X$ $\subset\Pi^{\circ}$,
and can be described in terms of three weights. To begin with, define
\begin{equation}
P\left(\lambda_{1},\lambda_{2},\lambda_{3}\right)=\Delta_{+}^{\flat}\left(\lambda_{1}\right)\cup\left(\Delta_{0}^{\flat}\left(\lambda_{1}\right)\cap\left(\Delta_{+}^{\flat}\left(\lambda_{2}\right)\cup\Delta_{+}^{\flat}\left(\lambda_{3}\right)\right)\right)\label{eq:parabolic from triple}
\end{equation}
for a triple of weights $\left(\lambda_{1},\lambda_{2},\lambda_{3}\right)$.
With \ref{eq:hyperplane defining weights}, we can write
\begin{align}
 & P\left(X,-S\right)=P\left(\phi_{X},\lambda_{-S},\phi_{\Pi^{\circ}}\right),\;P_{S}\left(\Pi^{\circ}\right)=P\left(\phi_{\Pi^{\circ}},\lambda_{-S},\lambda_{-\left\{ \alpha_{0}\right\} }\right)\;\text{ and }\;P_{S}=P\left(\lambda_{\Pi^{\circ}\smallsetminus S},\phi_{\Pi^{\circ}},\phi_{-\Pi^{\circ}}\right).\label{eq:parabolic_sets_new}
\end{align}
\begin{thm}
Any parabolic system which is not a positive system is equivalent
to either\\
\begin{minipage}[t]{0.9\columnwidth}%
\begin{itemize}
\item[(i) ]$P\left(\phi_{X},\lambda_{-S},\lambda\right)$ with 

\begin{itemize}
\item[(a) ]$\emptyset\subsetneq S\subset X\subsetneq\Pi^{\circ}$ and $\lambda=\phi_{\Pi^{\circ}}$
or with 

\item[(b) ]$S\subsetneq X=\Pi^{\circ}$ and $\lambda=\lambda_{\left\{ -\alpha_{0}\right\} }$,
or to

\end{itemize}

\item[(ii) ] $P\left(\lambda_{\Pi^{\circ}\smallsetminus S},\phi_{\Pi^{\circ}},\phi_{-\Pi^{\circ}}\right)$
with $S\subset\Pi^{\circ}$.

\end{itemize}%
\end{minipage} 

If the parabolic system does not contain $-\delta$, it falls in the
classes (i), if it contains $-\delta$, in class (ii).
\end{thm}

\begin{proof}
We need to show that the sets \ref{eq:parabolic_sets_new} coincide
with the ones defined in \cite[Sec. 2.]{Futorny1997}: 
\begin{align*}
P\left(X,S\right) & =\Delta_{+}\left(X\right)\cup\mathrm{add}_{\Delta}\left(-S\right)\quad\left(S\subset X\subset\Pi^{\circ},\,S\ne\emptyset\right)\\
P_{S}\left(\Pi^{\circ}\right) & =\mathrm{add}_{\Delta}\left(\Delta_{+}\left(\Pi^{\circ}\right)\cup\left(-S\right)\cup\left\{ -\alpha_{0}\right\} \right)\\
P_{S} & =\mathrm{add}_{\Delta}\left(\Delta_{+}\left(\emptyset\right)\cup\left(-S\right)\cup\left\{ -\delta\right\} \right)
\end{align*}

With tuples of weights, we have defined $\Delta_{+}\left(\lambda_{1},\lambda_{2}\right)=\Delta_{+}^{\flat}\left(\lambda_{1}\right)\cup\left(\Delta_{0}^{\flat}\left(\lambda_{1}\right)\cap\Delta_{+}^{\flat}\left(\lambda_{2}\right)\right)$
in \ref{eq:positive system by tuple} and inside the right hand side
we find 
\[
\Delta_{+}\left(X\right)=\Delta_{+}\left(\phi_{X},\phi_{\Pi^{\circ}}\right)=\Delta_{+}^{\flat}\left(\phi_{X}\right)\cup\left(\Delta_{0}^{\flat}\left(\phi_{X}\right)\cap\Delta_{+}^{\flat}\left(\phi_{\Pi^{\circ}}\right)\right).
\]
 Plugging \ref{eq:parabolic from triple} in \ref{eq:parabolic_sets_new},
for the left hand side 
\[
P\left(X,S\right)=\Delta_{+}^{\flat}\left(\phi_{X}\right)\cup\left(\Delta_{0}^{\flat}\left(\phi_{X}\right)\cap\left(\Delta_{0}^{\flat}\left(\phi_{-S}\right)\cup\Delta_{+}^{\flat}\left(\phi_{\Pi^{\circ}}\right)\right)\right),
\]
we obtain as the most inner term: $\Delta_{+}^{\flat}\left(\lambda_{-S}\right)\cup\Delta_{+}^{\flat}\left(\phi_{\Pi^{\circ}}\right)$.
This term contains both, $\mathrm{add}_{\Delta}\left(-S\right)$ and
$\Delta_{+}^{\flat}\left(\phi_{\Pi}\right)$. Since $S\subset X$
implies $\mathrm{add}_{\Delta}\left(-S\right)\subset\Delta_{+}^{\flat}\left(\phi_{X}\right)\cup\Delta_{0}^{\flat}\left(\phi_{X}\right)$,
the first coincidence follows.

For $P_{S}\left(\Pi^{\circ}\right)$ it is sufficient to see that
$\Delta_{0}^{\flat}\left(\phi_{\Pi^{\circ}}\right)\cap\left(\Delta_{+}^{\flat}\left(\lambda_{-S}\right)\cup\Delta_{+}^{\flat}\left(\lambda_{\left\{ -\alpha_{0}\right\} }\right)\right)$
contains\linebreak{}
 $\mathrm{add}_{\Delta}\left(\Delta^{\circ}\left(S\right)\cup\left\{ -\alpha_{0}\right\} \right)$,
taking into consideration the fact that $-\alpha_{0}=\frac{1}{a_{0}}\left(\theta-\delta\right)\in\Delta_{0}^{\flat}\left(\phi_{\Pi^{\circ}}\right)$.

For $P_{S}$, we have $\Delta_{0}^{\flat}\left(\lambda_{S^{c}}\right)=\Delta_{0}^{\flat}\left(\sum_{\alpha\in\Pi^{\circ}\smallsetminus S}\omega_{\alpha}\right)=\left(\mathrm{add}_{\Delta}\left(\pm S\right)+\mathbb{Z}\delta\right)\cup\mathbb{Z}_{\ne0}\delta$
and 
\begin{eqnarray*}
\Delta_{+}^{\flat}\left(\lambda_{S^{c}}\right) & = & \left(\Delta_{+}^{\circ}\smallsetminus\mathrm{add}_{\Delta^{\circ}}S+\Delta_{0}^{\flat}\left(\lambda_{S^{c}}\right)\cup\left\{ 0\right\} \right)\cap\Delta\\
 & = & \Delta_{+}^{\circ}\smallsetminus\mathrm{add}_{\Delta^{\circ}}S+\left(\left(\mathrm{add}_{\Delta}\left(\pm S\right)+\mathbb{Z}\delta\right)\cup\mathbb{Z}_{\ne0}\delta\cup\left\{ 0\right\} \right)\cap\Delta\\
 & = & \left(\Delta_{+}^{\circ}\smallsetminus\mathrm{add}_{\Delta^{\circ}}S+\mathrm{add}_{\Delta^{\circ}}S\right)\cap\Delta^{\circ}+\mathbb{Z}\delta,
\end{eqnarray*}
 furthermore, 
\[
\Delta_{+}^{\flat}\left(\phi_{\Pi^{\circ}}\right)\cup\Delta_{+}^{\flat}\left(\phi_{-\Pi^{\circ}}\right)=\Delta,
\]
because $\Delta_{+}^{\flat}\left(\phi_{\Pi^{\circ}}\right)$ is a
positive system by \ref{thm:positive system by tuple}. Thus, the
left hand side reads
\begin{eqnarray*}
 &  & \Delta_{+}^{\flat}\left(\lambda_{S^{c}}\right)\cup\left(\Delta_{0}^{\flat}\left(\lambda_{S^{c}}\right)\cap\left(\Delta_{0}^{\flat}\left(\phi_{\Pi^{\circ}}\right)\cup\Delta_{+}^{\flat}\left(\phi_{-\Pi^{\circ}}\right)\right)\right)\\
 &  & =\left(\Delta_{+}^{\circ}\smallsetminus\mathrm{add}_{\Delta^{\circ}}S+\mathrm{add}_{\Delta^{\circ}}S\right)\cap\Delta^{\circ}+\mathbb{Z}\delta\cup\left(\left(\left(\mathrm{add}_{\Delta}\left(\pm S\right)+\mathbb{Z}\delta\right)\cup\mathbb{Z}_{\ne0}\delta\right)\cap\Delta\right)\\
 &  & =\left(\left(\Delta_{+}^{\circ}\smallsetminus\mathrm{add}_{\Delta^{\circ}}S+\mathrm{add}_{\Delta^{\circ}}S\right)\cap\Delta^{\circ}+\mathbb{Z}\delta\right)\cup\left(\mathrm{add}_{\Delta}\left(\pm S\right)+\mathbb{Z}\delta\right)\cup\mathbb{Z}_{\ne0}\delta\\
 &  & =\left(\Delta_{+}^{\circ}\cup\mathrm{add}_{\Delta}\left(-S\right)+\mathbb{Z}\delta\right)\cup\mathbb{Z}_{\ne0}\delta,
\end{eqnarray*}
and the right hand side 
\begin{eqnarray*}
\mathrm{add}_{\Delta}\left(\Delta_{+}\left(\emptyset\right)\cup\left(-S\right)\cup\left\{ -\delta\right\} \right) & = & =\mathrm{add}_{\Delta}\left(\left(\Delta_{+}^{\circ}+\mathbb{Z}\delta\right)\cup\mathbb{Z}_{+}\delta\cup\left(-S\right)\cup\left\{ -\delta\right\} \right)\mathrm{add}_{\Delta}\left(\Delta_{+}^{\circ}\cup\left\{ \pm\delta\right\} \cup-S\right)
\end{eqnarray*}
and the coincidence is demonstrated. Now the theorem is a corollary
of \cite[Theorems 2.5 and 2.6]{Futorny1997}.
\end{proof}
From the description it is also clear, that in the case $S=\Pi^{\circ}$
the parabolic systems $P_{S}\left(\Pi^{\circ}\right)$ and $P\left(\Pi^{\circ},S\right)$
coincide.

\subsection{Parabolic subalgebras}

Let $\left(\cdot,\cdot\right)$ be the standard form on $\mathfrak{g}$.
The \emph{(standard) Hermitian form} on $\mathfrak{g}$ is given by
$\left(x,y\right)_{0}=\left(\sigma_{0}\left(x\right),y\right)$. A
unitary involution $\hat{\sigma}$ is given by the negative Chavelley
involution and defined as $\hat{\sigma}\left(x_{\alpha}\right)=x_{-\alpha}$
for $\alpha\in\Delta,\,x_{\alpha}\in\mathfrak{g}$ and $\hat{\sigma}\left(h\right)=h$
for $h\in\mathfrak{h}$. If $\sigma:\Delta\to\Delta$ is the linear
involutive automorphism defined by $\sigma\left(\beta\right)=-\beta$
for any $\beta\in\Delta$, then both $\sigma_{0}$ and $\hat{\sigma}$
are functorial extensions to $\mathfrak{g}$.
\begin{defn}
\label{def:Borel-type subalgebra}A subalgebra $\mathfrak{b}$ is
called \emph{Borel-type subalgebra} if $\hat{\sigma}\left(\mathfrak{b}\right)\cap\mathfrak{b}=\mathfrak{h}$.
It is called \emph{Borel subalgebra} if there is a positive system
$\Delta_{+}$ such that $\mathfrak{b}=\mathfrak{h}\oplus\sum_{\beta\in\Delta_{+}}\mathfrak{g}_{\beta}$. 
\end{defn}

In contrast to Borel subalgebras of finite dimensional Lie algebras,
the above definition admits Borel-type subalgebras that do not correspond
to positive systems of $\Delta$ (cf. \cite{BekkertBenkartFutornyKashuba}). 
\begin{defn}
\label{def:parabolic subalgebra}The subalgebra $\mathfrak{p}\left(P\right)=\mathfrak{h}\oplus\sum_{\beta\in P}\mathfrak{g}_{\beta}$
for the parabolic system $P\subset\Delta$, is called \emph{parabolic
subalgebra}. Additionally, $\mathfrak{p}$ is called \emph{maximally
parabolic,} if it is maximal as a proper subalgebra. 
\end{defn}

Introduce the \emph{derived algebra} $\mathfrak{g}'=\left[\mathfrak{g},\mathfrak{g}\right]$
(cf. \cite[p. 335]{Carter2005}) and the \emph{derived algebra} \emph{related
to the parabolic system} $P$, given by $\mathfrak{g}'_{P}=\left[\mathfrak{g}_{P\cap\sigma\left(P\right)},\mathfrak{g}_{P\cap\sigma\left(P\right)}\right]$.
Let $B$ be a basis for $P\cap\sigma\left(P\right)$, then in root
space decomposition this writes as
\[
\mathfrak{g}'_{P}=\mathrm{span}_{\mathbb{C}}\check{B}\otimes\mathbb{C}\left[t,t^{-1}\right]\oplus\sum_{\beta\in P\cap\sigma\left(P\right)\cap\Delta_{\mathrm{i}}}\mathfrak{g}_{\beta}\,.
\]
If $\delta\in P\cap\sigma\left(P\right)$, then $d=\check{\delta}\in\mathrm{span}_{\Bbbk}\check{B}$.
The \emph{Heisenberg subalgebra} of $\mathfrak{g}$ is the sum of
the isotropic root spaces and the center, 
\[
\mathcal{H}=\sum_{\alpha\in\mathbb{Z}_{\ne0}\delta}\mathfrak{g}_{\alpha}\oplus\mathbb{C}c.
\]
If $\mathfrak{g}$ is realized as affinization of a split simple Lie
algebra $\left(\mathfrak{g}^{\circ},\mathfrak{h}^{\circ},\kappa^{\circ}\right)$
(cf. Section 1.2), then $\mathcal{H}=\mathcal{L}\left(\mathfrak{h}^{\circ}\right)\oplus_{\omega_{D}}\Bbbk$.
With this we can introduce a subalgebra of the Heisenberg subalgebra
$\mathcal{H}_{\Pi\ominus B}\subset\mathcal{H}$ \index{mathcal{H}_{Piominus B}, Heisenberg algebra, subalgebra of the@$\mathcal{H}_{\Pi\ominus B}$, Heisenberg algebra, subalgebra of the}
defined by the relations $\left[\mathfrak{g}'_{P},\mathcal{H}_{\Pi\ominus B}\right]=0$
and $\mathcal{H}_{\Pi\ominus B}+\left(\mathfrak{g}'_{P}\cap\mathcal{H}\right)=\mathcal{H}$.
These relations are satisfied by 
\[
\mathcal{H}_{\Pi\ominus B}=\mathrm{span}_{\Bbbk}\left(\Psi\right)^{\vee}\otimes\Bbbk\left[t,t^{-1}\right]\oplus\Bbbk c\text{ for }\Psi=\left\{ \psi\in\Pi\mid\left(\psi,B\right)=\left\{ 0\right\} \right\} .
\]
\begin{thm}
\cite[Th. 3.3]{Futorny1997} \label{thm:parabolic induction}Let $\mathfrak{g}$
be the affine Lie algebra and $\Delta$ its root system and $P$ be
an arbitrary parabolic system in $\Delta$. \\
\begin{minipage}[t]{0.9\columnwidth}%
\begin{itemize}
\item[(i) ]

The subalgebra $\mathfrak{p}\left(P\right)$ of $\mathfrak{g}$ has
a decomposition 
\[
\mathfrak{p}\left(P\right)=\mathcal{L}\oplus\mathcal{N}\mbox{, where }\mathcal{N}=\mathfrak{g}_{P\smallsetminus\sigma\left(P\right)}
\]
and $\mathcal{L}$ is one of the following types,\begin{itemize}
\item[ {\scriptsize \bf (I)} ]

a locally finite Lie algebra or

\item[ {\scriptsize \bf (II)} ] $\mathcal{L}=\left(\mathfrak{g}'_{P}+\mathcal{H}_{\Pi\ominus B}\right)+\mathfrak{h}$.

\end{itemize} 
%\vspace{0pt}

\item[(ii) ]

$\mathfrak{g}$ has a split triangular decomposition associated to
$P$, 
\[
\mathfrak{g}=\mathcal{N}_{-}\oplus\mathcal{L}\oplus\mathcal{N}\mbox{, where }\mathcal{N}_{-}=\mathfrak{g}_{\sigma\left(P\right)\smallsetminus P}\,.
\]

\end{itemize} %
\end{minipage}
\end{thm}

\section{Weight modules }

\subsection{Induced Representations}

Consider a subset $S\subset\Pi^{\circ}$ and define the subalgebra
\[
\mathfrak{g}\left(S\right)=\mathrm{cl}_{\mathfrak{g}}\left\langle \mathfrak{g}_{\varphi},\mathfrak{g}_{-\varphi}\mid\varphi\in S\right\rangle \subset\mathfrak{g}^{\circ}
\]
\begin{defn}
\label{def: Levi and nilpotent}The \emph{Levi subalgebra} associated
with $S\subset\Pi^{\circ}$ is the finite-dimensional reductive Lie
algebra $\mathcal{L}_{S}=\mathfrak{h}+\mathfrak{g}\left(S\right)$.
Denote for $S\subset\Pi^{\circ}$ 
\[
\mathcal{N}_{S^{c}}^{\pm}=\sum_{\varphi\in\Delta_{+}\smallsetminus\mathsf{add}_{\Delta}S}\mathfrak{g}_{\pm\varphi},\quad\mathcal{P}_{S}^{\pm}=\mathcal{L}_{S}\oplus\mathcal{N}_{S^{c}}^{\pm}\;.
\]
\end{defn}

Start with a semisimple Levi component $\mathcal{L}_{S}=\mathfrak{h}+\mathfrak{g}\left(S\right)$
for a subset $S\subset\Pi^{\circ}$. Because $\mathcal{L}_{S}$ is
semisimple, it admits a triangular decomposition $\mathcal{L}_{S}=\mathcal{L}_{-}\oplus\mathcal{L}_{0}\oplus\mathcal{L}_{+}$.
Its universal enveloping algebra decomposes according to PBW Theorem,
$\mathcal{U}\left(\mathcal{L}_{S}\right)=\mathcal{U}\left(\mathcal{L}_{-}\right)\mathcal{U}\left(\mathcal{L}_{0}\right)\mathcal{U}\left(\mathcal{L}_{+}\right)$,
with a non-trivial center $Z_{S}\subset\mathcal{U}\left(\mathcal{L}\right)_{0}$.
Define the polynomial ring $T_{S}=\mathrm{Sym}\left(\mathfrak{h}\right)\otimes Z_{S}$.
Here $\mathrm{Sym}\left(\mathfrak{h}\right)\cong\mathcal{U}\left(\mathfrak{h}\right)\cong\mathbb{C}\left[\check{\Pi},c\right]$.
The generalized Harish-Chandra homomorphism associated to $S\subset\Pi^{\circ}$
is the projection $\phi_{\mathrm{HC}}^{\left(S\right)}:\mathcal{U}\left(\mathfrak{g}\right)\to T_{S}$
with respect to the decomposition 
\[
\begin{aligned}\mathcal{U}\left(\mathfrak{g}\right) & =\left(\mathcal{N}_{S^{c}}^{-}\mathcal{U}\left(\mathfrak{g}\right)+\mathcal{U}\left(\mathfrak{g}\right)\mathcal{N}_{S^{c}}^{+}\right)\oplus\left(\mathcal{L}_{-}\mathcal{U}\left(\mathcal{L}\right)+\mathcal{U}\left(\mathcal{L}\right)\mathcal{L}_{+}\right)\oplus T_{S}\end{aligned}
\,.
\]

For a parabolic system $P$ with subalgebra $\mathfrak{p}\left(P\right)$
that is of type {\scriptsize \bf (I)} according to Theorem \ref{thm:parabolic induction}
($\mathcal{L}$ locally finite), there is a Levi decomposition $\mathfrak{p}\left(P\right)=\mathcal{L}\oplus\mathcal{N}$
that meets the condition for $\mathcal{L}$ to be a Levi subalgebra.
Since the set $P\cap-P$ is additively closed in $\Delta$, we can
designate a basis $S$ for the positive root monoid inside $P\cap-P$.
Then $\mathcal{L}$ equals the Levi subalgebra $\mathcal{L}_{S}=\mathfrak{h}+\mathfrak{g}\left(S\right)$.
We can construct an irreducible $\mathcal{L}_{S}$-module as follows:

Fix $\lambda\in\mathfrak{h}\left(S\right)^{*}$ and $\gamma\in\mathbb{C}$.
Let $\mathbb{C}_{\lambda,\gamma}$ be a 1-dimensional $T_{S}$-module
with the action $\left(h\otimes z\right)v_{\lambda,\gamma}=\lambda\left(h\right)\gamma\left(z\right)v_{\lambda,\gamma}$
for $h\in\mathfrak{h}\left(S\right)$ and $z\in\mathbb{C}\left[\mathbf{z}_{S}\right]$,
and $\mathcal{L}_{+}v_{\lambda,\gamma}=0$. Then 
\[
V_{S}\left(\lambda,\gamma\right)=\mathcal{U}\left(\mathcal{L}_{S}\right)\underset{T_{S}}{\otimes}\mathbb{C}_{\lambda,\gamma}
\]
is a $\mathcal{L}_{S}$-module, which has a unique irreducible quotient.
Because $\mathfrak{g}\left(S\right)$ is semisimple, the irreducible
weight $\mathfrak{h}+\mathfrak{g}\left(S\right)$-modules with finite-dimensional
weight spaces are classified by S. Fernando and O. Mathieu (cf. \cite{Mathieu2000})
as being isomorphic to certain parabolically induced modules. Let
$\mathfrak{p}$ be a parabolic subalgebra of $\mathfrak{g}\left(S\right)$
and $W$ be a cuspidal $\mathfrak{p}$-module (see introduction or
Section 3.7 for definition), then the induced module 
\[
M_{\mathfrak{p}}\left(W\right)=\mathrm{Ind}_{\mathfrak{p}}^{\mathfrak{g}}W=\mathcal{U}\left(\mathfrak{g}\left(S\right)\right)\underset{\mathcal{U}\left(\mathfrak{p}\right)}{\otimes}W
\]
has a unique irreducible quotient $L_{\mathfrak{p}}\left(W\right)$.
\begin{thm}
\label{Prop: 3.6. Classifiaction of irrep L_S}\cite{Fernando1990}
If $V$ is an irreducible weight $\mathfrak{h}+\mathfrak{g}\left(S\right)$-module
with only finite-dimensional weight components, then $V$ is isomorphic
to $W$ or to $L_{\mathfrak{p}}\left(W\right)$ for some parabolic
subalgebra $\mathfrak{p}$ and some cuspidal $\mathfrak{p}$-module
$W$. 
\end{thm}

\subsection{Parabolic induction}
\begin{defn}
If $S_{k}$, $\left(k=1,\dots,N\right)$ are the connected components
of the partition $S\subset\Pi^{\circ}$ with $S=\bigcup_{k=1}^{N}S_{k}$,
then write 
\[
\widehat{\mathfrak{g}}\left(S_{k}\right)=\left(\mathfrak{g}\left(S_{k}\right)\otimes\mathbb{C}\left[t,t^{-1}\right]\oplus\mathbb{C}c\right)\rtimes\mathbb{C}d
\]
for the corresponding affine Lie algebra and denote the Lie subalgebra
$\widehat{\mathfrak{g}}\left(S\right)=\sum_{k=1}^{N}\widehat{\mathfrak{g}}\left(S_{k}\right)$. 
\end{defn}

If $P\subset\Delta$ is a parabolic system, then $P\cap\sigma\left(P\right)$
is additively closed in $\Delta$. It is thus possible to chose $S$
such that $P\cap\sigma\left(P\right)=\mathrm{span}_{\mathbb{Z}}S$
and meaningful to set $\mathcal{H}_{P}=\mathcal{H}\left(\check{S^{c}}\right)=\mathcal{H}\left(\check{\alpha}\mid\alpha\in S^{c}\right)\subset\mathcal{H}$.

Let $\mathfrak{p}=\mathfrak{p}\left(P\right)=\mathcal{L}\oplus\mathcal{N}$
be a parabolic subalgebra corresponding to a parabolic system $P$.
Given an irreducible weight $\mathcal{L}$-module $V$, we extend
the action to $\mathcal{N}$ trivially, obtaining an irreducible $\mathfrak{p}$-module
$V$. Construct the $\mathfrak{g}$-module 
\[
M_{\mathfrak{p}}=\mathrm{ind}_{\mathfrak{p}}^{\mathfrak{g}}\left(V\right)=\mathcal{U}\left(\mathfrak{g}\right)\underset{\mathcal{U}\left(\mathfrak{p}\right)}{\otimes}V.
\]
Depending on $\mathfrak{p}$, it will be called \emph{Generalized
Verma Type module,} if $\mathfrak{p}$ is of type {\scriptsize \bf (I)},
and \emph{generalized loop module,} if $\mathfrak{p}$ is of type
{\scriptsize \bf (II)}. As special cases, $M_{\mathfrak{p}}$ is
a classical Generalized Verma Module if $\mathfrak{p}$ is a standard
parabolic subalgebra and $V$ a finite-dimensional $\mathfrak{p}$-module.
If $\mathfrak{p}$ is of type {\scriptsize \bf (I)}, denomination
is further refined.

\subsection{Standard Generalized Verma Modules}

Standard Generalized Verma Modules (GVM) are induced from an irreducible
module for a Levi subalgebra $\mathcal{L}_{S}=\mathfrak{h}+\mathfrak{g}\left(S\right)$.
First the general case:

If $S\subset\Pi^{\circ}$, then $\mathcal{P}_{S}^{\pm}=\mathcal{L}_{S}\oplus\mathcal{N}_{S^{c}}^{\pm}$
and, for an irreducible $\mathcal{L}_{S}$-module $V_{S}$, we set
$\mathcal{N}_{S^{c}}^{\pm}V_{S}=0$ and
\begin{equation}
M_{S}^{\pm}\left(\lambda,\gamma\right)=\mathrm{ind}_{\mathcal{P}_{S}^{\pm}}^{\mathfrak{g}}\left(V_{S}\right)=\mathcal{U}\left(\mathfrak{g}\right)\underset{\mathcal{U}\left(\mathcal{P}_{S}^{\pm}\right)}{\otimes}V_{S}.\label{eq:M_S and L_S}
\end{equation}
It has a unique irreducible quotient $L_{S}^{\pm}\left(\lambda,\gamma\right)$.
Notice that $V_{S}$ does not have to be finite-dimensional. 
\begin{prop}
\label{prop:G-mod induced GVM}\cite[Prop. 3.6. (iii)]{Futorny1997}
Let $\tilde{V}$ be an irreducible weight $\mathfrak{g}$-module and
$0\ne v\in\tilde{V}_{\lambda}$ such that\textup{ $\mathcal{N}_{S^{c}}^{\pm}v=0$,
then }$\tilde{V}\cong L_{S}^{\pm}\left(\lambda,V_{S}\right)$ for
an irreducible $\mathcal{L}_{S}$-module $V_{S}$.
\end{prop}

\subsection{Irreducible representations of the Heisenberg subalgebra}
\begin{defn}
\label{def: Heisenberg-subalgebra}The \emph{Heisenberg subalgebra}
$\mathcal{H}\left(\check{\alpha}_{1},\dots,\check{\alpha}_{k}\right)$,
$\left(k\le n\right)$, is the image of \linebreak{}
$\sum_{i=1}^{k}\check{\alpha}_{i}\otimes\mathbb{C}\left[t,t^{-1}\right]\oplus\mathbb{C}c$
under the quotient map $\mathfrak{g}\rightarrow\mathfrak{g}/\left(\mathfrak{h}\otimes1\right)$.
\end{defn}

The Heisenberg subalgebras admit the maximal abelian subalgebras $\mathcal{H}_{\pm}\left(\check{\alpha}_{1},\dots,\check{\alpha}_{k}\right)\oplus\mathbb{C}c$
with 
\[
\mathcal{H}_{\pm}\left(\check{\alpha}_{1},\dots,\check{\alpha}_{k}\right)=\sum_{i=1}^{k}\check{\alpha}_{i}\otimes t^{\pm1}\mathbb{C}\left[t^{\pm1}\right].
\]
Consider the full Heisenberg subalgebra and its maximal abelian subalgebra
\[
\mathcal{H}=\mathcal{H}\left(\check{\alpha}_{1},\dots,\check{\alpha}_{n}\right)=\sum_{m\ne0}\mathfrak{g}_{m\delta}\oplus\mathbb{C}c\qquad\mathcal{H}_{+}=\mathcal{H}_{+}\left(\check{\alpha}_{1},\dots,\check{\alpha}_{n}\right)=\sum_{m>0}\mathfrak{g}_{m\delta}.
\]
Let $a\in\mathbb{C}^{*}$ and $\mathbb{C}v_{a}$ be the the 1-dimensional
$\mathcal{H}_{+}\oplus\mathbb{C}c$-module for which $\mathcal{H}_{+}v_{a}=0$,
$cv_{a}=av_{a}$. Consider the $\mathcal{H}$-module
\begin{equation}
M^{+}\left(\lambda\right)=\mathcal{U}\left(\mathcal{H}\right)\underset{\mathcal{U}\left(\mathcal{H}_{+}\oplus\mathbb{C}c\right)}{\otimes}\mathbb{C}v_{a}\,.\label{eq:loop_module non-zero}
\end{equation}
It carries a natural $\mathbb{Z}$-grading with the $i$-th component
$\sigma\left(\mathcal{U}\left(\mathcal{H}_{+}\right)_{-i}\right)v_{a}$.
Define $M^{+}\left(a\right)$ for $\mathcal{H}_{-}$ analogously.

Define another family of modules, so-called loop modules as in \cite{CharyPressley1986}.
Let\linebreak{}
 $p:\mathcal{U}\left(\mathcal{H}\right)\to\mathcal{U}\left(\mathcal{H}\right)/\mathcal{U}\left(\mathcal{H}\right)c$
be the canonical projection. For $r>0$, consider the $\mathbb{Z}$-graded
ring $L_{r}=\mathbb{C}\left[t^{-r},t^{r}\right]$. Denote by $P_{r}$
the set of graded ring epimorphisms $\Lambda:\mathcal{U}\left(\mathcal{H}\right)/\mathcal{U}\left(\mathcal{H}\right)c\to L_{r}$
with $\Lambda\left(1\right)=1$. Define a $\mathcal{H}$-module structure
on $L_{r}$ by the following action of any $e_{k\delta}^{\left(\alpha\right)}=\check{\alpha}\otimes t^{k}\in\mathfrak{g}_{k\delta}$,
$\alpha\in\Pi^{\circ}$:
\[
e_{k\delta}^{\left(\alpha\right)}t^{sr}=\left(\Lambda\circ p\right)\left(e_{k\delta}^{\left(\alpha\right)}\right)t^{sr}=t^{\left(k+s\right)r},\:\Lambda\in P_{r},\:k\in\mathbb{Z}\smallsetminus\left\{ 0\right\} ,\:ct^{rs}=0,\,s\in\mathbb{Z}.
\]
Denote this $\mathcal{H}$-module by $L_{r,\Lambda}$. Define $\Lambda_{0}$
the trivial homomorphism onto $\mathbb{C}$ with $\Lambda_{0}\left(1\right)=1$,
then $L_{0,\Lambda_{0}}$ is the trivial module. 
\begin{prop}
(i) \textup{\cite{Futorny1996}} Every irreducible $\mathbb{Z}$-graded
$\mathcal{H}$-module $V$ with central charge $\lambda\left(c\right)=a\in\mathbb{C}^{*}$
with at least one finite-dimensional weight component $V_{\varphi}$
is isomorphic to \textup{$M^{\pm}\left(\lambda\right)$ up to a shifting
of gradation.}\\
(ii) \textup{\cite{Chary1986}} Every irreducible $\mathbb{Z}$-graded
$\mathcal{H}$-module with central charge zero is isomorphic to $L_{r,\Lambda}$
for some $r\ge0$, $\Lambda\in P_{r}$ \textup{up to a shifting of
gradation.}
\end{prop}

\subsection{Generalized loop modules}

Following \cite{Chary1986}, we define the category $\tilde{\mathcal{O}}\left(\Pi^{\circ}\right)$
to be the category of weight $\mathfrak{g}$-modules $V$ satisfying
the condition that there exist finitely many elements $\lambda_{1},\dots,\lambda_{r}\in\mathfrak{h}^{*}$
such that 
\[
\mathrm{supp}\left(V\right)\subset\bigcup_{i=1}^{r}R\left(\lambda_{i}\right)\mbox{ where }R\left(\lambda_{i}\right)=\lambda_{i}-\mathrm{span}_{\mathbb{N}_{0}}\Pi^{\circ}+\mathbb{Z}\delta.
\]

In $\Delta^{\circ}$, the set $\Phi=-\Delta_{+}^{\circ}\cup\mathrm{add}_{\Delta^{\circ}}\left(S\right)$
is a parabolic system for every $S\subset\Pi^{\circ}$. Thus, we
can define $\mathfrak{n}_{S^{c}}=\mathfrak{g}_{P\smallsetminus\sigma P}$
if $S$ is a basis for $P\smallsetminus\sigma P$ and $P$ is given
by 
\[
P=\Delta^{\circ}\smallsetminus\mathrm{add}_{\Delta^{\circ}}\left(S\right)+\mathbb{Z}\delta\,.
\]
Thus
\begin{equation}
\mathfrak{n}_{S^{c}}=\sum_{\varphi\in P\smallsetminus\sigma P}\mathfrak{g}_{\varphi}\,.\label{eq:natural_nilpotent_subalgebra}
\end{equation}

Back to the case where $S=\Pi^{\circ}$: Denote $\mathfrak{n}_{+}=\mathfrak{n}_{\Pi^{\circ}}$
and $\mathfrak{n}_{-}=\mathfrak{n}_{-\Pi^{\circ}}$. Then $\mathfrak{g}=\mathfrak{n}_{-}\oplus\left(\mathfrak{h}+\mathcal{H}\right)\oplus\mathfrak{n}_{+}$
is a triangular decomposition with Borel subalgebra $\mathfrak{b}=\left(\mathfrak{h}+\mathcal{H}\right)\oplus\mathfrak{n}_{+}$.
This $\mathfrak{b}$ is called the \index{natural Borel subalgebra}\emph{natural
Borel subalgebra}. Let $V$ be an irreducible $\mathbb{Z}$-graded
$\mathcal{H}$-module with central charge $a\in\mathbb{C}$ and $\lambda\in\mathfrak{h}^{*}$
with $\lambda\left(c\right)=a$. Define a $\mathfrak{b}$-module structure
on $V$ by the action $hv_{k}=\left(\lambda+k\delta\right)\left(h\right)v_{k}$,
$\mathfrak{n}_{+}v_{k}=0$ for all $h\in\mathfrak{h}^{\circ}$, $v_{k}\in V_{k}$,
$i\in\mathbb{Z}$. From this $\mathcal{H}$-module $V\left(\lambda\right)$
we can obtain a $\mathfrak{g}$-module by induction, 
\[
M_{\mathfrak{b}}\left(V\left(\lambda\right)\right)=\mathrm{ind}_{\mathfrak{b}}^{\mathfrak{g}}\left(V\left(\lambda\right)\right)=\mathcal{U}\left(\mathfrak{g}\right)\underset{\mathcal{U}\left(\mathfrak{b}\right)}{\otimes}V\left(\lambda\right).
\]
This module is called \emph{\index{imaginary Verma module}imaginary
Verma module}. By definition, $M_{\mathfrak{b}}\left(V\left(\lambda\right)\right)$
is $Q$-graded. Denote by $\left[\lambda\right]$ the image of $\lambda$
under the quotient map $\mathfrak{h}^{*}\rightarrow\mathfrak{h}^{*}/\mathbb{Z}\delta$.
This image admits the classical Bruhat order and the relation $\mathfrak{n}_{+}v_{k}=0$
entails 
\[
M_{\mathfrak{b}}\left(V\left(\lambda\right)\right)=\bigoplus_{\left\{ \mu\mid\left[\mu\right]\le\left[\lambda\right]\right\} }V_{\mu}\:,
\]
thus $\mathrm{supp}\left(V\right)\subset\lambda-\mathrm{span}_{\mathbb{N}_{0}}\Pi^{\circ}+\mathbb{Z}\delta.$

The following facts hold for the above and an irreducible object $\tilde{V}$
in $\tilde{\mathcal{O}}\left(\Pi^{\circ}\right)$: 
\begin{prop}
\label{pro: loop modules}(i) \textup{$M_{\mathfrak{b}}\left(V\left(\lambda\right)\right)$}
is $\mathcal{U}\left(\mathfrak{n}_{-}\right)$-free.\\
(ii) $M_{\mathfrak{b}}\left(V\left(\lambda\right)\right)$ has a unique
irreducible quotient $L_{\mathfrak{b}}\left(V\left(\lambda\right)\right)$.\\
(iii) \label{prop:G-mod induced-1} \cite{Futorny1996} There exist
$\lambda\in\mathfrak{h}^{*}$ and an irreducible $\mathcal{H}$-module
$V$ such that $\tilde{V}$ is isomorphic to the unique irreducible
quotient of $\mathrm{ind}_{\mathfrak{b}}^{\mathfrak{g}}\left(V\right)$.\\
(iv) \label{thm:classificationChari-modules-1}\cite{CharyPressley1986}
If $\tilde{V}$ has central charge zero, then $\tilde{V}\cong L_{\mathfrak{b}}\left(L_{r,\Lambda}\right)$
for some $\lambda\in\mathfrak{h}^{*}$, $\lambda\left(c\right)=0$,
$\Lambda\in P_{r}$.\\
(v) \cite{CharyPressley1986} If $\tilde{V}$ has central charge $a\in\mathbb{C}^{*}$
and $\mathrm{dim}\,\tilde{V}_{\mu}<\infty$ for at least one $\mu\in\mathrm{supp}\left(\tilde{V}\right)$,
then $\tilde{V}\cong L_{\mathfrak{b}}\left(V\left(\lambda\right)\right)$
for some $\lambda\in\mathfrak{h}^{*}$, $\lambda\left(c\right)=a$.\\
(vi) \cite{CharyPressley1986} If $\tilde{V}$ is integrable then
$\tilde{V}$ has central charge zero.
\end{prop}

\subsection{Non-standard or mixed modules}

The modules presented in this section are called \emph{mixed type
modules,} because they are parabolically induced from irreducible
modules for a Levi subalgebra that is a sum of a subalgebra of the
Heisenberg subalgebra and (possibly several) affine Lie subalgebras.
Recall $\mathfrak{n}_{X^{c}}=\mathfrak{g}_{\Delta_{+}^{\circ}\smallsetminus\mathrm{add}_{\Delta^{\circ}}\left(X\right)+\mathbb{Z}\delta}$
for $X\subset\Pi^{\circ}$.

Let $V$ be a $\mathbb{Z}$-graded $\mathcal{H}\left(\check{S^{c}}\right)$-module
with central charge $a\in\mathbb{C}$ and $\lambda\in\mathfrak{h}^{*}$
with $\lambda\left(c\right)=a$. Define a $\mathfrak{b}_{S}$-module
structure on $V$ by the action $hv_{i}=\left(\lambda+i\delta\right)\left(h\right)v_{i}$,
$\mathfrak{n}_{S^{c}}v_{i}=0$ for all $h\in\mathfrak{h}$, $v_{i}\in V_{i}$,
$i\in\mathbb{Z}$. 

Now we are able to state the central theorem for parabolic induction,
which points out the fact that the parabolic induction functor ``produces''
irreducible representations for every parabolic subalgebra $\mathfrak{p}$
of $\mathfrak{g}$ and for every irreducible module of the Levi component
of $\mathfrak{p}$. Therefore we have to see, that inducing from irreducible
module of a subalgebra of type {\scriptsize \bf (II)} (non-standard),
\[
\mathfrak{p}=\bigoplus_{i}\widehat{\mathfrak{g}}\left(S_{i}\right)\oplus\mathcal{H}\left(\check{S^{c}}\right)\oplus\mathfrak{n}_{S^{c}}+\mathfrak{h}
\]
with disconnected components $S_{i}$ with $\bigcup_{i}S_{i}=S$,
is well behaved.
\begin{thm}
\label{prop:O(S)-induction}Let $\mathfrak{p}=\bigoplus_{i}\widehat{\mathfrak{g}}\left(S_{i}\right)\oplus\mathcal{H}\left(\check{S^{c}}\right)\oplus\mathfrak{n}_{S^{c}}+\mathfrak{h}$.\textbf{
}If $V$ is an irreducible weight $\mathfrak{p}$-module such that
$\mathfrak{n}_{S^{c}}$ acts zero, then the induction 
\[
M_{\mathfrak{p}}=\mathrm{ind}_{\mathfrak{p}}^{\mathfrak{g}}\left(V\right)=\mathcal{U}\left(\mathfrak{g}\right)\underset{\mathcal{U}\left(\mathfrak{p}\right)}{\otimes}V
\]
 admits a unique irreducible quotient, $L_{\mathfrak{p}}$.
\end{thm}

The proof is standard.

\subsection{Classification problem for irreducible non-dense weight $\mathfrak{g}$-modules}

We have shown already, that when inducing from the different parabolics,
this yields different families of induced modules that have the same
nice property of admitting unique irreducible quotients. The natural
question then would be to ask in how far induction exhausts all irreducible
modules. The main conjecture states that every non-dense module is
induced. Our main theorem is
\begin{thm}
\label{thm:Main}Let $\mathfrak{g}$ be $A_{1}^{\left(1\right)}$,
$A_{2}^{\left(2\right)}$, $A_{2}^{\left(1\right)}$, $A_{3}^{\left(1\right)}$
or $A_{4}^{\left(1\right)}$ and $V$ be an irreducible non-dense
$\mathfrak{g}$-module, then there exists a vector $v\in V$ that
is primitive with respect to the nilpotent part $\mathcal{N}$ of
a parabolic subalgebra $\mathfrak{p}\left(P\right)=\mathcal{L}\oplus\mathcal{N}$.
\end{thm}

If the affine root system has rank 2 and the module has only finite-dimensional
weight subspaces, we are able to give a precise classification statement,
because the non-trivial Levi subalgebras can only take a shape of
a simple Lie algebra or a Heisenberg algebra (both of rank one). Whereas
in general the Levi subalgebra itself could be a sum of affine Lie
algebras, whose cuspidal modules still are not classified. Also nothing
is known about the dimension of their weight spaces \textendash{}
although the latter we believe to be only infinite-dimensional. 
\begin{thm}
\cite{Futorny1996,Bunke2009} If $\mathfrak{g}$ has rank 2, i.e.
$\mathfrak{g}=A_{1}^{\left(1\right)}$ or $\mathfrak{g}=A_{2}^{\left(2\right)}$,
and $V$ is an irreducible non-dense $\mathfrak{g}$-module with at
least one finite-dimensional weight subspace, then $V$ is equivalent
to one module out of the following pairwise non-equivalent classes:
\[
L_{\alpha}^{\pm}\left(\lambda,\gamma\right),\:\text{if }\lambda\left(c\right)=0,\;L_{\mathfrak{b}}\left(L_{r,\Lambda}\right),\:\text{if }\lambda\left(c\right)=a,\;L_{\mathfrak{b}}\left(M^{\pm}\left(\lambda\right)\right)
\]
 for $\alpha\in\Delta^{\mathrm{re}},\,\lambda\in\mathfrak{h}^{*},\,\gamma,a\in\mathbb{C},\,r\in\mathbb{N}_{0}$,
$\Lambda\in P_{r}$, where $\mathfrak{b}=\mathfrak{h}\oplus\bigoplus_{\varphi\in\left(\mathbb{Z}\alpha+\mathbb{Z}\delta\right)\cap\Delta}\mathfrak{g}_{\alpha}$
and $M^{\pm}\left(\lambda\right)$ the irreducible $\mathbb{Z}$-graded
$\mathcal{H}\left(\check{\alpha}\right)$-module defined in \ref{eq:loop_module non-zero}.

If moreover $V$ has only finite-dimensional weight subspace and the
central charge $\lambda\left(c\right)=a\ne0$, then $V\cong L_{\alpha}^{\pm}\left(\lambda,\gamma\right)$
for some $\alpha\in\Delta^{\mathrm{re}}$, $\lambda\in\mathfrak{h}^{*}$,
$\gamma\in\mathbb{C}$. 
\end{thm}

Beyond the conjectures and Theorem \ref{thm:Main}, the following
classification problems for irreducible weight $\mathfrak{g}$-modules
are open problems: \\
\noindent\begin{minipage}[t]{1\columnwidth}%
\begin{itemize} \item[(i)] Modules of type $L_{S}\left(\lambda,V\right)$
where $V$ is a graded irreducible $\mathcal{H}$-module of non-zero
central charge with all infinite-dimensional components. Here, recent
progress has been made in \cite{FutornyKashuba2009} and \cite{BekkertBenkartFutornyKashuba}.

\item[(ii)] Dense $\mathfrak{g}$-modules of central charge zero. 

\item[(iii)] Dense $\mathfrak{g}$-modules of non-zero central charge
with an infinite-dimensional weight subspace. 

\end{itemize}%
\end{minipage}

\section{Quasicone Arithmetics}

\subsection{Pre-prosolvable subalgebras}

Recall that a Lie algebra $\mathfrak{s}$ is called solvable if the
derived series yields $\left\{ 0\right\} $ after finitely many steps.
Fix a positive integer $N$. The Lie algebra
\[
\mathcal{T}_{N}\mathfrak{g}=\mathfrak{g}\otimes\Bbbk\left[t\right]/t^{N+1}\Bbbk\left[t\right]
\]
is called truncated current Lie algebra \cite{Wilson2011,Takiff1971}.
It inherits the triangular decomposition from $\mathfrak{g}$. 
\begin{defn}
\label{def:pre-prosolvable}A Lie subalgebra $\mathfrak{s}\subset\mathfrak{g}$
is called \\
\begin{minipage}[t]{0.9\columnwidth}%
\begin{itemize}
\item[ {\scriptsize \bf (P)} ]\emph{perfect,} if\emph{ }$\left[\mathfrak{s},\mathfrak{s}\right]=\mathfrak{s}$,

\item[ {\scriptsize \bf (HA)} ] 

\emph{hypoabelian, }if its perfect radical (or perfect core), i.e.
its largest perfect subalgebra, is trivial,

\item[ {\scriptsize \bf (PS)} ] 

pre-prosolvable\emph{,} if the completion of $\mathfrak{s}/\Bbbk c$
is isomorphic to the projective limit of an inverse system of solvable
Lie algebras.\emph{ }

\end{itemize} 
\vspace{0pt}%
\end{minipage}
\end{defn}

Note that the inverse limit $\underleftarrow{\lim}\mathfrak{g}\otimes\Bbbk\left[t\right]/t^{N+1}\Bbbk\left[t\right]=\mathfrak{g}\otimes\Bbbk\left[\left[t\right]\right]$
is the completion of the loop algebra $\mathfrak{g}\otimes\Bbbk\left[\left[t\right]\right]$
with respect to an appropriate topology, for instance the product
topology on $\mathfrak{g}^{N}$. 

Construct a pre-prosolvable subalgebra $\mathfrak{s}$ as follows:
Let $\mathfrak{g}^{\circ}$ be locally finite with Cartan algebra
$\mathfrak{h}^{\circ}$ and consider $\left\{ \mathcal{T}_{N}\mathfrak{g}^{\circ}\mid N\in\mathbb{N}_{0}\right\} $
as a directed family of solvable Lie algebras with the obvious epimorphisms
of Lie algebras $\pi_{n,m}:\mathcal{T}_{n}\mathfrak{g}^{\circ}\rightarrow\mathcal{T}_{m}\mathfrak{g}^{\circ}$,
$n\ge m$ and the limit $\mathfrak{r}=\mathfrak{g}^{\circ}\otimes\Bbbk\left[\left[t\right]\right]=\underleftarrow{\lim}\mathfrak{g}^{\circ}\otimes\Bbbk\left[t\right]/t^{N+1}\Bbbk\left[t\right]$.
Now, $\mathfrak{s}$ is the algebra of polynomials in $\mathfrak{r}$.

Now, if the quotient of a locally affine Lie algebra $\mathfrak{g}$
by its center admits a non-trivial Lie algebra homomorphism $\varphi:\mathfrak{g}/\Bbbk c\rightarrow\mathcal{T}_{N}\mathfrak{g}^{\circ}$,
the image of a pre-prosolvable subalgebra $\mathfrak{s}\subset\mathfrak{g}$
under the map $\varphi$ is solvable for every $N\in\mathbb{Z}_{+}$. 
\begin{prop}
The pre-prosolvable subalgebra $\mathfrak{s}$ is hypoabelian.
\end{prop}

\begin{proof}
Assume $\mathfrak{g}^{\circ}$ to be locally finite of arbitrary choice.
Now the assumption is that the derived series becomes constant $\varphi\left(\mathfrak{s}\right)^{\left(m\right)}=\left\{ 0\right\} $
for some $m\in\mathbb{Z}_{+}$. We want to show that the perfect radical
of $\mathfrak{s}$ is trivial.

Assume on the contrary, the perfect radical of $\mathfrak{s}$ is
non-trivial. Consequently the derived series becomes constant $\mathfrak{s}^{\left(i\right)}=\mathfrak{r}\ne\left\{ 0\right\} $
for all $i$ large enough. This is equivalent to saying $\mathfrak{s}^{\left(i\right)}$
lies in the kernel of $\varphi$ for all $i\ge m$. Since $\mathfrak{r}$
is perfect itself, it must contain a subalgebra isomorphic to $\mathfrak{sl}_{2}$.
If $\mathfrak{r}$ lies in the kernel of $\varphi$ for every homomorphism,
then $\mathfrak{g}^{\circ}$ must not contain a subalgebra isomorphic
to $\mathfrak{sl}_{2}$, which is a contradiction to $\mathfrak{g}^{\circ}$
being locally finite of arbitrary choice.
\end{proof}
For $x\in\mathcal{H}'=\mathfrak{h}^{\circ}\otimes\Bbbk\left[t\right]\cap\left\{ \left[\mathfrak{s}_{-\alpha},\mathfrak{s}_{\alpha}\right]\mid\alpha\in\Delta\right\} $,
$\left(\mathfrak{s}_{\alpha}=\mathfrak{g}_{\alpha}\cap\mathfrak{s}\right)$,
let $\nu\left(x\right)$ be the greatest number such that $t^{-\nu}x$
is a non-zero element in $\mathfrak{h}^{\circ}\otimes\Bbbk\left[t\right]\cap\mathfrak{s}$.
The number $g=\max_{x\in\mathcal{H}'}\nu\left(x\right)$ will be called
the \emph{gap} of $\mathcal{H}'$. Let further $p_{\alpha}$ be the
projection maps parametrized by $\alpha\in\Delta_{+}$ such that $p_{\alpha}\left(\mathfrak{g}\right)\cong\widehat{\mathcal{L}}\left(\mathfrak{sl}_{2}\left(\Bbbk\right)\right)$. 
\begin{defn}
(i) \label{def: quasicone loc affine}A \emph{conic subalgebra} is
a pre-prosolvable subalgebra that contains $\mathfrak{h}^{\circ}\otimes t\Bbbk\left[t\right]$.
\\
(ii) The conic subalgebra $\mathfrak{s}$ is a \emph{quasicone subalgebra}
if it has trivial intersection with $\mathfrak{h}$ and $p_{\alpha}\left(\mathfrak{s}/\mathcal{H}'\right)$
is not solvable for any $\alpha\in\Delta_{+}$. 
\end{defn}

\begin{rem}
\noindent (i) The subalgebra of the locally finite Lie algebra of
type $A$, given by $\mathfrak{s}=\mathrm{cl}_{\mathfrak{g}}\left\{ e_{\alpha_{1}},e_{\alpha_{2}},\dots\right\} $
is also hypoabelian, yet not pre-prosolvable according to our definition.
With $\mathfrak{g}^{\circ}=A_{\infty}$ and the homomorphism $\varphi$
sending $x\mapsto x\otimes1$, the image of $\varphi$ is still not
solvable.

\noindent (ii) Apart from conic subalgebras there are other pre-prosolvable
subalgebras that are not nilpotent. The $A_{2}^{\left(1\right)}$-subalgebra
$\mathrm{cl}_{\mathfrak{g}}\left\{ e_{\alpha},e_{\beta},e_{-\alpha},e_{-\beta}\right\} \otimes t^{2}\Bbbk\left[t\right]$
has a non-finite derived series with $\mathfrak{a}^{\left(\mathbb{N}\right)}=\left\{ 0\right\} $.
A more exhaustive study of affine Kac-Moody subalgebras can be found
in \cite{FeliksonRetakhTumarkin2008}.
\end{rem}

\subsection{\label{sec:tropical-matrix-alg}The tropical matrix algebra of quasicone
subalgebras for $A_{n}^{\left(1\right)}$}

From now on, let $\mathfrak{g}$ be $A_{n}^{\left(1\right)}$. We
use the notations $e_{k\delta}^{\left(\alpha\right)}=t^{k}\otimes h_{\alpha}$
for $k\in\mathbb{Z}$ and $\alpha\in\Delta$, and $e_{I\delta}^{\left(\alpha,\beta,\dots\right)}=\left\{ e_{k\delta}^{\left(\alpha\right)},e_{k\delta}^{\left(\beta\right)},\dots\mid k\in I\subseteq\mathbb{Z}\right\} $.
Let $\Pi^{\circ}=\left\{ \alpha_{1},\dots,\alpha_{n}\right\} $ be
the standard root basis for $\Delta_{+}^{\circ}$.

\noindent \begin{flushleft}
\textbf{Notation.} We denote a $\mathfrak{g}$-subspace
\begin{eqnarray}
 &  & \mathfrak{X}=\mathrm{span}_{\mathbb{C}}\left\{ \begin{array}{l|c}
e_{\alpha_{i}+\cdots+\alpha_{j}+\left(A_{i,j+1}+S\right)\delta} & e_{\Omega_{i,j}\delta}^{\left(\alpha_{i}+\cdots+\alpha_{j}\right)}\,,c\,,d\\
e_{-\left(\alpha_{i}+\cdots+\alpha_{j}\right)+\left(A_{j+1,i}+S\right)\delta} & \left(i\le j\in\left\{ 1,\dots,n\right\} \right)
\end{array}\right\} \label{eq:subspace}
\end{eqnarray}
with $A_{ij},\Omega_{i,j}\subseteq\mathbb{Z}$ and the sum of sets
being the Minkowski sum, i.e. $A+\Omega=\left\{ a+\omega\mid a\in A,\,\omega\in\Omega\right\} $,
by 
\[
\left\{ \begin{array}{ccccc|}
* & A_{0,1} & A_{0,2} & \cdots & A_{0,n}\\
A_{1,0} & * &  &  & \\
A_{2,0} &  & \ddots & \ddots & \vdots\\
\vdots &  & \ddots &  & A_{n-1,n}\\
A_{n,0} &  & \cdots & A_{n,n-1} & *
\end{array}\begin{array}{ccccc}
\mathbb{Z} & \Omega_{0,1} & \Omega_{0,2} & \cdots & \Omega_{0,n}\\
\Omega_{1,0} & \mathbb{Z}\\
\Omega_{2,0} &  & \ddots & \ddots & \vdots\\
\vdots &  & \ddots &  & \Omega_{n-1,n}\\
\Omega_{n,0} &  & \cdots & \Omega_{n,n-1} & \mathbb{Z}
\end{array}\right\} .
\]
\par\end{flushleft}

\noindent \begin{flushleft}
Let $\check{\alpha}_{i}$ be the standard basis of $\mathfrak{h}^{\circ}$.
Since $\mathbb{C}\left(\alpha_{i}+\cdots+\alpha_{j}\right)^{\vee}\otimes t^{l}$,
$i\le j\in\left\{ 1,\dots,n\right\} ,\,l\in\mathbb{N}_{0}$, are vector
spaces, the relations 
\begin{equation}
\Omega_{i,j}=\Omega_{j,i}\text{ and }\Omega_{i,j}\cap\Omega_{i,k}\subset\Omega_{i,k}\label{eq:omega for Heisenberg}
\end{equation}
 must hold for all $i,j,k\in\left\{ 0,\dots,n\right\} $. Thus the
omegas are determined by a selection of $n$ omegas, such that the
other omegas can be generated by these relations, e.g. $\left\{ \Omega_{i,i}\right\} _{i=1,\dots,n}$.
\par\end{flushleft}

\noindent \begin{flushleft}
For $\mathfrak{X}$ to be a subalgebra, the sets $A_{i,j}$ and $\Omega_{i,j}$,
$\left(i\ne j\in\left\{ 0,\dots,n\right\} \right)$, have to satisfy
the relations 
\begin{align}
A_{i,j}+A_{j,i}\subset\Omega_{i,j} & \ \left(i,j\in\left\{ 0,\dots,n\right\} \right)\nonumber \\
A_{i,j}+A_{j,k}\subset A_{i,k} & \ \left(i\ne k\ne j\ne i\in\left\{ 0,\dots,n\right\} \right)\label{eq:relations subalgebra exponent sets}\\
A_{i,j}\supset A_{i,j}+\Omega_{k,l}\text{ and }A_{j,i}\supset A_{j,i}+\Omega_{k,l} & \ \left(i\ne j\right)\text{ for }i-1\le k<l\le j+1\nonumber 
\end{align}
It follows that $\Omega_{i,j}+\Omega_{j,k}\subset\Omega_{i,k}$ for
all $i\ne j\ne k\in\left\{ 0,\dots,n\right\} $.
\par\end{flushleft}
\begin{defn}
\label{Def: matrix presentation} Consider a $\mathfrak{g}$-subspace
as above. Denote a matrix $\mathcal{A}\in\mathcal{P}\left(\mathbb{Z}\right)^{n\times n}$
as \emph{presentation matrix}, if the matrix entries of
\[
\mathcal{A}=\left\{ \begin{array}{cccc}
A_{0} & A_{01} & \cdots & A_{0n}\\
A_{10} & A_{1} &  & A_{1n}\\
\vdots &  & \ddots & \vdots\\
A_{n0} & A_{n1} & \cdots & A_{n}
\end{array}\right\} 
\]
satisfy $\Omega_{0,1}=A_{0}$, $\Omega_{n-1,n}=A_{n}$, $\Omega_{i-1,i}\cup\Omega_{i,i+1}=A_{i}$,$\left(0<i<n\right)$
and $\Omega_{i,j}\cap\Omega_{j,k}=\Omega_{i,k}$ for all $i,j,k$
such that $i-1\le j\le k+1$.
\end{defn}

By the above considerations every presentation matrix can be associated
with a subalgebra. For a presentation matrix to correspond to a subalgebra
uniquely, we need more conditions to be satisfied.

Let $\uplus$ and $\cup$ be the matrix operations in $\mathcal{P}\left(\mathbb{Z}\right)^{n\times n}$
inherited from the underlined set algebra $\left(\mathcal{P}\left(\mathbb{Z}\right),+,\cup\right)$.
The additive identity matrix therein is given by
\[
\mathrm{I}=\left\{ \begin{array}{cccc}
\left\{ 0\right\}  & \emptyset & \cdots & \emptyset\\
\emptyset & \left\{ 0\right\}  &  & \emptyset\\
\vdots &  & \ddots & \vdots\\
\emptyset & \emptyset & \cdots & \left\{ 0\right\} 
\end{array}\right\} .
\]
\begin{prop}
If the presentation matrix $\mathcal{A}\in\mathcal{P}\left(\mathbb{Z}\right)^{\left(n+1\right)\times\left(n+1\right)}$
satisfies the relation 
\[
\mathcal{A}\uplus\left(\mathcal{A}\cup\mathrm{I}\right)=\mathcal{A}
\]
 then it identifies a subalgebra of $A_{n}^{\left(1\right)}$ uniquely. 
\end{prop}

\begin{proof}
\noindent \begin{flushleft}
Let $\mathcal{A}$ be a presentation matrix. Relations \ref{eq:relations subalgebra exponent sets}
are a sufficient condition for $\mathfrak{X}$ to be a subalgebra,
because they take the Lie algebra relations into account, in particular
\begin{align*}
A_{i,j}+A_{j,i}\subset\Omega_{j,j} & \Leftrightarrow\mathbb{C}\left[e_{\alpha}\otimes t^{n_{1}},e_{-\alpha}\otimes t^{n_{2}}\right]=\mathbb{C}\check{\alpha}\otimes t^{n_{1}+n_{2}}\\
A_{i,j}+A_{j,k}\subset A_{i,k} & \Leftrightarrow\mathbb{C}\left[e_{\alpha}\otimes t^{n_{1}},e_{\beta}\otimes t^{n_{2}}\right]=\mathbb{C}e_{\alpha+\beta}\otimes t^{n_{1}+n_{2}}\\
A_{i,j}+\Omega_{k,l}\subset A_{i,j} & \Leftrightarrow\mathbb{C}\left[\check{\alpha}\otimes t^{n_{1}},e_{\beta}\otimes t^{n_{2}}\right]=\mathbb{C}e_{\beta}\otimes t^{n_{1}+n_{2}}\text{ if }\left(\alpha\mid\beta\right)\ne0
\end{align*}
for indices as above and corresponding roots $\alpha$ and $\beta$. 
\par\end{flushleft}
\noindent \begin{flushleft}
Denote $A_{ii}=A_{i}$. Then $\mathcal{A}\uplus\left(\mathcal{A}\cup\mathrm{I}\right)=\mathcal{A}$
is equivalent to $A_{i,j}=A_{i,j}\cup\bigcup_{k=0,\dots,n}\left(A_{i,k}+A_{k,j}\right)$
for all $i,j\in\left\{ 1,\dots,n\right\} $ or $A_{i,j}\supset\bigcup_{k=0,\dots,n}\left(A_{i,k}+A_{k,j}\right)$.
The first relation from \ref{eq:relations subalgebra exponent sets}
follows from 
\begin{align*}
A_{i,j}+A_{j,i} & \subset\bigcup_{k}\left(A_{i,k}+A_{k,i}\right)\subset A_{i,i}=\begin{cases}
\Omega_{0,1} & \text{if }i=0\\
\Omega_{n-1,n} & \text{if }i=n\\
\Omega_{i-1,i}\cup\Omega_{i,i+1} & \text{if }0<i<n
\end{cases}
\end{align*}
and $\Omega_{i-1,i}\cup\Omega_{i.i+1}=\left(\Omega_{i-1,j}\cap\Omega_{j,i}\right)\cup\left(\Omega_{i,j}\cap\Omega_{j,i+1}\right)=\left(\Omega_{i-1,j}\cup\Omega_{i+1,j}\right)\cap\Omega_{i,j}\subset\Omega_{i,j}$
for $i-1\le j\le i+1$ (Definition \ref{Def: matrix presentation}).
The second is obvious. For the third, we observe 
\[
A_{i,j}\supset\left(A_{i,j}+A_{j,j}\right)\cup\left(A_{i,i}+A_{i,j}\right)=A_{i,j}+\left(A_{i}\cup A_{j}\right)
\]
 and herein the second term, $A_{i}\cup A_{j}=\bigcup_{m=i,i+1,j,j+1}\Omega_{m-1,m}$.
Since 
\[
\Omega_{k,l}=\Omega_{k.r}\cap\Omega_{r,l}=\Omega_{k,r}\cap\left(\Omega_{r,k}\cap\Omega_{k,l}\right)=\Omega_{k,r}\cap\Omega_{k,l}
\]
and thus $\Omega_{k,l}\subset\Omega_{k,r}$ for $r$ such that $k-1\le r\le l+1$
(Definition \ref{Def: matrix presentation}). Repeating this step
results in $\Omega_{k,r}\subset\Omega_{s,r}$ for $s$ such that $k-1\le s\le r+1$.
Consequently, $\Omega_{k,l}\subset\Omega_{s,r}$ for $k,l$ such that
$r-1\le l$, $k\le s+1$ and thus $A_{i,j}\supset A_{i,j}+\Omega_{k,l}$
and $A_{j,i}\supset A_{j,i}+\Omega_{k,l}$ $\left(i<j\right)$, for
$i-1\le k<l\le j+1$ as claimed.
\par\end{flushleft}
For $\mathcal{S}ub\left(\mathfrak{g}\right)$ the category of subalgebras
of $\mathfrak{g}$, this implies also that all coefficients that determine
$\mathfrak{X}\in\mathcal{S}ub\left(A_{n}^{\left(1\right)}\right)$
are uniquely determined. 
\end{proof}
Denote the category of presentation matrices $\mathcal{A}\in\mathcal{P}\left(\mathbb{Z}\right)^{\left(n+1\right)\times\left(n+1\right)}$
that satisfy $\mathcal{A}\uplus\left(\mathcal{A}\cup\mathrm{I}\right)=\mathcal{A}$
by $\mathcal{E}_{n}^{\star}$ and the set of subalgebras of $\mathfrak{g}=A_{n}^{\left(1\right)}$
that contain $c$ and $d$ by $\mathcal{S}ub\left(\mathfrak{g}\right)$.
With this, define a map 
\[
\mathcal{E}_{n}^{\star}\rightarrow\mathcal{S}ub\left(\mathfrak{g}\right):\:\mathcal{A}\mapsto\mathfrak{X}
\]

according to Equation \ref{eq:subspace}. The preimage of an $A_{n}^{\left(1\right)}$-subalgebra
$\mathfrak{X}$ under this map will be called \emph{matrix presentation}
of $\mathfrak{X}$\index{matrix presentation of an A_{n}^{left(1right)}-subalgebra@matrix presentation of an $A_{n}^{\left(1\right)}$-subalgebra}. 
\begin{notation}
We may also abuse the matrix notation to denote the Lie algebra closure
\begin{align*}
\mathrm{cl}_{\mathfrak{g}}\left(\left\{ e_{\alpha_{i}+\cdots+\alpha_{j}+A_{i,j+1}\delta},\:e_{-\left(\alpha_{i}+\cdots+\alpha_{j}\right)+A_{j+1,i}\delta}\mid i\le j\in\left\{ 1,\dots,n\right\} \right\} \right.\\
\cup\left.\left\{ e_{A_{i}\delta}^{\alpha_{i}},\,e_{A_{i}\delta}^{\alpha_{i+1}},\,e_{A_{0}\delta}^{\alpha_{1}},\,e_{A_{n}\delta}^{\alpha_{n}}\mid i\in\left\{ 1,\dots,n-1\right\} \right\} \cup\left\{ c,d\right\} \right)
\end{align*}
if this is clear from context. In this case the denomination is not
unique, e.g. 
\[
\left\{ \begin{array}{ccc}
* & * & *\\
* & \left\{ 1\right\}  & \left\{ 1\right\} \\
* & \left\{ 1\right\}  & \left\{ 1\right\} 
\end{array}\right\} =\left\{ \begin{array}{ccc}
* & * & *\\
* & \mathbb{Z}_{+} & \mathbb{Z}_{+}\\
* & \mathbb{Z}_{+} & \mathbb{Z}_{+}
\end{array}\right\} 
\]
where the left hand side is not a presentation matrix.
\end{notation}

\begin{defn}
If all of the sets $A_{i,j}$ and $A_{i}$, $\left(i,j=1,\dots,n\right)$
are of type $\mathbb{Z}_{\ge k}$, then we use round paranthesis and
write 
\[
\left\{ \begin{array}{cccc}
\mathbb{Z}_{\ge k_{0}} & \mathbb{Z}_{\ge k_{01}} & \mathbb{Z}_{\ge k_{02}} & \cdots\\
\mathbb{Z}_{\ge k_{10}} & \mathbb{Z}_{\ge k_{1}} & \mathbb{Z}_{\ge k_{12}}\\
\mathbb{Z}_{\ge k_{20}} & \mathbb{Z}_{\ge k_{21}} & \mathbb{Z}_{\ge k_{2}}\\
\vdots &  &  & \ddots
\end{array}\right\} =\left(\begin{array}{cccc}
k_{0} & k_{01} & k_{02} & \cdots\\
k_{10} & k_{1} & k_{12}\\
k_{20} & k_{21} & k_{2}\\
\vdots &  &  & \ddots
\end{array}\right).
\]
\end{defn}

\begin{fact}
If $k_{0}=k_{1}=\cdots=k_{n}=1$ then these subalgebras are quasicones,
i.e. elements in $\mathfrak{C}$. 
\end{fact}

Denote $\left(\hat{\mathbb{Z}},\oplus,\odot\right)=\left(\mathbb{Z}\cup\left\{ \pm\infty\right\} ,\max,+\right).$
the max-plus semi-ring. The identity ($\odot$-multiplicative neutral
element) in the corresponding max-plus matrix algebra is given by
$\mathbf{I}=\mathbf{I}_{\max}=\mathrm{diag}\left(\mathbf{0}\right)$
(see \cite{SpeyerSturmfels2009} for an introduction to tropical mathematics).
\begin{lem}
\label{lem:tropicalLieClosure}If $B$ is a subset of $\mathfrak{g}$
that contains $\mathfrak{h}^{\circ}\otimes t\mathbb{C}\left[t\right]$,
but $B\cap\mathfrak{h}^{\circ}\otimes\mathbb{C}\left[t^{-1}\right]=\emptyset$,
then its Lie-algebraic closure is given by 
\[
\mathrm{cl}_{\mathfrak{g}}\left(B\right)=B\odot'\left(B\oplus'\mathbf{I}'\right).
\]
\end{lem}

\begin{proof}
With the above, we have
\[
c_{i,j}=\left(B\odot'\left(B\oplus'\mathbf{I}'\right)\right)_{i,j}=\min_{k}\left(b_{i,k}+b_{k,j},b_{i,j}+0\right)\,,\;\mbox{ for all }i,j=0,\dots,n,
\]
we need to show formula \ref{eq:relations subalgebra exponent sets},
which is in this case simply
\begin{equation}
c_{i,j}\le c_{i,k}+c_{k,j}\mbox{ for all }k\in\left\{ 0,\dots,n\right\} \smallsetminus\left\{ i,j\right\} ,\:i\ne j.\label{eq:inequalities}
\end{equation}
If $b_{i,j}\le b_{i,k}+b_{k,j}$ for all $k$, then $c_{i,j}\le b_{i,j}+0\le b_{i,k}+b_{k,j}\le c_{i,k}+c_{k,j}$.
If otherwise $b_{i,j}>b_{i,k}+b_{k,j}$ for some $k$, then $c_{i,j}\le b_{i,k}+b_{k,j}$
for that $k$. Now $b_{i,k}+b_{k,j}\le\min_{k}\left(b_{i,k}+b_{k,j},b_{i,j}+0\right)=c_{i,j}\le c_{i,k}+c_{k,j}$,
as desired. 
\end{proof}
\begin{cor}
The subalgebras containing $\mathfrak{h}^{\circ}\otimes\mathbb{C}\left[t\right]$
are exactly the idempotents in the tropical matrix rings $\left(C,\odot'\right)$
and $\left(C,\overleftrightarrow{\odot}_{\min}\right)$.
\end{cor}

\begin{rem}
(ii) Fix a basis $\Pi^{\circ}=\left\{ \varphi_{1},\dots,\varphi_{n}\right\} $
and denote $c_{2^{i}}=c_{\varphi_{i}}\:\left(i=1,\dots,n\right)$.
Denote by $I_{n}=\left\{ 2^{i}+\cdots+2^{j}\mid0\le i\le j\le n\right\} $
the set of admissible indices for quasicone matrices for $A_{n}^{\left(1\right)}$.
The property of $I_{n}$ to be a linearly ordered set of order $q$
is required in the main algorithm at the end of the paper.
\end{rem}

\subsection{Defect of a quasicone}
\begin{defn}
Define the defect function $\#:\mathfrak{C}\to\mathbb{N}$ by 
\begin{equation}
\#C=\sum_{\varphi\in\Delta_{+}^{\circ}}\left(c_{\varphi}+c_{-\varphi}-2\right)_{+}\,.\label{eq:Defect}
\end{equation}
It aims to measure how much a quasicone fails to be a cone, and therefore
the corresponding subalgebra fails to be a maximal parabolic.
\end{defn}

\begin{rem}
A subalgebra $C$ may only fail to be a quasicone if $\mathfrak{h}^{\circ}\otimes t\mathbb{C}\left[t\right]$
is not entirely contained or there exists a root $\varphi\in\Delta^{\circ}$
such that either 
\begin{align*}
 & \max_{k\in\mathbb{Z}}\left\{ e_{\varphi+k\delta}\notin C\right\} \mbox{ does not exist, or }\min_{k\in\mathbb{Z}}\left\{ e_{\varphi+k\delta}\in C\right\} \mbox{ does not exist.}
\end{align*}
\end{rem}

If $\Pi^{\circ}=\left\{ \alpha_{1},\dots,\alpha_{n}\right\} $, then
a change of basis $\Pi\mapsto\Pi'$ of $\Delta$ is accomplished by
choosing linearly independent roots $\tilde{\alpha}_{1},\dots,\tilde{\alpha}_{n}\in\mathrm{add}_{\Delta}\Pi^{\circ}$
and extending it to $\Pi'$ canonically. Then $\left(\Pi'\right)^{\circ}=\left\{ \tilde{\alpha}_{1},\dots,\tilde{\alpha}_{n}\right\} $. 
\begin{defn}
\label{def:normal quasicone}A quasicone matrix, respectively a quasicone
subalgebra, is given in \emph{normal form} or \emph{normal} if $c_{\varphi}=1$
for all $\varphi\in\Pi^{\circ}$ and $c_{\kappa}+c_{-\kappa}\ge c_{\nu}+c_{-\nu}$
for all $\kappa,\nu\in I_{n}$ with $\kappa<\nu$. For the rest of
the thesis we will generically refer to a quasicone matrix, a quasicone
subalgebra or a quasicone of roots by \emph{quasicone\index{quasicone}}
if the structure is clear from the context.
\end{defn}

\begin{lem}
\label{lem:base_change} Any quasicone $C$ is equivalent to a normal
quasicone, i.e. there is an automorphism $\varphi\in\mathrm{Aut}\left(\mathfrak{g}\right)$
that induces a change of basis and thereby a map of quasicones $\varphi:C\mapsto C'$
with $C_{\kappa,\kappa+1}=1$ \textup{and $c_{\kappa}+c_{-\kappa}\ge c_{\nu}+c_{-\nu}$
for all $\kappa,\nu\in I_{n}$ with $\kappa<\nu$.}
\end{lem}

\begin{proof}
Because any quasicone $C\subset\mathfrak{g}$ is a subalgebra of $\mathfrak{g}$,
any automorphism of $\mathfrak{g}$ induces an isomorphism of quasicones.
The Weyl group $\mathcal{W}^{\circ}$ acts transitively and faithfully
on the set of bases $\mathcal{B}$ for the root system $\Delta^{\circ}$.
First, we show that there is a $w\in\mathcal{W}^{\circ}$ such that
$w\left(C\right)$ satisfies $c_{\kappa}+c_{-\kappa}\ge c_{\nu}+c_{-\nu}$
for all $\kappa,\nu\in I_{n}$ with $\kappa<\nu$. Select $w_{0}=\max_{\preceq}\left(\mathcal{W}^{\circ}\right)$
with respect to the Bruhat order $\preceq$ on $\mathcal{W}^{\circ}$.
The order of $\left|\mathcal{W}^{\circ}\left(\mathfrak{g}\right)/\left\langle w_{0}\right\rangle \right|=q!$.
Since $\mathcal{W}^{\circ}\left(\mathfrak{g}\right)/\left\langle w_{0}\right\rangle $
acts faithfully on the set
\[
\left\{ \left\{ \Pi,w_{0}\Pi\right\} \mid\Pi\in\mathcal{B}\right\} \subset\mathcal{B}\times\mathcal{B},
\]

which is of order $q!$, it acts by permutation on $\Delta_{+}^{\circ}\left(\Pi\right)$.
This induces a canonical action on the ordered set $\left(c_{\kappa}+c_{-\kappa}\mid\kappa\in I_{n}\right)$,
because $I_{n}\equiv\Delta_{+}^{\circ}\left(\Pi\right)$. Eventually
$\mathcal{W}^{\circ}\left(\mathfrak{g}\right)/\left\langle w_{0}\right\rangle $
contains an element $w$ such that $\left(c_{\kappa}+c_{-\kappa}\mid\kappa\in w\left(I_{n}\right)\right)$
has the desired order.

Now we show the existence of an isomorphism $\tau$ of quasicones
that yields only 'ones' on the superdiagonal in $\tau\left(C\right)$.
Recall that $C=\left\langle \left\{ e_{\alpha_{\kappa}+c_{\kappa}\delta}\mid\kappa\in w\left(I_{n}\right)\right\} \cup\mathfrak{h}^{\circ}\otimes t\mathbb{C}\left[t\right]\right\rangle $.
The components for the desired map $\tau:\left(c_{\kappa}\mid\kappa\in w\left(I_{n}\right)\right)\mapsto\left(c_{\kappa}'\mid\kappa\in w\left(I_{n}\right)\right)$
are \emph{ad hoc} given by
\[
\left(\tau\mid_{\kappa}\right)\left(c_{\kappa}\right)=c_{\kappa}+\sum_{i=k_{1}}^{k_{2}}\left(1-c_{2^{i}}\right)\mbox{ if }\kappa=\sum_{i=k_{1}}^{k_{2}}2^{i}\in w\left(I_{n}\right),\,\left(0\le k_{1}\le k_{2}<n\right)\,.
\]
\end{proof}

\subsection{Order relations on quasicones}
\begin{defn}
\label{Partial order on C}Let's define three partial orders on $\mathcal{C}$
by 

(i) $C\le^{\mathrm{\left(i\right)}}C'$ if $C_{\nu}=C'_{\nu}$ for
all $\nu\in I_{n}$ and $C_{\kappa}<C'_{\kappa}$ for some $\kappa\in-I_{n}$
or

(ii) $C\le^{\mathrm{\left(ii\right)}}C'$ if $C_{\kappa}<C'_{\kappa}$
for some $\kappa\in-I_{n}$ and $C_{\nu}+C_{-\nu}=C_{\nu}'+C{}_{-\nu}'$
for all $\nu\in I_{n}$. 

(iii) $\le=\le^{\mathrm{\left(i\right)}}\cup\le^{\mathrm{\left(ii\right)}}$,
i.e. $C\le C'$ if $C\le^{\mathrm{\left(i\right)}}C'$ or $C\le^{\mathrm{\left(ii\right)}}C'$.
\end{defn}

\begin{rem}
(i) The set of representatives of $\mathcal{C}$ in normal form is
equipped with the inclusion partial order forms a complete join-semilattice
$\left(\mathcal{C},\subseteq,\bigvee^{\subseteq}\right)$ of subalgebras
with infimum $\left\{ 0\right\} $ and the greatest upper bound $\mathfrak{g}$
itself. Thus, its order dual is a complete semilattice $\left(\mathcal{C},\supseteq,\bigwedge^{\supseteq}\right)$
(cf. \cite{Nation1984}).

(ii) Since the union $\le^{\mathrm{\left(i\right)}}\cup\le^{\mathrm{\left(ii\right)}}$
is disjoint, there is a split exact sequence of posets 
\begin{equation}
0\rightarrow\left(\mathcal{C},\le^{\mathrm{\left(ii\right)}}\right)\rightarrow\left(\mathcal{C},\le\right)\rightarrow\left(\mathcal{C},\le^{\mathrm{\left(ii\right)}}\right)\rightarrow0\,.\label{eq:exact_sequence_posets}
\end{equation}
Recall that the sequence $\left(C_{\nu}+C_{-\nu}\mid\nu\in I_{n}\right)$
is monotonically decreasing. Consider the set of monotonically decreasing
positive integer sequences 
\[
\mathbb{N}_{\ge}^{q}=\left(k_{\nu}\mid k_{\nu}\ge k_{\nu'}\mbox{ if }\nu\le\nu',\,\nu\in I_{n}\right)\subset\mathbb{N}^{q}\,,
\]
equipped with the natural partial order, which is equivalent to the
lexical total order thereon. Define the map $\gamma:\mathbb{N}_{\ge}^{q}\to\left(\mathcal{C},\le^{\mathrm{\left(i\right)}}\right)$
by 
\[
a\mapsto\gamma\left(a\right)=\bigwedge\left\{ C\in\mathcal{C}\mid\left(C_{\nu}+C_{-\nu}\right)_{\nu\in I_{n}}=a\right\} .
\]
In fact, $\gamma$ is well-defined. It is injective since two quasicones
with a different defect for any of its sub-quasicones cannot be equal.
Denote the vector $\gamma^{-1}\left(C\right)=\left(C_{\nu}+C_{-\nu}\mid\nu\in I_{n}\right)$
by gap of $C$,
\begin{equation}
\mbox{gap}\left(C\right)=\gamma^{-1}\left(C\right)\,.\label{eq:gap}
\end{equation}

(iii) The gap of $C$ is closely related to the defect function, precisely
$\#C=\sum_{\nu\in I_{n}}\left(\mbox{gap}\left(C\right)_{\nu}-2\right)_{+}$.
\end{rem}

(iv) The non-trivial representative that is the greatest lower bound
in $\mathcal{C}$ is the cone associated to the matrix 
\[
\bigwedge\mathcal{C}\sim\left(\begin{array}{ccccccc}
1 & 1 & 2 & 3 &  & \cdots & n\\
-1 & 1 & 1 & 2 &  & \cdots\\
-2 & -1 & 1 & 1\\
-3 & -2 & -1 & \ddots & \ddots &  & \vdots\\
\vdots &  &  & \ddots &  & 1 & 2\\
 &  &  &  &  & 1 & 1\\
-n &  &  & \cdots &  & -1 & 1
\end{array}\right)=\gamma\left(0^{\times q}\right).
\]
 Thus $\left(\mathcal{C},\le\right)$ is a complete semilattice. 
\begin{defn}
Select arbitrary elements $C^{\mathrm{up}},C^{\mathrm{low}}\in\mathcal{C}$.
A quotient (complete) sublattice $C^{\mathrm{up}}/C^{\mathrm{low}}\subset\mathcal{C}$
is defined as 
\[
C^{\mathrm{up}}/C^{\mathrm{low}}=\left\{ C\in\mathcal{C}\mid C^{\mathrm{low}}\le C\le C^{\mathrm{up}}\right\} .
\]
Consequently, every subset $C^{\mathrm{up}}/C^{\mathrm{low}}\subset\mathcal{C}$
is a (complete) lattice. Now, consider the following quasicones 
\begin{equation}
\tilde{C}\sim\left(\begin{array}{ccccccc}
1 & 1 & 2 & 3 &  & \cdots & n\\
\tilde{c_{1,0}} & 1 & 1 & 2 &  & \cdots & n-1\\
\tilde{c_{2,0}} & \tilde{c_{2,1}} & 1 & 1\\
\tilde{c_{3,0}} & \tilde{c_{3,1}} & \tilde{c}_{3,2} & 1 & \ddots &  & \vdots\\
\vdots &  &  &  & \ddots &  & 2\\
 &  &  &  &  & 1 & 1\\
\tilde{c}_{n,0} &  &  & \cdots &  & \tilde{c}_{n,n-1} & 1
\end{array}\right).\label{eq:up-cone}
\end{equation}
These are a lower bound with respect to $\le^{\mathrm{\left(ii\right)}}$,
because there are no normal quasicones $C$ such that $C_{\nu}>\left(\tilde{C}\right)_{\nu}$
can be true for any $\nu>0$. To define the matrices $\tilde{C}_{\tilde{d}}^{\mathrm{up}}=\gamma\left(\left(\tilde{d}+1\right)^{\times q}\right)$,
$\tilde{d}\in\mathbb{Z}_{\ge-1}$, set
\begin{align*}
\tilde{c}_{-1}=\tilde{c}_{-2}=\cdots=\tilde{c}_{-2^{n-1}} & =:\tilde{d}\\
\tilde{c}_{-3}=\tilde{c}_{-6}=\cdots=\tilde{c}_{-\left(2^{n-2}+2^{n-1}\right)} & =\tilde{d}-1\\
 & \vdots\\
\tilde{c}_{-\left(2^{n}-1\right)} & =\tilde{d}-n+1
\end{align*}
with the general rule $\tilde{c}_{\nu}=\tilde{d}-\ell\left(\nu\right)+1$
for all $\nu<0$. Then $\tilde{c}_{\nu}+\tilde{c}_{-\nu}=\tilde{d}+1$
for all $\nu$ and all inequalities \ref{eq:inequalities} are satisfied,
so that this really represents a quasicone. Note that $\tilde{C}_{-1}^{\mathrm{up}}$
is the upper bound of the lattice.
\end{defn}

\subsection{Affine Weyl group actions and direct sums}

The group of translations $\mathcal{T}$ are the $\mathbb{Z}$-modules
generated by rank two block matrices $t_{i}$ that act via common
addition on $\mathfrak{C}$ (tropical Hadamard $\odot$-product).
The Weyl group $\mathcal{W}^{\circ}$ for the simple root system $\Delta^{\circ}$
is isomorphic to $S_{n}$, thus generated by transpositions which
we will identify with the rank one matrices $s_{i,j}$: 
\[
t_{i}=\left(\begin{array}{cccccc}
0 & \cdots & 0 & 1 & \cdots & 1\\
\vdots & \ddots &  & \vdots &  & \vdots\\
 &  & 0 & 1 & \cdots & 1\\
-1 & \cdots & -1 & 0 &  & \vdots\\
\vdots &  & \vdots &  & \ddots\\
-1 & \cdots & -1 &  &  & 0
\end{array}\right),\;i=1,\dots,n,\qquad s_{i,j}=\left(\begin{array}{cccccc}
0 & \cdots &  &  & \cdots & 0\\
\vdots & \ddots &  &  &  & \vdots\\
 &  & 1 & -1\\
 &  & -1 & 1\\
\vdots &  &  &  & \ddots & \vdots\\
0 & \cdots &  &  & \cdots & 0
\end{array}\right),
\]
$i,j=0,\dots,n,\:i\ne j,$ where the $s_{i,j}$ with $j=i+1$ form
a minimal generating set. This group acts on $\mathfrak{C}$ via row-column
permutations $s_{ij}Cs_{ij}$. Thus the affine Weyl group $\mathcal{W}=\mathcal{W}^{\circ}\ltimes\mathcal{T}$
acts via $s_{ij}\left(C+t_{i}\right)s_{ij}.$

Because of matrix multiplication and the fact that the empty set $\emptyset$
serves as $\cup$-additive neutral element, we can build new pre-prosolvable
presentation matrices, by taking direct sums. 

For $A\in\mathcal{P}\left(\mathbb{Z}\right)^{n\times n}$ and $B\in\mathcal{P}\left(\mathbb{Z}\right)^{k\times k}$,
define

\[
A\boxplus B=\left\{ \begin{array}{cc}
A & \emptyset^{n\times k}\\
\emptyset^{k\times n} & B
\end{array}\right\} \text{ and }A\boxbslash B=\left\{ \begin{array}{cc}
A & \emptyset^{n\times k}\\
\mathbb{Z}^{k\times n} & B
\end{array}\right\} .
\]

Let $\alpha_{1},\dots,\alpha_{k}\in\Pi^{\circ}$. Then denote $\mathfrak{n}_{\pm\left\{ \alpha_{1},\dots,\alpha_{k}\right\} }^{\square}:=\sum_{i=1}^{k}\mathfrak{g}_{\mathcal{W}\left(\Pi^{\circ}\smallsetminus\alpha_{i}\right)\left\{ \pm\alpha_{i}\right\} +\mathbb{Z}\delta}$.
The basic example would be 
\[
\mathfrak{n}_{\left\{ -\alpha_{k}\right\} }^{\square}=\emptyset^{k\times k}\boxbslash\emptyset^{\left(n-k\right)\times\left(n-k\right)}=\left\{ \begin{array}{cc}
\emptyset^{k\times k} & \emptyset^{n\times k}\\
\mathbb{Z}^{k\times n} & \emptyset^{\left(n-k\right)\times\left(n-k\right)}
\end{array}\right\} =\left(\begin{array}{cc}
\infty^{k\times k} & \infty^{n\times k}\\
-\infty^{k\times n} & \infty^{\left(n-k\right)\times\left(n-k\right)}
\end{array}\right).
\]

For a principal parabolic system $P=\Delta_{+}\dot{\cup}\Delta{}_{0}$,
consider the corresponding ideal in the Lie algebra, i.e. $\mathfrak{g}_{+}=\mathfrak{g}_{\Delta_{+}}$.
Thanks to the matrix presentation \ref{Def: matrix presentation},
it is possible to describe those ideals by means of a block decomposition,
each block representing a different type of ideal for a parabolic
in a subalgebra $A_{k}^{\left(1\right)}\subset A_{n}^{\left(1\right)}$.

\section{Futorny's Support Conjecture for $A_{n}^{\left(1\right)}$ }

\subsection{\label{sec:Applying-strategies-on}Tropical Lie Actions on annihilating
quasicones}

From now on, $V$ is always an irreducible non-dense weight $A_{n}^{\left(1\right)}$-module.
For all $n\ge2$, we can reduce the problem of finding primitive elements
when only finitely many quasicone subalgebras do not act trivially.
We will give an explicite upper bound in the quasicone lattice $\mathcal{C}$
for those cases to occur. 

By Definition \ref{def: quasicone loc affine} a subalgebra $\mathfrak{s}\subset\mathfrak{g}$
is a quasicone subalgebra if it contains $\mathcal{H}_{+}$ and has
a trivial intersection with $\mathfrak{h}$. Therefore, any quasicone
subalgebra of $\mathfrak{g}\left(S\right)$, for a partition $S\subset\Pi^{\circ}$
of a basis of $\Delta^{\circ}$, is equal to 
\[
C_{S}\left(\mathbf{k}\right)=\left(\sum_{\alpha\in\Delta^{\circ}\left(S\right)}\sum_{n_{\alpha}\ge k_{\alpha}}\mathfrak{g}_{\alpha+n_{\alpha}\delta}\right)\oplus\mathcal{H}_{+}\left(S\right)
\]
for integers $\mathbf{k}=\left(k_{\alpha}\mid\alpha\in\Delta^{\circ}\left(S\right)\right)$
with the property $k_{\alpha}+k_{-\alpha}\ge1$ and $k_{\alpha}+k_{\beta}\ge k_{\alpha+\beta}$
whenever $\alpha,\beta$ and $\alpha+\beta$ are roots. 

It is obvious that the integer vector $\mathbf{k}$ defines the quasicone
subalgebra $C_{S}\left(\mathbf{k}\right)$ uniquely up to isomorphism.
Denote the sets of quasicone subalgebras by $\mathfrak{C}=\mathfrak{C}_{\Pi^{\circ}}$
and $\mathfrak{C}_{S}$, respectively. The name quasicone subalgebra
is justified, since in general $\mathfrak{g}_{0}\cap C_{S}\left(\mathbf{k}\right)=\emptyset$
and the index set, where $\mathcal{U}\left(\mathfrak{h}+C_{S}\left(\mathbf{k}\right)\right)$
is supported, is equal to the intersection of the root lattice with
a cone. In other words, a quasicone subalgebra is a subalgebra over
a blunt cone of roots.
\begin{defn}
\label{def:semi-primitive}Let $\mathfrak{a}\subset\mathfrak{g}$
be a subalgebra and $V$ a weight $\mathfrak{g}$-module. \\
(i) A vector $v\in V$ is called $\mathfrak{a}$\emph{-semiprimitive}\index{semiprimitive, mathfrak{a}-semiprimitive@semiprimitive, $\mathfrak{a}$-semiprimitive}
if $\mathfrak{a}v=0$. \\
(ii) $C_{S}\left(\mathbf{k}\right)$ is a \emph{GVM-complete quasicone
subalgebra\index{GVM-complete quasicone subalgebra} }or just\emph{
complete }\index{complete, rightarrowGMV-complete@complete, $\rightarrow$GMV-complete}\emph{,}
if $k_{\alpha}+k_{-\alpha}\in\left\{ 1,2\right\} $ for all $\alpha\in\mathrm{add}_{\Delta^{\circ}}\left(S\right)$.\\
(iii) If $\mathfrak{a}=C_{\Pi^{\circ}}$ is complete or $\mathfrak{a}=\mathfrak{n}_{\Pi^{\circ}}$
or $\mathfrak{a}=C_{X}\oplus\mathfrak{n}_{X^{c}}$ for some $\left\{ 0\right\} \ne X\subset\Pi^{\circ}$
and $C_{X}$ is complete and $v\in V$ is $\mathfrak{a}$\emph{-semiprimitive},
then $v$ is called \index{primitive, mathfrak{a}-primitive@primitive, $\mathfrak{a}$-primitive}$\mathfrak{a}$-\emph{primitive}.
\end{defn}

\begin{fact}
Fact. If $v$ is $\mathfrak{a}$-primitive then $v$ is primitive.
\end{fact}

\begin{proof}
It is sufficient to show that for each subalgebra $\mathfrak{a}$,
there exists a parabolic subalgebra $\mathfrak{p}=\mathcal{L}\oplus\mathcal{N}$
according to Thm. \ref{thm:parabolic induction} such that $\mathfrak{a}=\mathcal{N}$.

Let $A$ be one of the root sets $\left\{ \alpha+n_{\alpha}\delta\mid n_{\alpha}\ge k_{\alpha},\alpha\in\Delta^{\circ}\right\} \cup\mathbb{Z}_{+}\delta$
($\Delta_{+}^{\circ}+\mathbb{Z}\delta$, $\left\{ \alpha+n_{\alpha}\delta\mid n_{\alpha}\ge k_{\alpha},\alpha\in\Delta^{\circ}\left(X\right)\right\} \cup\mathbb{Z}_{+}\delta\cup\Delta_{+}^{\circ}\left(X^{c}\right)$)
corresponding to $C_{\Pi^{\circ}}$ ($\mathfrak{n}_{\Pi^{\circ}}$,
$C_{X}\oplus\mathfrak{n}_{X^{c}}$, respectively). Set $\mathfrak{p}=\mathfrak{p}\left(P\right)$
with $P=\Delta\smallsetminus-A$, which is indeed parabolic. Also
$\mathfrak{a}=\mathfrak{g}_{A}=\mathcal{N}$ and $\mathfrak{p}=\mathfrak{g}_{P\smallsetminus A}\oplus\mathfrak{g}_{A}$.
\end{proof}
For a $\mathfrak{g}$-module $V$ and some $v\in V$, denote by $\mathcal{A}nn\left(v\right)=\left\{ g\in\mathfrak{g}\mid gv=0\right\} \subset\mathfrak{g},$
and 

\[
\mathrm{Ann}\left(v\right)=\left\{ \varphi\in\Delta\mid\mbox{there is a }g\in\mathcal{A}nn\left(v\right)\mbox{ that satisfies }g\in\mathfrak{g}_{\varphi}\right\} \subset\Delta.
\]
As an immediately obvious matter of fact, $\mathrm{Ann}\left(v\right)$
is closed under addition in $\Delta$.

Denote by $\mathfrak{G}\subset2^{\mathfrak{g}}$ the set of subalgebras
of $\mathfrak{g}$. For the root operator $e_{\varphi}$ and $C=\mathcal{A}nn\left(w\right)$,
$\left(\varphi\in\Delta^{\mathrm{re}},w\in V\right)$, define a map
$e_{\varphi}:V\times\mathfrak{C}\to V\times\mathfrak{G}$ by $e_{\varphi}\left(w,C\right)=\left(e_{\varphi}w,C'\right)$,
such that $C'=\mathcal{A}nn\left(e_{\varphi}w\right)$ for some given
$w\in V$. Choose $\varphi\in\Delta$ such that $\mathcal{A}nn\left(e_{\varphi}v\right)\in\mathfrak{C}$,
then this gives rise to an action 
\begin{align}
e_{\varphi}:Q\times\mathfrak{C}\rightarrow & Q\times\mathfrak{C}\,:\quad\left(\vartheta,C\right)\mapsto\left(\vartheta',C'\right)\label{eq:action}
\end{align}
where $\vartheta'=\vartheta+\varphi$ and 
\[
C'=\min_{V\mid V_{\mu}=\left\{ 0\right\} }\left\{ \mathcal{A}nn\left(e_{\varphi}v\right)\mid v\in V_{\mu+\vartheta}\mbox{ and }\mathcal{A}nn\left(v\right)=C\right\} ,
\]
$V$ going over all irreducible non-dense weight $\mathfrak{g}$-modules
and the minimum refers to the partial order given by inclusion on
the subalgebras. The lower bound in $\mathfrak{G}$ is $\left\{ 0\right\} $.
For that reason, Zorn's lemma guarantees the existence of such a minimal
subalgebra. The function is well-defined. In fact, the image of $C$
under $e_{\varphi}\left(\vartheta,\cdot\right)$ is given by the Lie
algebra closure
\begin{align*}
\mathrm{cl}_{\mathfrak{g}}\left\langle \left(\mathrm{ad}e_{\varphi}\right)^{-1}\left(C\right)\cup\left\{ g\right\} \right\rangle  & =\mathrm{cl}_{\mathfrak{g}}\left\langle \left\{ e_{\psi}\mid\left[e_{\varphi},e_{\psi}\right]\in C\right\} \cup\left\{ g\right\} \right\rangle ,\mbox{ where}\\
g & =\begin{cases}
e_{-\left(\varphi+\vartheta\right)} & \mbox{if }\varphi+\vartheta\in\Delta^{re}\\
\mathfrak{h}^{\circ}\otimes t^{-k} & \mbox{if }\varphi+\vartheta=k\delta\:.
\end{cases}
\end{align*}
\begin{lem}
\label{h-Argument}Let $w\in V$ and $C=\mathcal{A}nn\left(w\right)$
be a quasicone subalgebra. If $\varphi\in\Delta^{\mathrm{re}}$ and
$e_{\varphi+\delta}\in\mathcal{A}nn\left(w\right)$, then $\mathcal{A}nn\left(e_{\varphi}w\right)$
contains $\mathfrak{h}^{\circ}\otimes t\mathbb{C}\left[t\right]$.
\end{lem}

\begin{proof}
Choose a basis $\Pi$ for $\Delta$ such that $\varphi\in\Delta_{+}^{\circ}\left(\Pi\right)$.
Write $e_{k\delta}^{\left(\varphi\right)}=\left(h_{\varphi}\otimes t^{k}\right)$.
If $e_{\varphi}$ does not act trivially, which would trivially meet
the assertion, then 
\[
e_{k\delta}^{\left(\psi\right)}e_{\varphi}w=\left(\varphi\left(h_{\psi}\right)e_{\varphi+k\delta}+e_{\varphi}e_{k\delta}^{\left(\psi\right)}\right)w=0\mbox{ for all }\psi\in\Delta^{\circ}\left(\Pi\right)\mbox{ and }k>0.
\]
\end{proof}
\begin{defn}
\label{def: strategy} A \emph{\index{strategy}strategy for $v$}
is a composition of operators $s=e_{\varphi_{m}}\circ\cdots\circ e_{\varphi_{1}}$,
$\left(m\in\mathbb{Z}_{+}\right)$, that satisfies

\noindent\begin{minipage}[t]{1\columnwidth}%
\begin{itemize}
\item[ {\scriptsize \bf (S1)} ]$sv=e_{\varphi_{m}}\circ\cdots\circ e_{\varphi_{1}}v\ne0$ and

\item[ {\scriptsize \bf (S2)} ] if $C=\mathcal{A}nn\left(v\right)$
is a quasicone, then $\mathcal{A}nn\left(sv\right)$ is also a quasicone.

\item[ {\scriptsize \bf (S3)} ] $\varphi_{i}+\cdots+\varphi_{1}\in\Delta$
for $i\in\left\{ 1,\dots,m\right\} $.

\end{itemize}%
\end{minipage}

Denote the set of strategies by $\mathfrak{S}$. The strategy is said
to \emph{succeed }(or to be\emph{ \index{successful strategy}successful})
on the quasicone $C=\mathcal{A}nn\left(v\right)$ if and only if $\#\left(\mathcal{A}nn\left(sv\right)\right)<\#C$.
A strategy is called \emph{\index{circular strategy}circular} if
$\varphi_{1}+\cdots+\varphi_{n}\in\mathbb{Z}\delta$. We may say \emph{strategy
for} $C$ assuming implicitely the existence of a $v$ with the properties
given above. The \index{length of a strategy}\emph{length} of the
strategy $\ell\left(s\right)$ is the integer $n$.
\end{defn}

The length function and the function 
\[
s:Q\times\mathfrak{C}\rightarrow Q\times\mathfrak{C}\,:\;\left(\vartheta,C\right)\mapsto s\left(\vartheta,C\right)=e_{\varphi_{1}}\circ\cdots\circ e_{\varphi_{n}}\left(\vartheta,C\right)
\]
are well-defined. We use the arrow '$\rightsquigarrow$' and index
notation to indicate this transformation as $C_{\vartheta}\overset{s}{\rightsquigarrow}C'_{\vartheta'}$
and omit the $\vartheta$-subscript if it is clear from the context.
\begin{conjecture}
\textup{\label{conj:finite set of strategies} There is a finite set
of strategies $\mathcal{S}\subset\mathfrak{S}$ such that the number
of normal quasicones where no strategy succeeds is zero,
\[
\bigcap_{C\in\mathfrak{C}}\left\{ \#\mathcal{A}nn\left(sv\right)\ge\#C\mbox{ for all }s\in\mathcal{S}\right\} =\emptyset.
\]
}
\end{conjecture}

Assume $C$ is given in normal form. Because of inequalities \ref{eq:inequalities},
it follows that $c_{\nu}\le\ell\left(\nu\right)$. Given $v\in V_{\mu-\varepsilon\delta}$,
there is a $k\in\mathbb{N}$ such that
\[
s_{k}=e_{-\alpha_{1}+k\delta}\circ e_{\alpha_{1}}
\]
 is a circular strategy for $v$ because of Lemma \ref{h-Argument}.
It is of minimal length and called a \emph{shortest strategy}\index{shortest strategy}.
With the following lemma we establish the fact that $k$ is determined
by $\varepsilon$, and that only for a finite number of annihilating
quasicones we cannot find a $k$ such that the shortest strategy $s_{k}$
succeeds.
\begin{lem}
\label{lem:proof for c_1 above lower bound} For $\varepsilon\in\left\{ 1,\dots,n\right\} $,
let $v\in V_{\mu-\varepsilon\delta}$ be an arbitrary vector, $V_{\mu}=\left\{ 0\right\} $
and the annihilating set of $v$ be a quasicone $\mathcal{A}nn\left(v\right)=C$
in normal form. Then there exists a number $k\in\left\{ 1,\dots,n-1\right\} $
such that the set $\mathcal{A}nn\left(e_{-\alpha_{1}+k\delta}\circ e_{\alpha_{1}}v\right)$
non-trivially contains a quasicone $C'$ satisfying $\#C'-\#C<0\,,$
if the number $c_{-1}=c_{1,0}$ is greater than $\left(n+1\right)\cdot\varepsilon$.
\end{lem}

\begin{proof}
The transforms are summarized schematically as
\begin{align*}
\left(\begin{array}{cccccc}
1 & 1 & c_{0,2} & c_{0,3} & \cdots & c_{0,n}\\
c_{1,0} & 1 & 1 & c_{1,3} & \cdots & c_{1,n}\\
c_{2,0} & c_{2,1} & 1\\
c_{3,0} & c_{3,1} &  & \ddots & \ddots\\
\vdots & \vdots &  & \ddots & 1 & 1\\
c_{n,0} & c_{n,1} &  &  & c_{n,n-1} & 1
\end{array}\right)_{-\varepsilon\delta} & \overset{e_{\alpha_{1}}}{\rightsquigarrow}\left(\begin{array}{cccccc}
1 & 1 & c_{0,2} & c_{0,3} & \cdots & c_{0,n}\\
\widehat{c_{1,0}} & 1 & c_{1,2}' & c_{1,3}' & \cdots & c_{1,n}'\\
c_{2,0}' & c_{2,1} & 1\\
c_{3,0}' & c_{3,1} &  & \ddots & \ddots\\
\vdots & \vdots &  & \ddots & 1 & 1\\
c_{n,0}' & c_{n,1} &  &  & c_{n,n-1} & 1
\end{array}\right)_{\alpha_{1}-\varepsilon\delta}
\end{align*}

\begin{align*}
\overset{e_{-\alpha_{1}+k\delta}}{\rightsquigarrow}\left(\begin{array}{cccccc}
1 & 1 & c_{0,2}'' & c_{0,3}'' & \cdots & c_{0,n}''\\
\widehat{c_{1,0}} & 1 & c_{1,2}' & c_{1,3}' & \cdots & c_{1,n}'\\
c_{2,0}' & c_{2,1}'' & 1\\
c_{3,0}' & c_{3,1}'' &  & \ddots & \ddots\\
\vdots & \vdots &  & \ddots & 1 & 1\\
c_{n,0}' & c_{n,1}'' &  &  & c_{n,n-1} & 1
\end{array}\right)_{-\delta}\\
\end{align*}
where $\widehat{c_{1,0}}=\min\left(\varepsilon,c_{1,0}\right)$, and
the last step determines $k=\widehat{c_{1,0}}-1=\min\left(\varepsilon-1,c_{1,0}-1\right)$.
Without loss of generality, there are no additional trivial actions
in the second step. For all $i=1,\dots,n$,
\begin{align*}
c_{1,i}' & =\begin{cases}
\min\left(\widehat{c_{1,0}}+c_{0,i},c_{0,i}\right) & \quad\mbox{ if }c_{1,i}\le c_{0,i}\\
\min\left(\widehat{c_{1,0}}+c_{0,i},c_{1,i}\right) & \quad\mbox{ if }c_{1,i}>c_{0,i}
\end{cases}\\
 & =\min\left(\widehat{c_{1,0}}+c_{0,i},\max\left(c_{0,i},c_{1,i}\right)\right)\\
c_{i,0}' & =\begin{cases}
\min\left(\widehat{c_{1,0}}+c_{i,1},c_{i,1}\right) & \quad\mbox{ if }c_{i,0}\le c_{i,1}\\
\min\left(\widehat{c_{1,0}}+c_{i,1},c_{i,0}\right) & \quad\mbox{ if }c_{i,0}>c_{i,1}
\end{cases}\\
 & =\min\left(\widehat{c_{1,0}}+c_{i,1},\max\left(c_{i,1},c_{i,0}\right)\right)\\
c_{0,i}'' & =\begin{cases}
\min\left(1+c_{1,i}',c_{0,i}\right) & \quad\mbox{ if }c_{0,i}\ge c_{1,i}-k\\
\min\left(1+c_{1,i}',c_{1,i}'-k\right) & \quad\mbox{ if }c_{0,i}<c_{1,i}-k
\end{cases}\\
 & =\min\left(1+c_{1,i}',\max\left(c_{0,i},c_{1,i}'-k\right)\right)\\
c_{i,1}'' & =\begin{cases}
\min\left(1+c_{i,0},c_{i,1}\right) & \quad\mbox{ if }c_{i,1}\ge c_{i,0}'-k\\
\min\left(1+c_{i,0},c_{i,0}'-k\right) & \quad\mbox{ if }c_{i,1}<c_{i,0}'-k
\end{cases}\\
 & =\min\left(1+c_{i,0},\max\left(c_{i,1},c_{i,0}'-k\right)\right)\,.
\end{align*}
First, $c_{1,0}-\widehat{c_{1,0}}=-\min\left(\varepsilon,c_{1,0}\right)+c_{1,0}=\max\left(c_{1,0}-\varepsilon,0\right)\,.$
Further $c_{1,i}'=\max\left(c_{0,i},c_{1,i}\right)$, $c_{i,0}'=\max\left(c_{i,1},c_{i,0}\right)$
and
\begin{align*}
c_{0,i}'' & =\min\left(1+\max\left(c_{0,i},c_{1,i}\right),\max\left(c_{0,i},\max\left(c_{0,i},c_{1,i}\right)-k\right)\right)\\
 & =\min\left(\max\left(c_{0,i}+k,c_{1,i}+k\right)-\left(k-1\right),\max\left(c_{0,i}+k,c_{1,i}\right)-k\right)\\
 & =\begin{cases}
c_{0,i} & \mbox{ if }c_{0,i}+k>c_{1,i}\\
c_{1,i}-k & \mbox{ else }
\end{cases}\\
 & =\max\left(c_{0,i},c_{1,i}-k\right)\\
c_{i,1}'' & =\min\left(1+c_{i,0},\max\left(c_{i,1},\max\left(c_{i,1},c_{i,0}\right)-k\right)\right)\\
 & =\begin{cases}
c_{i,0}-k & \mbox{ if }c_{i,0}-k>c_{i,1}\\
\min\left(1+c_{i,0},c_{i,1}\right) & \mbox{ if }c_{i,0}-k\le c_{i,1}
\end{cases}\\
 & =\max\left(c_{i,0}-k,\min\left(1+c_{i,0},c_{i,1}\right)\right)\,.
\end{align*}
 Now in each expression, $\left(i=1,\dots,n\right)$,

\begin{align*}
\left(c_{0,i}''+c_{1,i}'\right) & -\left(c_{0,i}+c_{1,i}\right)\\
 & =\max\left(-c_{0,i},-c_{1,i}\right)+\max\left(-c_{1,i},-c_{0,i}-k\right)\\
 & =-\min\left(2c_{1,i},c_{1,i}+c_{0,i},2c_{0,i}+k\right)\\
 & \le-\min\left(2c_{1,i},2c_{1,i}-1,2c_{1,i}+k-2\right)\mbox{, because }c_{0,i}+1\ge c_{1,i}\\
 & =-2c_{1,i}+\min\left(1,2-k\right)
\end{align*}
 and 
\begin{align*}
\left(c_{i,0}''+c_{i,1}'\right) & -\left(c_{i,0}+c_{i,1}\right)\\
 & =\max\left(-c_{i,1},-c_{i,0}\right)+\max\left(-c_{i,1}-k,\min\left(1-c_{i,1},-c_{i,0}\right)\right)\\
 & =-\min(2c_{i,1}+k,c_{i,1}+c_{i,0}+k,\max\left(2c_{i,1}-1,c_{i,0}+c_{i,1}\right),\\
 & \qquad\qquad\max\left(c_{i,1}+c_{i,0}-1,2c_{i,0}\right))\\
 & \le-\min(2\left(c_{i,0}-1\right)+k,2c_{i,0}-1+k,\\
 & \qquad\qquad\max\left(2\left(c_{i,0}-1\right)-1,2c_{i,0}-1\right),\max\left(2\left(c_{i,0}-1\right),2c_{i,0}\right))\\
 & \mbox{(because }c_{i,0}\le c_{i,1}+1)\\
 & =-\min(2c_{i,0}-2+k,2c_{i,0}-1,2c_{i,0})\\
 & =-2c_{i,0}+\max(2-k,1)\,.
\end{align*}
 Consequently, the sum 
\begin{align*}
\left(c_{0,i}''+c_{1,i}'\right) & -\left(c_{0,i}+c_{1,i}\right)+\left(c_{i,0}''+c_{i,1}'\right)-\left(c_{i,0}+c_{i,1}\right)\\
 & =-2c_{1,i}+\min\left(1,2-k\right)-2c_{i,0}+\max(2-k,1)\\
 & \le-2\left(\left(c_{0,i}-c_{1,0}\right)+c_{i,0}\right)+\min\left(1,2-k\right)+\max(1,2-k)\\
 & \mbox{(because }c_{0,i}\le c_{1,i}+c_{1,0})\\
 & \le-2\left(1+c_{1,0}-c_{1,0}\right)+3-k\mbox{ (because }C\mbox{ is normal)}\\
 & =1-k\,.
\end{align*}
 Therefore, the gap decreases by at least 
\begin{align*}
\#C-\#\left(s_{k}\left(C\right)\right) & \ge\max\left(c_{1,0}-\varepsilon,0\right)+n\cdot\left(1-k\right)\\
 & =\max\left(c_{1,0}-\varepsilon,0\right)+n\cdot\left(1-\min\left(\varepsilon,c_{1,0}\right)-1\right)\\
 & =c_{1,0}-\varepsilon-n\cdot\varepsilon\mbox{ if \ensuremath{c_{1,0}\ge n}}.
\end{align*}
So, the balance is greater than zero if $c_{-1}=c_{1,0}>\left(n+1\right)\cdot\varepsilon$,
as claimed.
\end{proof}
\begin{cor}
\label{cor:Finite}The number of annihilating quasicones for which
$s_{k}$ does not succeed for any $k\in\mathbb{N}$ is finite.
\end{cor}

\subsection{Monoid of strategies}

Recall the exponential root notation, $\alpha_{i}\mapsto2^{i-1}$,
$\left(i=1,\dots,n\right)$. Consider the strategy
\[
s=e_{-\left(2^{n}-1\right)+k_{n}\delta}\circ e_{\alpha_{2}+k_{n-1}\delta}\circ\cdots\circ e_{2^{2}+k_{2}\delta}\circ e_{2^{1}+k_{1}\delta}\circ e_{2^{0}}
\]
for $v\in V_{\lambda-\delta}$ with $\mathcal{A}nn\left(v\right)=C$
and $\mathbf{k}=\left(k_{0},k_{1},\dots k_{n-1},k_{n}\right)\in\mathbb{Z}^{n+1}$
recursively defined by
\begin{align*}
k_{0} & =0\\
k_{i} & =\left(e_{2^{i-1}+k_{i-1}\delta}\circ\cdots\circ e_{2^{1}+k_{1}\delta}\circ e_{2^{0}}\left(-\delta,C\right)\right)_{i,i+1}-1,\mbox{ for }i=1,\dots,n-1\\
k_{n} & =\left(e_{2^{n-1}+k_{n-1}\delta}\circ\cdots\circ e_{2^{1}+k_{1}\delta}\circ e_{2^{0}}\left(-\delta,C\right)\right)_{n,0}+r\left(\mathbf{k}\right)
\end{align*}
where $r\left(\mathbf{k}\right)$ may be $-1$ or such that $k_{1}+\cdots+k_{n}=0$
(this way it is granted for $\mathfrak{h}^{\circ}\otimes t$ to annihilate
$sv$). This $s$ is a circular strategy and we call it \label{def: shortest long strategy}
the \emph{shortest long strategy}.
\begin{prop}
\label{prop:Strategy sequence of roots}A strategy for $C$ can be
uniquely identified with a sequence of roots $\varphi_{0},\dots,\varphi_{m}\in\Delta^{\circ}$,
$m\ge0$.
\end{prop}

\begin{proof}
Let the strategy $s$ for $v\in V_{\lambda+\psi}$ with $\mathcal{A}nn\left(v\right)=C$
be given by
\[
s=e_{\varphi_{m}+k_{m}\delta}\circ\cdots\circ e_{\varphi_{1}+k_{1}\delta}\circ e_{\varphi_{0}+k_{0}\delta}
\]
with $\varphi_{0},\dots,\varphi_{m}\in\Delta^{\circ}\left(\Pi^{\circ}\right)$
and the $k_{0},\dots,k_{n}\in\mathbb{Z}$ recursively defined by
\begin{align*}
k_{0} & =C_{\varphi_{0}}-1\\
k_{\ell+1} & =\left(e_{\varphi_{\ell}+k_{\ell}\delta}\circ\cdots\circ e_{\varphi_{1}+k_{1}\delta}\circ e_{\varphi_{0}+k_{0}\delta}\left(\psi,C\right)\right)_{\varphi_{\ell}}-1,\mbox{ for }\ell=0,\dots,m-1.
\end{align*}
Ad hoc, $\left[e_{-\delta},e_{\varphi_{\ell}+k_{\ell}\delta}\right]$
acts trivially on $e_{\varphi_{\ell}+k_{\ell}\delta}\circ\cdots\circ e_{\varphi_{1}+k_{1}\delta}\circ e_{\varphi_{0}+k_{0}\delta}\left(\psi;C\right)$
for all $\ell=0,\dots,m$ and thus $\mathfrak{h}\otimes t\mathbb{C}\left[t\right]$
continues to do so. 
\end{proof}
Define a finite set of strategies for $\mathfrak{g}=A_{n}^{\left(1\right)}$
by iteratively taking the set of strategies for $A_{n-1}^{\left(1\right)}$
and certain strategies that comprise the corresponding root operators
for all basic roots $\alpha_{1},\dots,\alpha_{n}$. 
\begin{defn}
\label{def: simple_basic_strategy} A \emph{simple basic strateg}y
is defined as
\begin{align*}
s_{\mathbf{r}} & =\mathbf{e}_{-\mathbf{r}}\circ e_{2^{n-1}}\circ\cdots\circ e_{2^{0}}\,,\qquad\mathbf{e}_{-\mathbf{r}}=e_{-r_{k}}\circ\cdots\circ e_{-r_{1}}\:,
\end{align*}
where $\mathbf{r}=\left(r_{k},\dots,r_{1}\right)$ is a monotonously
ordered root partition of $\theta=\alpha_{1}+\cdots+\alpha_{n}$,
i.e. $\alpha_{r_{1}}+\cdots+\alpha_{r_{k}}=\theta$, $r_{1}<\cdots<r_{k}$
or $r_{1}>\cdots>r_{k}$, and $\alpha_{r_{1}},\dots,\alpha_{r_{k}}\in\Delta_{+}^{\circ}$. 
\end{defn}

By choice, all simple basic strategies are circular. The element corresponding
to the partion with the 'just vertical' Young tableau is $s_{0}=e_{-\left(2^{n}-1\right)}\circ e_{2^{n-1}}\circ e_{2^{n-2}}\cdots\circ e_{1}.$
\begin{example}
The simple basic strategies for $A_{2}^{\left(1\right)}$ are 
\[
\left\{ e_{-1}\circ e_{1},\,e_{-3}\circ,\,e_{2}\circ e_{1},\,e_{-2}\circ e_{-1}\circ e_{2}\circ e_{1},\,e_{-1}\circ e_{-2}\circ e_{2}\circ e_{1}\right\} .
\]
\end{example}

\begin{rem}
The number of simple basic strategies for $A_{n}$ is $F_{n}=\sum_{i=0}^{n-1}2^{i+1}-1$,
which we leave to the reader as an easy exercise.
\end{rem}

\textbf{Conjecture.} The strategy that is successful for any quasicone
lies in the monoid generated set of simple basic strategies. 

\textbf{Answer.} The conjecture is wrong. There are $A_{4}^{\left(1\right)}$-quasicones
\textendash{} listed in Section 5.5 \textendash , for which a successful
strategy cannot be obtained by concatenating simple basic strategies. 

The root graph is the graph of roots $\varphi\in\Delta^{\circ}$,
with the edges between each two roots $\varphi_{1},\varphi_{2}$ which
satisfy $\varphi_{1}-\varphi_{2}\in\Delta^{\circ}$. The centered
root graph $\underset{\cdot}{\Delta}^{\circ}$ is the root graph with
a center added and connected to all nodes of the root graph, i.e.
the cone graph. By \ref{prop:Strategy sequence of roots}, a strategy
is uniquely identified by an oriented path on the centered root graph
of $\mathfrak{g}$. A path may contain circles and $n$-cycles but
no loops. The corresponding path monoid is generated by circle-free
paths and circles. For $A_{n}$ , there are $\left(2n-1\right)^{\ell-1}$
paths of length $\ell$.

\subsection{General Approach}
\begin{lem}
\label{lem:Delta-induction} Let $v\in V_{\mu-\delta}$, $V_{\mu}=\left\{ 0\right\} $,
and $\mathcal{A}nn\left(v\right)$ contain a pre-prosolvable subalgebra
\textup{
\[
\mathcal{A}nn\left(v\right)\supset\left\{ \begin{array}{ccccc}
\emptyset & A_{0,1} & A_{0,2} & \cdots & A_{0,n}\\
A_{1,0} & \emptyset\\
A_{2,0} &  & \ddots & \ddots & \vdots\\
\vdots &  & \ddots & \mathbb{Z}_{\ge c+c'} & \mathbb{Z}_{\ge c}\\
A_{n,0} &  & \cdots & \mathbb{Z}_{\ge c'} & \mathbb{Z}_{\ge c+c'}
\end{array}\right\} 
\]
}with $\min\left(A_{i,j}\right)+\min\left(A_{j,i}\right)>0$ for all
$i\ne j\in\left\{ 0,\dots,n\right\} $. Then there exists a $k\in\mathbb{Z}_{+}$
and a vector $w\in V_{\mu-k\delta}$ such that $\mathcal{A}nn\left(w\right)$
contains $\mathfrak{h}\otimes t\mathbb{C}\left[t\right]$. 
\end{lem}

\begin{proof}
We aim to prove by induction on $n$. Provided the statement is proved
for $A_{n-1}^{\left(1\right)}$, it is also true for the submatrix/-algebra
with index set $\left\{ 1,\dots,n-1\right\} $ in $A_{n}^{\left(1\right)}$
by extending the action of $A_{n-1}^{\left(1\right)}$ canonically.
Induction beginning is $A_{1}^{\left(1\right)}$, which is well known
\cite{Futorny1996}. Assume, $\mathcal{A}nn\left(v\right)$ does not
contain $\check{\alpha}_{n}\otimes t\mathbb{C}\left[t\right]$, i.e.
$e_{k\delta}^{\left(\alpha_{n}\right)}v\ne0$ for some $k\in\left\{ 2,\dots,c+c'-1\right\} $,
choose the maximal of such $k$ and set $w=e_{k\delta}^{\left(\alpha_{n}\right)}v$.
Then $w\in V_{\mu+\left(k-1\right)\delta}$ and 
\[
e_{\left(1-k\right)\delta}^{\left(\alpha_{1}\right)},\dots,e_{\left(1-k\right)\delta}^{\left(\alpha_{n}\right)}\in\mathcal{A}nn\left(w\right).
\]
But then, running through all $\kappa,\eta\in\mathbb{N}$, the operators
\begin{align*}
\left[\left(\mathrm{ad}^{\kappa}\,e_{\delta}^{\left(\alpha_{n}\right)}\circ\mathrm{ad}^{\eta}\,e_{\left(1-k\right)\delta}^{\left(\alpha_{n}\right)}\right)\left(e_{\alpha_{n}+A_{n-1,n}\delta}\right),e_{-\alpha_{n}+A_{n,n-1}\delta}\right] & =e_{\left(\kappa+\eta\left(1-k\right)+A_{n-1,n}+A_{n,n-1}\right)\delta}^{\left(\alpha_{n}\right)}
\end{align*}
comprise $e_{l\delta}^{\left(\alpha_{n}\right)}$ for all integers
$l\in\mathbb{Z}$, since $1-k\le-1$. They all act trivially on $w$,
as desired.
\end{proof}
\begin{prop}
Let $V$ be a non-dense weight $\mathfrak{g}$-module, then there
exists a basis\linebreak{}
 $\Pi=\left\{ \alpha_{1},\dots,\alpha_{n},\delta\right\} \subset\Delta$
and a vector $v\in V$ that is primitive with respect to a quasicone
$C$ or to $\mathfrak{n}_{S}$ \textup{for some $S\subset\Pi^{\circ}$.}
\end{prop}

\begin{rem}
We enumerate according to the following paradigm: Write \textbf{1.}
if the respective operator $e_{\varphi}$ acts injectively on the
corresponding weight subspace, and \textbf{2.} if all of the operators
$\left\{ e_{\varphi+m\delta}\mid m\in\mathbb{Z}\right\} $ act trivially.
This will represent a binary tree. At the end, we have to check if
the resulting binary tree is complete, meaning that any leaf is one
of the subalgebras of the claimed type or such that an appropriate
lemma asserts the existence of a primitive element of the claimed
type. 
\end{rem}

\emph{Proof.} We proceed by induction on $n$. The induction start,
for $n=1$, we have two cases,

\begin{longtable}{lc|c}
 & \multicolumn{1}{c}{$\mathcal{A}nn\left(v_{\alpha}\right)\supset\left\{ \begin{array}{cc}
\emptyset & \emptyset\\
\left\{ k\right\}  & \emptyset
\end{array}\right\} _{\alpha-k\delta}$} & $\mathcal{A}nn\left(v_{-\delta}\right)\supset\left\{ \begin{array}{cc}
\left\{ 1\right\}  & \emptyset\\
\emptyset & \left\{ 1\right\} 
\end{array}\right\} _{-\delta}$\tabularnewline
 & \multicolumn{1}{c}{} & \tabularnewline
\cline{2-3} 
 &  & \tabularnewline
\textbf{1.} & $\overset{e_{-\alpha-m\delta}}{\rightsquigarrow}\left\{ \begin{array}{cc}
\left\{ m\right\}  & \emptyset\\
m\mathbb{Z}_{\ge0} & \left\{ m\right\} 
\end{array}\right\} _{-m\delta}$ & $\overset{e_{\alpha-\left(k-1\right)\delta}}{\rightsquigarrow}\left\{ \begin{array}{cc}
\emptyset & \emptyset\\
\left\{ m+1\right\}  & \emptyset
\end{array}\right\} _{\alpha-k\delta}$\tabularnewline
 &  & \tabularnewline
\textbf{1.1.} & $\overset{e_{-\alpha-m'\delta}}{\rightsquigarrow}\left\{ \begin{array}{cc}
* & \left\{ m+m'\right\} \\
m\mathbb{Z}_{\ge0} & *
\end{array}\right\} _{-\alpha-\left(m+m'\right)\delta}$ & %
\begin{minipage}[t]{0.42\columnwidth}%
and we can continue with the argument \textbf{1.} on the left side
of this table.%
\end{minipage}\tabularnewline
 &  & \tabularnewline
 & %
\begin{minipage}[t]{0.42\columnwidth}%
Now, Lemma \ref{lem:Delta-induction} applies, giving us a quasicone%
\end{minipage} & \tabularnewline
 &  & \tabularnewline
\textbf{2.} & $\left\{ \begin{array}{cc}
\emptyset & \emptyset\\
\mathbb{Z} & \emptyset
\end{array}\right\} _{\alpha}=\mathfrak{n}_{-\alpha}$ & $\left\{ \begin{array}{cc}
\left\{ 1\right\}  & \mathbb{Z}\\
\emptyset & \left\{ 1\right\} 
\end{array}\right\} _{-\delta}\supset\mathfrak{n}_{\alpha}$\tabularnewline
 & \multicolumn{1}{c}{} & \tabularnewline
\end{longtable}

This finishes the induction start.

By induction assumption, for $\mathfrak{g}\left(\Pi\smallsetminus\left\{ \alpha_{n}\right\} ,\delta\right)\cong A_{n-1}^{\left(1\right)}$
there is a root $\varphi\in\Delta\left(\Pi\smallsetminus\left\{ \alpha_{n}\right\} \right)$
and a vector in $v^{\left(n-1\right)}\in V_{\mu-\varphi}^{\left(n-1\right)}$
that is annihilated by a subalgebra equivalent to a quasicone $C$
(case \textbf{I}) or to $\mathfrak{n}_{S'}$ for some $S'\subset\Pi^{\circ}\smallsetminus\left\{ \alpha_{n}\right\} $(case
\textbf{II}) (cf. \ref{def: Levi and nilpotent} and \ref{eq:natural_nilpotent_subalgebra}).
By the non-density assumption, there exists a $\mu\in\mathfrak{h}^{*}\left(\Pi\smallsetminus\left\{ \alpha_{n}\right\} \right)$
with $\mu\notin\mathrm{supp}\left(V^{\left(n-1\right)}\right)$, i.e.
$\mathcal{U}\left(\mathfrak{g}\right)_{\varphi}v^{\left(n-1\right)}=0$,
if $v^{\left(n-1\right)}\in V_{\mu-\varphi}^{\left(n-1\right)}$. 

Now let $V$ be an irreducible $A_{n}^{\left(1\right)}$-module. The
natural $\mathfrak{g}$-module monomorphism $\iota:V^{\left(n-1\right)}\rightarrow V$
gives us an element $v\in V$ that contains as annihilator $\mathcal{A}nn\left(\iota\left(v^{\left(n-1\right)}\right)\right)$,
first of all, a quasicone (case \textbf{I}). Without restrictions
$v$ lies in the weight subspace $V_{\mu-\delta}$.%
\begin{comment}
{*}{*}{*}

The argument departing from here is that whenever we let an operator
act that does not commute with all of the $e_{*+\alpha_{n}+\mathbb{Z}\delta}$,
we won't \quotedblbase lose\textquotedblright{} any of those operators.
Precisely, 
\[
\mathrm{Com}\left(e_{\varphi}\right)\cap\left\{ e_{\nu}\mid\left[e_{\nu},\mathfrak{g}\right]\supset e_{\varphi}\right\} =\emptyset.
\]
This can be seen in the matrix notation. Commutants are always in
a different row and column, whereas constituents under a Lie bracket
are always in the same row or the same column. 

We have thus reduced this case to the corresponding problem in $A_{n-1}^{\left(1\right)}$
with the solution of type 
\[
\mathcal{N}_{\alpha_{1},\dots,\alpha_{n-1}}^{+}\oplus\mathfrak{n}_{\alpha_{n}}
\]
as the annihilating subalgebra of the desired primitive vector for
the corresponding mixed module. If on the other hand $e_{\alpha_{1}+*\delta}$
acts zero in any forthcoming step, then the assumptions of Lemma \ref{lem:close-delta-gapA2}
and Lemma\ref{lem:Delta-induction-A2} (ii) are satisfied, closing
the current case, either.
\end{comment}
\\
\\
\textbf{1.} $e_{\alpha_{n}+m\delta}v_{-\delta}\ne0$ for some $m\in\mathbb{Z}$

\noindent\begin{minipage}[t]{1\columnwidth}%
\begin{center}
$\overset{e_{\alpha_{n}+m\delta}}{\rightsquigarrow}\left\{ \begin{array}{ccccc}
\mathbb{Z}_{\ge2} & * & \cdots* & \emptyset & \emptyset\\
* & \ddots & \ddots & \vdots & \vdots\\
\vdots & \ddots & \mathbb{Z}_{\ge2} & \emptyset\\
* & \cdots & * & \emptyset & \emptyset\\
\mathbb{Z}_{\ge*-m} & \cdots & \mathbb{Z}_{\ge*-m} & \left\{ 1-m\right\}  & \emptyset
\end{array}\right\} _{\alpha_{n}+\left(m-1\right)\delta}$
\par\end{center}%
\end{minipage}

(The upper left $n\times\left(n-1\right)$-submatrix remains unchanged)\textbf{}\\
\textbf{1.1. }$e_{-\alpha_{n}+k\delta}\circ e_{-\alpha_{n}+j\delta}$
acts non-trivially for some $k<j<1-m$, without restrictions, $j=-m$
and $k=j-1$

\noindent\begin{minipage}[t]{1\columnwidth}%
\begin{center}
$\overset{e_{-\alpha_{n}-m\delta}^{2}}{\rightsquigarrow}\left\{ \begin{array}{ccccc}
\mathbb{Z}_{\ge2} & * & \cdots* & \mathbb{Z}_{\ge*} & \mathbb{Z}_{\ge-2m}\\
 & \ddots & \ddots & \vdots & \vdots\\
\vdots & \ddots & \mathbb{Z}_{\ge2} & \mathbb{Z}_{\ge*} & \mathbb{Z}_{\ge2m+*}\\
\mathbb{Z}_{\ge*} & \cdots\cdots & *\mathbb{Z}_{\ge*} & \mathbb{Z}_{\ge m+*} & \left\{ 2m+1\right\} \\
\mathbb{Z}_{\ge*-m} & \cdots & \mathbb{Z}_{\ge*-m} & \mathbb{Z}_{\ge*-m} & \mathbb{Z}_{\ge m+*}
\end{array}\right\} _{-\alpha_{n}-\left(2m+1\right)\delta}$
\par\end{center}%
\end{minipage}

Now, %
\begin{comment}
Lemma \ref{lem:close-delta-gapA2} and
\end{comment}
{} Lemma \ref{lem:Delta-induction} leads to the goal.\\
\\
\textbf{2.} (assuming $e_{\alpha_{n}+\mathbb{Z}\delta}v_{-\delta}=0$)

\noindent\begin{minipage}[t]{1\columnwidth}%
\begin{center}
$\left\{ \begin{array}{cccc}
 &  &  & \mathbb{Z}\\
 & * &  & \vdots\\
 &  &  & \mathbb{Z}\\
\emptyset & \cdots & \emptyset & \left\{ 1\right\} 
\end{array}\right\} _{-\delta}$
\par\end{center}%
\end{minipage}

This is a reduction to $A_{n-1}^{\left(1\right)}$, because $\mathfrak{n}_{\left\{ \alpha_{n}\right\} }^{\square}$
annihilates $v$.\textbf{}\\
\textbf{1.2.} (all $e_{-\alpha_{n}+\mathbb{Z}\delta}$ act zero)

\noindent\begin{minipage}[t]{1\columnwidth}%
\begin{center}
$\left\{ \begin{array}{ccccc}
\mathbb{Z}_{\ge2} & * & \cdots* & \emptyset & \emptyset\\
 & \ddots & \ddots & \vdots & \vdots\\
\vdots & \ddots & \mathbb{Z}_{\ge2} & \emptyset\\
\mathbb{Z}_{\ge*} & \cdots & \mathbb{Z}_{\ge*} & \emptyset & \emptyset\\
\mathbb{Z} & \cdots & \mathbb{Z} & \mathbb{Z} & \emptyset
\end{array}\right\} _{\alpha_{n}+\left(m-1\right)\delta}$
\par\end{center}%
\end{minipage}

This is a reduction to $A_{n-1}^{\left(1\right)}$ because $\mathfrak{n}_{-\left\{ \alpha_{n}\right\} }^{\square}$
annihilates $v$.
\begin{quote}
Before continuing with the main proof we will need an auxiliary lemma.
\end{quote}
\begin{lem}
\label{lem:last-step-main-proof}Let $\varphi\in\Delta^{\mathrm{re}}$
such that $\varphi+\alpha_{n-1}$ is not a root. If 
\[
\mathcal{A}nn\left(v_{\varphi}\right)\supset\left\{ \begin{array}{cccccc}
\ddots &  & \vdots & \vdots & \vdots & \vdots\\
\\
\\
\emptyset & \cdots & \emptyset & \emptyset\\
\mathbb{Z} & \cdots & \mathbb{Z} & \mathbb{Z} & \emptyset\\
\emptyset & \cdots & \emptyset & \emptyset & \emptyset & \emptyset
\end{array}\right\} _{\varphi}
\]
and $e_{-\varphi+k\delta}\notin\mathcal{A}nn\left(v_{\varphi}\right)$
for some $k\in\mathbb{Z}$, then 
\[
\mathcal{A}nn\left(v_{\varphi}\right)\supset\mathfrak{n}_{-\left\{ \alpha_{n}\right\} }^{\square}\text{ or }\mathcal{A}nn\left(v_{\varphi}\right)\supset s_{n}\mathfrak{n}_{-\left\{ \alpha_{n}\right\} }^{\square}.
\]
\end{lem}

\emph{Proof.} Since $e_{-\varphi-k\delta}v_{\varphi}\ne0$ for $k\in\mathbb{Z}$,
then\\
\textbf{1.}

\noindent\begin{minipage}[t]{1\columnwidth}%
\begin{center}
$\left\{ \begin{array}{cccccc}
\ddots &  & \vdots & \vdots & \vdots & \vdots\\
\\
\\
\emptyset & \cdots & \emptyset & \emptyset\\
\mathbb{Z} & \cdots & \mathbb{Z} & \mathbb{Z} & \emptyset\\
\emptyset & \cdots & \emptyset & \emptyset & \emptyset & \emptyset
\end{array}\right\} _{\varphi}\overset{e_{-\varphi-k\delta}}{\rightsquigarrow}\left\{ \begin{array}{cccccc}
\ddots &  & \vdots & \vdots & \vdots & \vdots\\
\\
 &  & \emptyset & \emptyset\\
\emptyset & \cdots & \emptyset & \left\{ k\right\}  & \emptyset\\
\mathbb{Z} & \cdots & \mathbb{Z} & \mathbb{Z} & \left\{ k\right\}  & \emptyset\\
\emptyset & \cdots & \emptyset & \emptyset & \emptyset & \left\{ k\right\} 
\end{array}\right\} _{-k\delta}$
\par\end{center}%
\end{minipage}

Since$\varphi+\alpha_{n-1}$ is not a root the subspaces $\mathfrak{g}_{-\left(\alpha_{i}+\cdots+\alpha_{n-1}\right)}$,
$\left(i\in\left\{ 1,\dots,n-1\right\} \right)$ commute with $e_{-\varphi-k\delta}$
and therefore annihilate $e_{-\varphi-k\delta}v$.\\
\textbf{1.1.}

\noindent\begin{minipage}[t]{1\columnwidth}%
\begin{center}
$\overset{e_{\alpha_{n}-k'\delta}}{\rightsquigarrow}\left\{ \begin{array}{cccccc}
\ddots &  & \vdots & \vdots & \vdots & \vdots\\
\\
 &  & \emptyset & \emptyset\\
\emptyset & \cdots & \emptyset & \left\{ k\right\}  & \emptyset\\
\mathbb{Z} & \cdots & \mathbb{Z} & \mathbb{Z} & \left\{ k\right\}  & \emptyset\\
\mathbb{Z} & \cdots & \mathbb{Z} & \mathbb{Z} & \left\{ k+k'\right\}  & \left\{ k\right\} 
\end{array}\right\} _{\alpha_{n}-\left(k+k'\right)\delta}$
\par\end{center}%
\end{minipage}

and $\mathfrak{n}_{-\left\{ \alpha_{n-1}\right\} }^{\square}$ annihilates
$v$ as claimed.\\
\textbf{1.2. }

\noindent\begin{minipage}[t]{1\columnwidth}%
\begin{center}
$\left\{ \begin{array}{cccccc}
\ddots &  & \vdots & \vdots & \vdots & \vdots\\
\\
 &  &  & \emptyset\\
\emptyset & \cdots & \emptyset & \left\{ k\right\}  & \emptyset & \emptyset\\
\mathbb{Z} & \cdots & \mathbb{Z} & \mathbb{Z} & \left\{ k\right\}  & \mathbb{Z}\\
\emptyset & \cdots & \emptyset & \emptyset & \emptyset & \left\{ k\right\} 
\end{array}\right\} _{-k\delta}\overset{s_{n}}{\sim}\left\{ \begin{array}{cccccc}
\ddots &  & \vdots & \vdots & \vdots & \vdots\\
\\
 &  &  & \emptyset\\
\emptyset & \cdots & \emptyset & \left\{ k\right\}  & \emptyset & \emptyset\\
\emptyset & \cdots & \emptyset & \emptyset & \left\{ k\right\}  & \emptyset\\
\mathbb{Z} & \cdots &  & \mathbb{Z} & \mathbb{Z} & \left\{ k\right\} 
\end{array}\right\} _{-k\delta}$
\par\end{center}%
\end{minipage}

and $\mathcal{A}nn\left(v_{\varphi}\right)\supset s_{n}\mathfrak{n}_{-\left\{ \alpha_{n}\right\} }^{\square}$
as claimed.\hfill{}$\square$

Now, let's turn our attention to case \textbf{II}. We aim to show
that $\mathcal{A}nn\left(v\right)$ contains $\mathfrak{n}_{\left\{ \pm\alpha_{k}\right\} }^{\square}$
for some $k\in\left\{ 0,\dots n\right\} $.

Without restrictions, we assume that $\mathcal{A}nn\left(v\right)$
contains $\mathfrak{n}_{-\left\{ \alpha_{k}\right\} }^{\square}\smallsetminus\mathfrak{n}_{-\left\{ \alpha_{n}\right\} }^{\square}$,
$\left(k\in\left\{ 0,\dots,n-1\right\} \right)$, but no other $\mathfrak{n}_{-\left\{ \alpha_{k'}\right\} }^{\square}$
with $k'>k$, \small 
\[
\left\{ \begin{array}{ccccccccc}
* & * & \emptyset &  & \overset{k}{} & \cdots &  & \emptyset\\
* & *\\
 &  & \emptyset\\
\vdots &  & \ddots & \ddots &  &  &  & \vdots\\
* &  & \cdots & * & \emptyset\\
\mathbb{Z} &  & \cdots &  & \mathbb{Z} & \emptyset\\
\vdots &  &  &  & \vdots & \vdots & \ddots\\
\mathbb{Z} &  & \cdots &  & \mathbb{Z} & \emptyset & \cdots & \emptyset\\
\emptyset &  &  &  & \cdots &  &  &  & \emptyset
\end{array}\right\} _{\alpha_{k}}.
\]
\normalsize 

This time we must assume that $v$ lies in the weight subspace $V_{\mu+\alpha_{k}}$
in order to be obtain a consistent induction step. Without loss of
generality, we can always assume the worst case that all unspecified
actions are non-zero ($\emptyset$). 

We start with the special case $n-k=1$ (induction start):
\[
\mathcal{A}nn\left(v\right)\supseteq\left\{ \begin{array}{ccccccc}
* & * & \emptyset &  & \cdots &  & \emptyset\\
* & * & \vdots\\
\vdots & \ddots & \emptyset\\
 &  &  & \ddots &  &  & \vdots\\
* & \cdots &  & *\\
\mathbb{Z} &  & \cdots & \mathbb{Z} & \mathbb{Z} & \emptyset\\
\emptyset &  & \cdots &  &  &  & \emptyset
\end{array}\right\} _{\alpha_{n-1}}
\]
\\
\textbf{1. }Assuming at first the worst case, that all unspecified
actions are non-zero. For some $m\in\mathbb{Z}$, 

\noindent\begin{minipage}[t]{1\columnwidth}%
\begin{center}
$\overset{e_{-\left(\alpha_{n-1}+\alpha_{n}\right)+m\delta}}{\rightsquigarrow}$\small $\left\{ \begin{array}{ccccccc}
* & * & \emptyset &  & \cdots &  & \emptyset\\
* & * & \vdots\\
\vdots & \ddots & \emptyset &  &  &  & \vdots\\
 &  &  & \ddots\\
\emptyset &  & \cdots & \emptyset & \emptyset & \emptyset & \emptyset\\
\mathbb{Z} &  & \cdots & \mathbb{Z} & \mathbb{Z} & \emptyset & \left\{ -m\right\} \\
\emptyset &  & \cdots &  &  & \emptyset & \emptyset
\end{array}\right\} _{-\alpha_{n}+m\delta}$ \normalsize 
\par\end{center}%
\end{minipage}

Now, Lemma \ref{lem:last-step-main-proof} with $\varphi=-\alpha_{n}$
applies to attain a $\mathfrak{n}_{-\left\{ \alpha\right\} }^{\square}$-primitive
vector.\\
\textbf{2.1.}

\noindent\begin{minipage}[t]{1\columnwidth}%
\begin{center}
\small $\left\{ \begin{array}{cccccccc}
* & * & \emptyset &  &  & \cdots &  & \emptyset\\
* & * & \vdots\\
\vdots & \ddots & * &  &  &  &  & \vdots\\
 &  &  & \ddots\\
\\
\emptyset &  &  & \cdots & \emptyset & \emptyset & \emptyset & \emptyset\\
\mathbb{Z} &  &  & \cdots & \mathbb{Z} & \mathbb{Z} & \emptyset\\
\emptyset &  &  & \cdots & \emptyset & \mathbb{Z} & \emptyset & \emptyset
\end{array}\right\} _{\alpha_{n-1}}$$\overset{e_{\alpha_{n}+*\delta}}{\rightsquigarrow}$\small $\left\{ \begin{array}{cccccccc}
* & * & \emptyset &  &  & \cdots &  & \emptyset\\
* & * & \vdots\\
\vdots & \ddots & * &  &  &  &  & \vdots\\
 &  &  & \ddots\\
\\
\emptyset &  &  & \cdots & \emptyset & \emptyset & \emptyset & \emptyset\\
\mathbb{Z} &  &  & \cdots & \mathbb{Z} & \mathbb{Z} & \emptyset\\
\emptyset &  &  & \cdots & \emptyset & \mathbb{Z} & \emptyset & \emptyset
\end{array}\right\} _{\alpha_{n-1}+\alpha_{n}+*\delta}$\normalsize 
\par\end{center}%
\end{minipage}\\
\textbf{2.1.1.}

\noindent\begin{minipage}[t]{1\columnwidth}%
\begin{center}
$\overset{e_{-\left(\alpha_{n-2}+\alpha_{n-1}+\alpha_{n}\right)+*\delta}}{\rightsquigarrow}$\small $\left\{ \begin{array}{cccccccc}
* & * & \emptyset &  &  & \cdots &  & \emptyset\\
* & * & \vdots\\
\vdots & \ddots & * &  &  &  &  & \vdots\\
 &  &  & \ddots\\
 &  &  &  & \emptyset & \left\{ -m\right\}  & \emptyset & \emptyset\\
\emptyset &  &  & \cdots & \emptyset & \emptyset & \emptyset & \emptyset\\
\mathbb{Z} &  &  & \cdots & \mathbb{Z} & \mathbb{Z} & \emptyset & \emptyset\\
\emptyset &  &  & \cdots & \emptyset & \mathbb{Z} & \emptyset & \emptyset
\end{array}\right\} _{-\alpha_{n-2}+m\delta}$\normalsize 
\par\end{center}%
\end{minipage}

and again, Lemma \ref{lem:last-step-main-proof} with $\varphi=-\alpha_{n-2}$
applies to attain a $\mathfrak{n}_{-\left\{ \alpha\right\} }^{\square}$-primitive
vector.\\
\textbf{2.2.}

\noindent\begin{minipage}[t]{1\columnwidth}%
\begin{center}
\small $\left\{ \begin{array}{ccccc}
\ddots &  & \vdots & \vdots & \vdots\\
 & \emptyset\\
\cdots & \emptyset & \emptyset & \emptyset & \emptyset\\
\cdots & \mathbb{Z} & \mathbb{Z} & \emptyset & \mathbb{Z}\\
\cdots & \emptyset & \mathbb{Z} & \emptyset & \emptyset
\end{array}\right\} _{\alpha_{n-1}}\overset{s_{n}}{\sim}\left\{ \begin{array}{ccccc}
\ddots &  & \vdots & \vdots & \vdots\\
 & \emptyset\\
\cdots & \emptyset & \emptyset & \emptyset & \emptyset\\
\cdots & \emptyset & \mathbb{Z} & \emptyset & \emptyset\\
\cdots & \mathbb{Z} & \mathbb{Z} & \mathbb{Z} & \emptyset
\end{array}\right\} _{\alpha_{n-1}+\alpha_{n}}\checkmark$\normalsize 
\par\end{center}%
\end{minipage}\\
\textbf{2.1.2.}

\noindent\begin{minipage}[t]{1\columnwidth}%
\begin{center}
\small $\left\{ \begin{array}{cccccc}
\ddots &  & \vdots & \vdots & \vdots & \vdots\\
\\
\\
\cdots & \emptyset & \emptyset & \emptyset\\
\cdots & \mathbb{Z} & \mathbb{Z} & \mathbb{Z} & \emptyset\\
\cdots & \emptyset & \mathbb{Z} & \mathbb{Z} & \emptyset & \emptyset
\end{array}\right\} _{\alpha_{n-1}+\alpha_{n}+*\delta}$\normalsize 
\par\end{center}%
\end{minipage}\\
\textbf{2.1.2.1. }Let $k\in\left\{ 1,\dots,n-3\right\} $ be the largest
number such that $e_{-\left(\alpha_{k}+\cdots+\alpha_{n}\right)+j\delta}$
does not act zero for some $j\in\mathbb{Z}$.

\noindent\begin{minipage}[t]{1\columnwidth}%
\begin{center}
$\overset{e_{-\left(\alpha_{k}+\cdots+\alpha_{n}\right)+*\delta}}{\rightsquigarrow}$\small $\left\{ \begin{array}{cccccccccc}
* & \emptyset &  &  & ^{k} & \cdots &  &  &  & \emptyset\\
\emptyset & \emptyset\\
 &  & \ddots &  &  &  &  &  &  & \vdots\\
 &  &  & \emptyset & \emptyset & \cdots & \emptyset & \left\{ m\right\}  & \emptyset & \emptyset\\
\vdots &  &  &  & \emptyset\\
 &  &  &  &  & \ddots\\
 &  &  &  &  &  &  &  &  & \vdots\\
\emptyset &  & \cdots & \emptyset &  & \cdots &  & \emptyset\\
\mathbb{Z} &  & \cdots & \mathbb{Z} &  & \cdots &  & \mathbb{Z} & \emptyset\\
\emptyset &  & \cdots & \emptyset & \mathbb{Z} & \cdots &  & \mathbb{Z} & \emptyset & \emptyset
\end{array}\right\} _{-\left(\alpha_{k}+\cdots+\alpha_{n-2}\right)-m\delta}$\normalsize 
\par\end{center}%
\end{minipage}

and Lemma \ref{lem:last-step-main-proof} with $\varphi=-\left(\alpha_{k}+\cdots+\alpha_{n-2}\right)$
applies to attain a $\mathfrak{n}_{-\left\{ \alpha\right\} }^{\square}$-primitive
vector.\\
\textbf{2.1.2.2.}

\noindent\begin{minipage}[t]{1\columnwidth}%
\begin{center}
\small $\left\{ \begin{array}{cccccc}
* & \emptyset &  & \cdots &  & \emptyset\\
\emptyset & \emptyset\\
\vdots &  & \ddots &  &  & \vdots\\
\emptyset &  & \cdots & \emptyset\\
\mathbb{Z} &  & \cdots & \mathbb{Z} & \emptyset & \emptyset\\
\mathbb{Z} &  & \cdots & \mathbb{Z} & \emptyset & \emptyset
\end{array}\right\} _{\alpha_{n-1}+\alpha_{n}+*\delta}\checkmark$\normalsize 
\par\end{center}%
\end{minipage}

We show that the hypothesis is also true for $n-k>1$, i.e when starting
with a zero action of $\mathfrak{n}_{-\left\{ \alpha_{k}\right\} }^{\square}\smallsetminus\mathfrak{n}_{-\left\{ \alpha_{n}\right\} }^{\square}$,
we can show that, by applying root operators, we can attain either
a $\mathfrak{n}_{-\left\{ \alpha_{k+1}\right\} }^{\square}\smallsetminus\mathfrak{n}_{-\left\{ \alpha_{k+2}\right\} }^{\square}$-primitive
element or a primitive element. This way, we reduce inductively to
the induction start where $n-k=1$, just above. The starting annihilating
subalgebra is\small 
\[
\left\{ \begin{array}{cccccccl}
* & \emptyset &  &  & ^{k} & \cdots &  & \emptyset\\
\emptyset & \emptyset &  &  &  &  &  & \vdots\\
\vdots & \ddots & \ddots & \ddots\\
\emptyset & \cdots & \emptyset & \emptyset & \emptyset\\
\mathbb{Z} & \cdots &  & \mathbb{Z} & \emptyset\\
\vdots &  &  & \vdots & \vdots\\
\mathbb{Z} & \cdots &  & \mathbb{Z} & \emptyset &  & \ddots & \vdots\\
\emptyset & \cdots &  & \emptyset & \emptyset &  & \cdots & \emptyset
\end{array}\right\} _{\alpha_{k}}.
\]
\normalsize \textbf{}\\
\textbf{1.}

\noindent\begin{minipage}[t]{1\columnwidth}%
\begin{center}
$\overset{e_{-\left(\alpha_{k}+\cdots+\alpha_{n}\right)-m\delta}}{\rightsquigarrow}$\small $\left\{ \begin{array}{cccccccc}
* & \emptyset &  &  & ^{k} & \cdots &  & \emptyset\\
\emptyset & \emptyset &  &  &  &  &  & \vdots\\
\vdots & \ddots & \ddots & \ddots\\
\emptyset & \cdots & \emptyset & \emptyset & \emptyset &  &  & \emptyset\\
\mathbb{Z} & \cdots &  & \mathbb{Z} & \emptyset &  &  & \left\{ m\right\} \\
\vdots &  &  & \vdots & \vdots &  &  & \emptyset\\
\mathbb{Z} & \cdots &  & \mathbb{Z} & \emptyset &  & \ddots & \vdots\\
\emptyset & \cdots &  & \emptyset & \emptyset &  & \cdots & \emptyset
\end{array}\right\} _{-\left(\alpha_{k+1}+\cdots+\alpha_{n}\right)-m\delta}$ \normalsize 
\par\end{center}%
\end{minipage}\\
\textbf{1.1.}

\noindent\begin{minipage}[t]{1\columnwidth}%
\begin{center}
$\overset{e_{\left(\alpha_{k+1}+\cdots+\alpha_{n}\right)+*\delta}}{\rightsquigarrow}$\small $\left\{ \begin{array}{cccccccc}
* & \emptyset &  &  & ^{k} & \cdots &  & \emptyset\\
\emptyset & * &  &  &  &  &  & \vdots\\
\vdots & \ddots & \ddots & \ddots\\
\emptyset & \cdots & \emptyset & * & \emptyset\\
\mathbb{Z} & \cdots &  & \mathbb{Z} & * &  &  & \left\{ m\right\} \\
\vdots &  &  & \vdots & \vdots\\
\mathbb{Z} & \cdots &  & \mathbb{Z} & \emptyset &  & \ddots & \vdots\\
\emptyset & \cdots &  & \emptyset & \emptyset &  & \cdots & *
\end{array}\right\} _{*\delta}$\normalsize 
\par\end{center}%
\end{minipage}\\
\textbf{1.1.1.}

\noindent\begin{minipage}[t]{1\columnwidth}%
\begin{center}
$\overset{e_{\alpha_{n}+*\delta}}{\rightsquigarrow}$\small $\left\{ \begin{array}{cccccccc}
* & \emptyset &  &  & ^{k} & \cdots &  & \emptyset\\
\emptyset & \emptyset &  &  &  &  &  & \vdots\\
\vdots & \ddots & \ddots & \ddots\\
\emptyset & \cdots & \emptyset & \emptyset & \emptyset &  &  & \emptyset\\
\mathbb{Z} & \cdots &  & \mathbb{Z} & \emptyset &  &  & \left\{ m\right\} \\
\vdots &  &  & \vdots & \vdots & \ddots &  & \emptyset\\
\mathbb{Z} & \cdots &  & \mathbb{Z} & \emptyset &  & \emptyset & \vdots\\
\mathbb{Z} & \cdots &  & \mathbb{Z} & \emptyset &  & \left\{ m'\right\}  & \emptyset
\end{array}\right\} _{\alpha_{n}-m'\delta}$\normalsize 
\par\end{center}%
\end{minipage}

contains $\mathfrak{n}_{-\left\{ \alpha_{k}\right\} }^{\square}$
as claimed.\\
\textbf{2.}

\noindent\begin{minipage}[t]{1\columnwidth}%
\begin{center}
\small $\left\{ \begin{array}{cccccccc}
* & \emptyset &  & ^{k} & \cdots &  &  & \emptyset\\
\emptyset & \emptyset &  &  &  &  &  & \vdots\\
\vdots & \ddots & \ddots\\
\emptyset & \cdots & \emptyset & \emptyset & \ddots\\
\mathbb{Z} & \cdots &  & \mathbb{Z} & \emptyset\\
\vdots &  &  & \vdots & \vdots & \ddots\\
\mathbb{Z} & \cdots &  & \mathbb{Z} & \emptyset & \cdots & \emptyset & \vdots\\
\emptyset & \cdots & \emptyset & \mathbb{Z} & \emptyset &  & \cdots & \emptyset
\end{array}\right\} _{\alpha_{k}}$\normalsize 
\par\end{center}%
\end{minipage}\\
\textbf{2.1. }Now choose the largest $p\in\left\{ 1,\dots,k-1\right\} $,
such that $e_{-\left(\alpha_{p}+\cdots+\alpha_{k}\right)-m\delta}$
does not act trivially for some $m\in\mathbb{Z}$.

\noindent\begin{minipage}[t]{1\columnwidth}%
\begin{center}
$\overset{e_{-\left(\alpha_{p}+\cdots+\alpha_{k}\right)-m\delta}}{\rightsquigarrow}$\small $\left\{ \begin{array}{ccccccccc}
* & \emptyset &  &  &  & ^{k} & \cdots &  & \emptyset\\
\emptyset & \emptyset &  &  &  &  &  &  & \vdots\\
\vdots & \ddots & \ddots\\
 &  &  & \emptyset & \left\{ m\right\} \\
\emptyset &  & \cdots & \emptyset & \emptyset & \emptyset\\
\mathbb{Z} &  & \cdots &  & \mathbb{Z} & \emptyset\\
\vdots &  &  &  & \vdots & \vdots\\
\mathbb{Z} &  & \cdots &  & \mathbb{Z} & \emptyset &  & \ddots & \vdots\\
\emptyset &  & \cdots & \emptyset & \mathbb{Z} & \emptyset &  & \cdots & \emptyset
\end{array}\right\} _{-\left(\alpha_{p}+\cdots+\alpha_{k-1}\right)-m\delta}$\normalsize 
\par\end{center}%
\end{minipage}\\
\textbf{2.1.1.}

\noindent\begin{minipage}[t]{1\columnwidth}%
\begin{center}
$\overset{e_{\alpha_{p}+\cdots+\alpha_{k-1}+*\delta}}{\rightsquigarrow}$\small $\left\{ \begin{array}{ccccc|cccc}
* & \emptyset &  &  & \cdots & ^{k} &  &  & \emptyset\\
\emptyset & \emptyset &  &  &  &  &  &  & \vdots\\
\vdots & \ddots & \ddots &  & \\
 &  &  & \emptyset & \left\{ m\right\} \\
\emptyset &  & \cdots & \emptyset & \emptyset & \emptyset\\
\hline \mathbb{Z} &  & \cdots &  & \mathbb{Z} & \emptyset\\
\vdots &  &  &  & \vdots & \vdots\\
\mathbb{Z} &  & \cdots &  & \mathbb{Z} & \emptyset &  & \ddots & \vdots\\
\emptyset &  & \cdots & \emptyset & \mathbb{Z} & \emptyset &  & \cdots & \emptyset
\end{array}\right\} _{*\delta}$\normalsize 
\par\end{center}%
\end{minipage}

Now by induction assumption the lower right matrix is either a quasicone,
with what $C_{n,n-1}\ne\emptyset$ and therefore $C_{n,q}=\mathbb{Z}$
for $q\in\left\{ 1,\dots,k-1\right\} $, as desired, or it contains
a subalgebra equivalent to $\mathfrak{n}_{-\left\{ \alpha_{r}\right\} }^{\square}\subset A_{n-k}^{\left(1\right)}$
for $r\in\left\{ k,\dots,n\right\} $ which completes to $\mathfrak{n}_{-\left\{ \alpha_{r}\right\} }^{\square}\smallsetminus\mathfrak{n}_{-\left\{ \alpha_{n}\right\} }^{\square}\subset A_{n}^{\left(1\right)}$,
proving the induction hypothesis.\\
\textbf{1.2.}

\noindent\begin{minipage}[t]{1\columnwidth}%
\begin{center}
\small $\left\{ \begin{array}{cccccccc}
* & \emptyset &  &  & ^{k} & \cdots &  & \emptyset\\
\emptyset & \emptyset &  &  &  &  &  & \vdots\\
\vdots & \ddots & \ddots & \ddots\\
\emptyset & \cdots & \emptyset & \emptyset & \emptyset &  &  & \emptyset\\
\mathbb{Z} & \cdots &  & \mathbb{Z} & \emptyset &  &  & \mathbb{Z}\\
\vdots &  &  & \vdots & \vdots &  &  & \emptyset\\
\mathbb{Z} & \cdots &  & \mathbb{Z} & \emptyset &  & \ddots & \vdots\\
\emptyset & \cdots &  & \emptyset & \emptyset &  & \cdots & \emptyset
\end{array}\right\} \sim\left\{ \begin{array}{ccccccccl}
* & \emptyset &  &  & ^{k} & \cdots &  &  & \emptyset\\
\emptyset & \emptyset &  &  &  &  &  &  & \vdots\\
\vdots & \ddots & \ddots & \ddots\\
 &  &  &  & \emptyset\\
\emptyset & \cdots &  & \emptyset & \emptyset\\
\mathbb{Z} & \cdots &  & \mathbb{Z} & \mathbb{Z}\\
\vdots &  &  & \vdots & \emptyset\\
\mathbb{Z} & \cdots &  & \mathbb{Z} & \emptyset &  &  & \ddots & \vdots\\
\emptyset & \cdots &  & \emptyset & \emptyset &  &  & \cdots & \emptyset
\end{array}\right\} $\normalsize 
\par\end{center}%
\end{minipage}

contains $\mathfrak{n}_{-\left\{ \alpha_{k+1}\right\} }^{\square}\smallsetminus\mathfrak{n}_{-\left\{ \alpha_{n}\right\} }^{\square}$
as claimed.\\
\textbf{1.1.2.}

\noindent\begin{minipage}[t]{1\columnwidth}%
\begin{center}
\small $\left\{ \begin{array}{ccccccccl}
* & \emptyset &  &  & ^{k} & \cdots &  &  & \emptyset\\
\emptyset & \emptyset &  &  &  &  &  &  & \vdots\\
\vdots & \ddots & \ddots & \ddots\\
\emptyset & \cdots & \emptyset & \emptyset & \emptyset\\
\mathbb{Z} & \cdots &  & \mathbb{Z} & \emptyset &  &  & \left\{ m\right\} \\
\vdots &  &  & \vdots & \vdots & \ddots &  &  & \vdots\\
\mathbb{Z} & \cdots &  & \mathbb{Z} & \emptyset &  & \emptyset & \mathbb{Z}\\
\mathbb{Z} & \cdots &  & \mathbb{Z} & \emptyset &  &  & \emptyset\\
\emptyset & \cdots &  & \emptyset &  &  & \cdots &  & \emptyset
\end{array}\right\} _{*\delta}\sim\left\{ \begin{array}{ccccccccl}
* & \emptyset &  &  & ^{k} & \cdots &  &  & \emptyset\\
\emptyset & \emptyset &  &  &  &  &  &  & \vdots\\
\vdots & \ddots & \ddots & \ddots\\
 &  &  &  & \emptyset\\
\emptyset & \cdots &  & \emptyset & \emptyset\\
\mathbb{Z} & \cdots &  & \mathbb{Z} & \mathbb{Z}\\
\vdots &  &  & \vdots & *\\
\mathbb{Z} & \cdots &  & \mathbb{Z} & \emptyset &  &  & \ddots & \vdots\\
\emptyset & \cdots &  & \emptyset & \emptyset &  &  & \cdots & \emptyset
\end{array}\right\} $\normalsize 
\par\end{center}%
\end{minipage}

also contains $\mathfrak{n}_{-\left\{ \alpha_{k+1}\right\} }^{\square}\smallsetminus\mathfrak{n}_{-\left\{ \alpha_{n}\right\} }^{\square}$.\\
\textbf{2.1.2.}

\noindent\begin{minipage}[t]{1\columnwidth}%
\begin{center}
\small $\left\{ \begin{array}{ccccccccl}
* & \emptyset &  &  & \cdots & ^{k} &  &  & \emptyset\\
\emptyset & \emptyset &  &  &  &  &  &  & \vdots\\
\vdots & \ddots & \ddots\\
 &  &  & \emptyset & \mathbb{Z}\\
\emptyset &  & \cdots & \emptyset & \emptyset & \emptyset\\
\mathbb{Z} &  & \cdots &  & \mathbb{Z} & \emptyset\\
\vdots &  &  &  & \vdots & \vdots\\
\mathbb{Z} &  & \cdots &  & \mathbb{Z} & \emptyset &  & \ddots & \vdots\\
\emptyset &  & \cdots & \emptyset & \mathbb{Z} & \emptyset &  & \cdots & \emptyset
\end{array}\right\} _{-\left(\alpha_{p}+\cdots+\alpha_{k-1}\right)-m\delta}$\normalsize 
\par\end{center}%
\end{minipage}\\
\textbf{2.1.2.1.}

\noindent\begin{minipage}[t]{1\columnwidth}%
\begin{center}
$\overset{e_{-\left(\alpha_{q}+\cdots+\alpha_{p-1}\right)+*\delta}}{\rightsquigarrow}$\small $\left\{ \begin{array}{ccccccccl}
* & \emptyset &  &  & \cdots & ^{k} &  &  & \emptyset\\
\emptyset & \emptyset &  &  & \left\{ m'\right\}  &  &  &  & \vdots\\
\vdots & \ddots & \ddots\\
 &  &  & \emptyset & \mathbb{Z}\\
\emptyset &  & \cdots & \emptyset & \emptyset & \emptyset\\
\mathbb{Z} &  & \cdots &  & \mathbb{Z} & \emptyset\\
\vdots &  &  &  & \vdots & \vdots\\
\mathbb{Z} &  & \cdots &  & \mathbb{Z} & \emptyset &  & \ddots & \vdots\\
\emptyset &  & \cdots & \emptyset & \mathbb{Z} & \emptyset &  & \cdots & \emptyset
\end{array}\right\} _{-\left(\alpha_{q}+\cdots+\alpha_{k-1}\right)-m'\delta}$\normalsize 
\par\end{center}%
\end{minipage}

Now we can apply $e_{\alpha_{q}+\cdots+\alpha_{k-1}+*\delta}$ and
be in the situation of \textbf{2.1.1.}, which was solved, or we are
in a situation analog to \textbf{2.1.2.} but with a $p'<p$. Thus
after finitely many steps, we arrive at a matrix equivalent to $\mathfrak{n}_{-\left\{ \alpha_{1}\right\} }^{\square}$,

\noindent\begin{minipage}[t]{1\columnwidth}%
\begin{center}
\small $\sim\left\{ \begin{array}{ccccccccl}
* & \emptyset &  &  & \cdots & ^{k} &  &  & \emptyset\\
\mathbb{Z} & \emptyset &  &  &  &  &  &  & \vdots\\
\vdots & \ddots & \ddots\\
 &  &  & \emptyset\\
\mathbb{Z} &  & \cdots & \emptyset & \emptyset & \emptyset\\
\mathbb{Z} &  & \cdots &  & \mathbb{Z} & \emptyset\\
\vdots &  &  &  & \vdots & \vdots\\
\mathbb{Z} &  & \cdots &  & \mathbb{Z} & \emptyset &  & \ddots & \vdots\\
\mathbb{Z} & \emptyset & \cdots &  &  & \emptyset &  & \cdots & \emptyset
\end{array}\right\} $ \normalsize 
\par\end{center}%
\end{minipage}

This proves the induction hypothesis and therefore the proposition.\hfill{}$\square$

\subsection{Main Algorithm}

The algorithmic part of the proof is structured in form of a search
forest. Let $S$ be the set of simple basic strategies. Applying strategies
to a set of quasicones, after $N$ iterations we obtain an $\left|S\right|$-ary
forest with $n$ trees of depth $N$. The image of the set of critical
normal quasicones after $N$ steps is a proper subset $S^{N}\left\{ C_{i}\right\} _{i=1,\dots,n}\cap\left\{ C_{i}\right\} _{i=1,\dots,n}\subsetneq\left\{ C_{i}\right\} _{i=1,\dots,n}$
but might become stationary $S^{N}\left\{ C_{i}\right\} _{i=1,\dots,n}=\left\{ C_{i_{k}}\right\} _{k=1,\dots,n'}$,
$\left(n'<n\right)$, at some point. Assume there is an element $C_{j}\in S^{N}\left\{ C_{i}\right\} _{i=1,\dots,n}\cap\left\{ C_{i}\right\} _{i=1,\dots,n}$.
Then $S^{2N}C_{j}\subset S^{N}\left\{ C_{i}\right\} _{i=1,\dots,n}$
so that applying strategies to $C_{j}$ is obsolete because there
is nothing new to attain. The following algorithm takes account to
these cases. 

\noindent \begin{center}
\setlength{\arrayrulewidth}{1pt}
\begin{longtable}{l>{\raggedright}p{0.5cm}|>{\raggedright}p{0.5cm}|>{\raggedright}p{0.5cm}|>{\raggedright}p{0.5cm}|>{\raggedright}p{0.36\columnwidth}>{\raggedright}p{0.2\columnwidth}}
\hline 
\noalign{\vskip2pt}
\multicolumn{7}{l}{\textbf{Algorithm. }\textsc{\label{lst:ConcatenateStrategies} Concatenate
Strategies}}\tabularnewline
\hline 
\noalign{\vskip2pt}
 & \multicolumn{2}{l}{\textbf{input}} & \multicolumn{4}{l}{\textbf{ $\left\{ C_{1},\dots,C_{n}\right\} $ }list quasicones with
unknown successfull strategy}\tabularnewline
\noalign{\vskip2pt}
 & \multicolumn{1}{>{\raggedright}p{0.5cm}}{} & \multicolumn{1}{>{\raggedright}p{0.5cm}}{} & \multicolumn{4}{l}{\ \ \  $\left\{ s_{1},\dots,s_{k}\right\} $ set of simple basic
strategies}\tabularnewline
\noalign{\vskip2pt}
 & \multicolumn{2}{l}{\textbf{output}} & \multicolumn{4}{l}{ list of quasicones where no concatenation of strategy was successful }\tabularnewline
\noalign{\vskip2pt}
 & \multicolumn{2}{l}{\textbf{tools}} & \multicolumn{4}{l}{%
\begin{minipage}[t]{0.75\columnwidth}%
- dictionary $\mathrm{dict}$, i.e. a list of key-value pairs $\left\{ key:value\right\} $;
\\
where $value$ is a list of tree indices and a $\mathrm{tree\_index}$
is the tuple that codifies the position in the quasicone-step tree;
\\
$key$ will be the index $i$ of the quasicone matrix $C_{i}$, $\left(i=1,\dots,n\right)$

- functions\texttt{ Apply\_strategy() }and\texttt{ Weyl\_normal\_form()} 

- method \texttt{successful}() that returns \texttt{True} if the gap
was reduced or a GVM-complete quasicone was achieved%
\end{minipage}}\tabularnewline
 & \multicolumn{1}{>{\raggedright}p{0.5cm}}{} & \multicolumn{1}{>{\raggedright}p{0.5cm}}{} & \multicolumn{1}{>{\raggedright}p{0.5cm}}{} & \multicolumn{1}{>{\raggedright}p{0.5cm}}{} &  & \tabularnewline
 & \multicolumn{6}{l}{\texttt{\# initialization}}\tabularnewline
 & \multicolumn{5}{l}{$\mathrm{dict}\leftarrow\left\{ 1:\left[\left(1\right)\right],\dots,1:\left[\left(n\right)\right]\right\} $} & \tabularnewline
 & \multicolumn{6}{l}{\# initialized with lists only containing one tuple, a one-tuple (start
index)}\tabularnewline
 & \multicolumn{5}{l}{$\lambda\leftarrow-\delta$} & \# startweight\tabularnewline
 & \multicolumn{5}{l}{\textbf{$\mathrm{list\_at\_start}\leftarrow\left[C_{1},\dots,C_{n}\right]$}} & \tabularnewline
 & \multicolumn{5}{l}{$\mathrm{list\_of\_successful}\leftarrow\left[\:\right]$} & \tabularnewline
 & \multicolumn{6}{l}{$\mathrm{strategy\_list}\leftarrow\left\{ s_{1},\dots,s_{k}\right\} $}\tabularnewline
 & \multicolumn{1}{>{\raggedright}p{0.5cm}}{} & \multicolumn{1}{>{\raggedright}p{0.5cm}}{} & \multicolumn{1}{>{\raggedright}p{0.5cm}}{} & \multicolumn{1}{>{\raggedright}p{0.5cm}}{} &  & \tabularnewline
 & \multicolumn{6}{l}{\texttt{\textbf{def }}\texttt{Step()}\texttt{\textbf{: }}}\tabularnewline
 &  & \multicolumn{4}{l}{\texttt{\textbf{foreach}} $C_{i}$ \texttt{\textbf{in}} $\mathrm{list\_at\_start}$\texttt{\textbf{:}} } & \tabularnewline
 &  &  & \multicolumn{3}{l}{\texttt{\textbf{foreach}} $s_{j}$ \texttt{\textbf{in}} $\mathrm{strategy\_list}$\texttt{\textbf{:}} } & \tabularnewline
 &  &  &  & \multicolumn{3}{l}{$\mathrm{step}$ \texttt{\textbf{$\leftarrow$}} \texttt{Apply\_strategy}($\lambda$,
$C$, $s$) }\tabularnewline
 &  &  &  & \multicolumn{3}{l}{$C'$ $\leftarrow$\texttt{ Weyl\_normal\_form}($\mathrm{step}.C$)}\tabularnewline
 &  &  &  & \multicolumn{3}{l}{\texttt{\textbf{if}} $\mathrm{step}$.\texttt{successful}() \texttt{\textbf{or}}
($C'$ \texttt{\textbf{in}} $\mathrm{list\_of\_successful}$)\texttt{\textbf{:}}}\tabularnewline
 &  &  &  &  & \multicolumn{2}{l}{$\mathrm{list\_of\_successful}$.\texttt{append}( }\tabularnewline
 &  &  &  &  & \multicolumn{2}{l}{\hspace*{0.5cm}$C_{i}$ \texttt{if} $\mathrm{dict}$$\left[i\right]\in\mathrm{step}.\mathrm{path\_to\_successful}$ }\tabularnewline
 &  &  &  &  & \hspace*{0.5cm}\texttt{for} $i\in\left\{ 1,\dots,n\right\} $) & \tabularnewline
 &  &  &  &  & \texttt{\textbf{break}} & \tabularnewline
 &  &  &  & \multicolumn{2}{l}{\texttt{\textbf{foreach}} $\mathrm{tree\_index}$ \texttt{\textbf{in}}
$\mathrm{dict}\left[i\right]$\texttt{\textbf{:}} } & \# add new index:\tabularnewline
 &  &  &  &  & \multicolumn{2}{l}{$k\leftarrow$list index of $C'$}\tabularnewline
 &  &  &  &  & $\mathrm{dict}\left[k\right]\leftarrow$$\mathrm{dict}\left[k\right]$.\texttt{append}($\mathrm{tree\_index}$.\texttt{append}($j$)) & \tabularnewline
 &  & \multicolumn{4}{l}{\texttt{\textbf{return}}} & \tabularnewline
 & \multicolumn{1}{>{\raggedright}p{0.5cm}}{} & \multicolumn{1}{>{\raggedright}p{0.5cm}}{} & \multicolumn{1}{>{\raggedright}p{0.5cm}}{} & \multicolumn{1}{>{\raggedright}p{0.5cm}}{} &  & \tabularnewline
 & \multicolumn{1}{>{\raggedright}p{0.5cm}}{} & \multicolumn{1}{>{\raggedright}p{0.5cm}}{} & \multicolumn{1}{>{\raggedright}p{0.5cm}}{} & \multicolumn{1}{>{\raggedright}p{0.5cm}}{} &  & \tabularnewline
 & \multicolumn{5}{l}{\texttt{\# main routine}} & \tabularnewline
 & \multicolumn{5}{l}{$\mathrm{old\_list\_of\_successful}\leftarrow\mathrm{list\_of\_successful}$} & \tabularnewline
 & \multicolumn{5}{l}{\texttt{\textbf{while }}$\mathrm{list\_of\_successful}$ \texttt{\textbf{!=
$\mathrm{list\_at\_start}$ :}}} & \tabularnewline
 &  & \multicolumn{4}{l}{\texttt{Step()}} & \tabularnewline
 &  & \multicolumn{5}{l}{\texttt{\textbf{if }}$\mathrm{old\_list\_of\_successful}$ \texttt{\textbf{==}}
$\mathrm{list\_of\_successful}$ \texttt{\textbf{: }}}\tabularnewline
 &  &  & \multicolumn{3}{l}{\texttt{\textbf{print }}(\texttt{\textbf{$\mathrm{list\_at\_start}$
$\smallsetminus$ }}$\mathrm{list\_of\_successful}$)} & \tabularnewline
 &  &  & \multicolumn{3}{l}{\texttt{\textbf{break}}} & \tabularnewline
 & \multicolumn{6}{l}{\texttt{\textbf{else: print }}\texttt{'successful strategies for all
quasicones found'}}\tabularnewline
 & \multicolumn{1}{>{\raggedright}p{0.5cm}}{} & \multicolumn{1}{>{\raggedright}p{0.5cm}}{} & \multicolumn{1}{>{\raggedright}p{0.5cm}}{} & \multicolumn{1}{>{\raggedright}p{0.5cm}}{} &  & \tabularnewline
\hline 
\end{longtable}
\par\end{center}

\setlength{\arrayrulewidth}{1pt} 

The implementation of the functions\texttt{ Apply\_strategy() }and\texttt{
Weyl\_normal\_form() }and the method\texttt{ successful() }is straightforward
from the definitions in the paper\footnote{full code available on http://github.com/thoma5B/Strategies-for-support-quasicones-of-affine-Lie-algebras-of-type-A}.

For certain anti-matroidal structures, greedy algorithms are unfeasible.
The anti-matroidal structure of a quasicone subalgebra lattice $\mathfrak{C}$
suggests that the problem of finding successful strategies is a \emph{constraint
satisfaction problem} (CSP) but with rather ``non-holonomic'' constraints
on the phase space $\mathfrak{C}\times\left(\Delta\cup\left\{ 0\right\} \right)$,
which is called \emph{arc consistency} or \emph{path consistency}
in this context \cite[pp. 121--137]{Wallace1995}. This class of problems
seems to be at least as complex as problems that can only be solved
by integer linear programming.

\begin{figure}[H]
\noindent \begin{centering}
\includegraphics[width=0.6\textwidth]{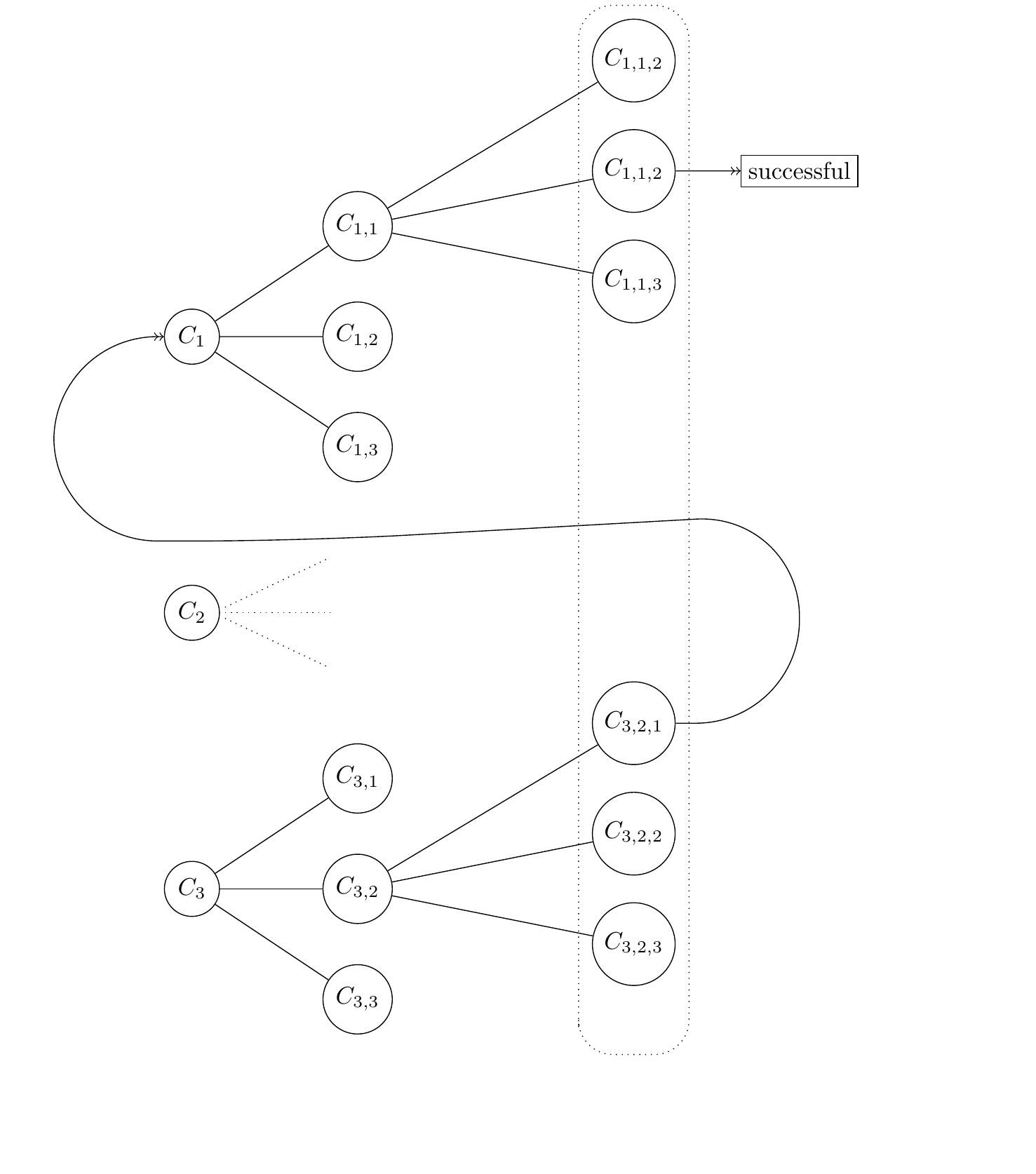}
\par\end{centering}
\caption{Illustration of the algoritm: For $C_{1}$ a succesfull solution has
been fund in the second step. Therefore all the quasicones $C_{1[,1[,2]]}$
are added to the list of solved cases. For $C_{3}$ the iteration
$C_{3,2,1}$ turns out to be in that list and thus all quasicones
$C_{3[,2[,1]]}$ can also be considered solved.}
\end{figure}

\subsection{The Kac-Moody algebras $A_{2}^{\left(1\right)}$, $A_{3}^{\left(1\right)}$
and $A_{4}^{\left(1\right)}$}

In Table 3, the compuational results for $A_{2}^{\left(1\right)}$,
$A_{3}^{\left(1\right)}$ and $A_{4}^{\left(1\right)}$ are summararized.
For $A_{3}^{\left(1\right)}$, almost all quasicones admit defect
reduction or even yield defect $0$ complete quasicones when applying
simple basic strategies (cf. Definition \ref{def: simple_basic_strategy}).
Only 8 quasicones $C_{1},\dots,C_{8}$ (having defects $2$ and $3$)
admit no defect reduction. But this set is not stabilized by defect-invariant
strategies: There exists at least one simple basic strategy such that
$\#C_{i}-\#sC_{i}=0$, for each $i\in\left\{ 1,\dots,8\right\} $.
For those we compare output $sC_{i}$ and input $C$, and find that
there is no pairwise equivalence between $sC_{i}$ and $C_{j}$ for
$i,j\in\left\{ 1,\dots,8\right\} $. Consequently, further concatenation
of strategies will lead to a defect reduction or yield a complete
quasicone. Thus we showed that there is a concatetanate strategy that
admits defect reduction in these cases either.
\noindent \begin{center}
\begin{table}[H]
\noindent \centering{}\caption{Computational results for algorithmic solution of finding successful
strategies for $A_{2}^{\left(1\right)}$, $A_{3}^{\left(1\right)}$
and $A_{4}^{\left(1\right)}$: number of quasicones with no successful
strategy in the approach set. The shortest strategy was defined below
Definition \ref{def: strategy} and the shortest long strategy is
according to Definition \ref{def: shortest long strategy}.}
\small %
\begin{tabular}{c>{\centering}p{0.13\columnwidth}>{\centering}p{0.13\columnwidth}>{\centering}p{0.13\columnwidth}>{\centering}p{0.13\columnwidth}>{\centering}p{0.13\columnwidth}}
\hline 
\noalign{\vskip3pt}
\multirow{2}{*}{$\mathfrak{g}$} & \multirow{2}{0.13\columnwidth}{total no. of considered quasicones } & \multicolumn{4}{c}{number of unsolved quasicones after applying ...}\tabularnewline
\noalign{\vskip5pt}
 &  & ... the shortest strategy & ... the shortest long strategy & ... a set of simple basic strategies & ... concatenations of simple basic strategies\tabularnewline
\hline 
\noalign{\vskip5pt}
$A_{2}^{\left(1\right)}$ & 48 & 32 & 0 &  & \tabularnewline[5pt]
\cline{2-6} 
\noalign{\vskip5pt}
$A_{3}^{\left(1\right)}$ & 669 & 242 & 38 & 8 & 0\tabularnewline[5pt]
\cline{2-6} 
\noalign{\vskip5pt}
$A_{4}^{\left(1\right)}$ & 23431 & 2747 & 536 & 65 & 8\tabularnewline[5pt]
\hline 
\noalign{\vskip5pt}
\end{tabular}\normalsize 
\end{table}
\par\end{center}

Recall the decomposition of the set of strategies $\mathfrak{S}=\bigcup_{\psi\in Q}\mathfrak{S}_{\psi}$.
In the $A_{4}^{\left(1\right)}$, concatenation of simple basic strategies
do no result in a successful strategy. The eight left-over cases are
solved manually according to the following paradigm, 

\textbf{1.} for every root $\varphi\in\Delta^{\circ}$ and $k\in\mathbb{Z}$,
the root step $\sigma\left(\varphi,k\right):\mathfrak{S}_{\psi}\rightarrow\mathfrak{S}$
sending $s\mapsto s\circ e_{\varphi+k\delta}$, seen as a partition
of the highest root, gives rise to a new partition, namely minus the
index of the component $\mathfrak{S}_{\psi+\varphi+k\delta}$, i.e.
$-\left(\psi+\varphi+k\delta\right)$, which is supposed to be a root

\textbf{2.} $\left(\varphi,k\right)$ is chosen such that $\mathrm{gap}\left(C\right)_{\left|\psi+\varphi\right|}=\max_{\nu\in I_{n}}\mathrm{gap}\left(C\right)_{\nu}$
and $k=C_{\varphi}-1$ (cf. \ref{eq:gap})

\begin{longtable}{ll}
1. & \small $\ensuremath{\left(\begin{array}{ccccc}
* & 1 & 1 & 0 & -1\\
2 & * & 1 & 1 & 0\\
1 & 2 & * & 1 & 0\\
2 & 1 & 1 & * & 1\\
2 & 2 & 2 & 1 & *
\end{array}\right)}$\normalsize $\overset{e_{3}\circ e_{-1}}{\rightsquigarrow}$\small $\ensuremath{\left(\begin{array}{ccccc}
* & 0 & 1 & 0 & -1\\
2 & * & 1 & 1 & 0\\
1 & 0 & * & 1 & 0\\
2 & 1 & 1 & * & 1\\
2 & 2 & 2 & 1 & *
\end{array}\right)}$\tabularnewline
\addlinespace[4pt]
2. & \small $\ensuremath{\left(\begin{array}{ccccc}
* & 1 & 0 & 1 & 1\\
2 & * & 1 & 1 & 2\\
2 & 2 & * & 1 & 2\\
1 & 1 & 1 & * & 1\\
1 & 0 & 0 & 1 & *
\end{array}\right)}$\normalsize $\overset{e_{-3}\circ e_{1}}{\rightsquigarrow}$\small $\ensuremath{\left(\begin{array}{ccccc}
* & 1 & 0 & 1 & 1\\
0 & * & -1 & 1 & 2\\
2 & 2 & * & 1 & 2\\
1 & 1 & 1 & * & 1\\
1 & 0 & 0 & 1 & *
\end{array}\right)}$\normalsize \tabularnewline
\addlinespace[4pt]
3. & \small $\ensuremath{\left(\begin{array}{ccccc}
* & 1 & 1 & 0 & 1\\
2 & * & 1 & 1 & 1\\
1 & 2 & * & 1 & 1\\
2 & 1 & 1 & * & 1\\
1 & 1 & 1 & 1 & *
\end{array}\right)}$\normalsize $\overset{e_{12}\circ e_{4}\circ e_{-2}\circ e_{-7}\circ e_{2}\circ e_{1}}{\rightsquigarrow}$\small $\ensuremath{\left(\begin{array}{ccccc}
* & 1 & 2 & 0 & 1\\
0 & * & 0 & -2 & -2\\
0 & 2 & * & -1 & 1\\
2 & 4 & 3 & * & 3\\
1 & 1 & 3 & 1 & *
\end{array}\right)}$\normalsize \tabularnewline
\addlinespace[4pt]
4. & \small $\ensuremath{\left(\begin{array}{ccccc}
* & 1 & 0 & 0 & 1\\
2 & * & 1 & 1 & 1\\
2 & 2 & * & 1 & 1\\
2 & 1 & 1 & * & 1\\
1 & 1 & 1 & 1 & *
\end{array}\right)}$\normalsize $\overset{e_{-2}\circ e_{-7}\circ e_{2}\circ e_{1}}{\rightsquigarrow}$\small $\left(\begin{array}{ccccc}
* & 1 & 0 & 0 & 1\\
0 & \mathbb{Z} & 1 & -2 & 1\\
0 & 2 & * & -1 & 1\\
2 & 1 & 1 & \mathbb{Z} & 1\\
1 & 1 & 1 & 1 & *
\end{array}\right)$\normalsize $\cong\mathfrak{g}$\tabularnewline
\addlinespace[4pt]
5. & \small $\ensuremath{\left(\begin{array}{ccccc}
* & 1 & 1 & 1 & 0\\
2 & * & 1 & 1 & 1\\
1 & 2 & * & 1 & 1\\
1 & 1 & 1 & * & 1\\
2 & 1 & 1 & 1 & *
\end{array}\right)}$\normalsize $\overset{e_{3}\circ e_{-1}}{\rightsquigarrow}$\small $\ensuremath{\left(\begin{array}{ccccc}
* & 0 & 1 & 1 & 0\\
2 & * & 1 & 1 & 1\\
1 & 0 & * & 1 & 1\\
1 & 1 & 1 & * & 1\\
2 & 1 & 1 & 1 & *
\end{array}\right)}$\normalsize \tabularnewline
\addlinespace[4pt]
6. & \small $\ensuremath{\left(\begin{array}{ccccc}
* & 1 & 0 & 1 & 0\\
2 & * & 1 & 1 & 1\\
2 & 2 & * & 1 & 1\\
1 & 1 & 1 & * & 1\\
2 & 1 & 1 & 1 & *
\end{array}\right)}$\normalsize $\overset{e_{-3}\circ e_{1}}{\rightsquigarrow}$\small $\ensuremath{\left(\begin{array}{ccccc}
* & 1 & 0 & 1 & 0\\
0 & * & -1 & 1 & 1\\
2 & 2 & * & 1 & 1\\
1 & 1 & 1 & * & 1\\
2 & 1 & 1 & 1 & *
\end{array}\right)}$\normalsize \tabularnewline
\addlinespace[4pt]
7. & \small $\ensuremath{\left(\begin{array}{ccccc}
* & 1 & 1 & 1 & 0\\
2 & * & 1 & 1 & 1\\
1 & 2 & * & 1 & 1\\
1 & 1 & 2 & * & 1\\
2 & 1 & 1 & 1 & *
\end{array}\right)}$\normalsize $\overset{e_{6}\circ e_{3}\circ e_{-6}\circ e_{-2}\circ e_{-8}\circ e_{-1}\circ e_{7}\circ e_{12}\circ e_{-7}\circ e_{2}\circ e_{1}}{\rightsquigarrow}$\small $\ensuremath{\left(\begin{array}{ccccc}
* & 2 & 2 & 2 & 2\\
-1 & * & 0 & 1 & 2\\
0 & 2 & * & 2 & 1\\
-1 & 1 & 0 & * & 2\\
0 & 0 & 1 & 0 & *
\end{array}\right)}$\normalsize \tabularnewline
\addlinespace[4pt]
8. & \small $\ensuremath{\left(\begin{array}{ccccc}
* & 1 & 1 & 0 & 0\\
2 & * & 1 & 1 & 0\\
1 & 2 & * & 1 & 0\\
2 & 1 & 2 & * & 1\\
2 & 2 & 2 & 1 & *
\end{array}\right)}$\normalsize $\overset{e_{-3}\circ e_{2}\circ e_{-12}\circ e_{-3}\circ e_{7}\circ e_{-2}\circ e_{-7}\circ e_{2}\circ e_{1}}{\rightsquigarrow}$\small $\ensuremath{\left(\begin{array}{ccccc}
* & 3 & 2 & 0 & 0\\
-1 & * & 0 & -2 & -2\\
0 & 2 & * & -1 & -1\\
2 & 4 & 3 & * & 1\\
2 & 4 & 3 & 1 & *
\end{array}\right)}$\normalsize \tabularnewline
\addlinespace[4pt]
\end{longtable}


\begin{thebibliography}{References}
\bibitem[BBFK13]{BekkertBenkartFutornyKashuba}Bekkert, V., Benkart,
G., Futorny, V., Kashuba, I., \emph{New Irreducible Modules for Heisenberg
and Affine Lie Algebras}, J. Algebra \textbf{373}, 284-298 (2013)

\bibitem[B09]{Bunke2009}Bunke, T., \emph{Classification of irreducible
non-dense modules of $A_{2}^{\left(2\right)}$}, Alg. and Discr. Math.,
\textbf{2}, 11\textendash 22 (2009)

\bibitem[Car]{Carter2005}Carter, R., \emph{Lie algebras of finite
and affine type}, Cambridge (2005)

\bibitem[ChP86]{CharyPressley1986}Chari, V., Pressley, A., \emph{Integrable
Representations of Twisted Affine Lie Algebras}, J. of Algebra \textbf{113},
438\textendash 46 (1986)

\bibitem[Ch86]{Chary1986}Chari, V., \emph{Integrable Representations
of Affine Lie Algebras}, Invent. math. \textbf{85}, 317\textendash 335
(1986)

\bibitem[De80]{Deodhar1980}Deodhar, V.V., \emph{On a Construction
of Representations and a Problem of Enright,} lnventiones math. \textbf{57},
101\textendash 118 (1980)

\bibitem[DG09]{DimitrovGrantscharov} Dimitrov, I., Grantcharov, D.,
\emph{Classification of simple weight modules for affine Lie algebras},
$\mathtt{arXiv:0910.0688v1}$ (2009)

\bibitem[DMP00]{DimitrovMathieuPenkov}Dimitrov, I., Mathieu., O.,
Penkov, I., \emph{On the structure of weight modules}, Trans. Amer.
Math. Soc. \textbf{352} , 2857-2869 (2000)

\bibitem[FRT08]{FeliksonRetakhTumarkin2008}Felikson, A., Retakh,
A., Tumarkin, P., \emph{Regular subalgebras of affine Kac Moody algebras},
J. Physics A: Mathematical and Theoretical \textbf{41}, no. 36 (2008)

\bibitem[Fe90]{Fernando1990} Fernando, S., \emph{Lie algebra modules
with finite dimensional weight spaces, I}, TAMS \textbf{322}, 757\textendash 781
(1990)

\bibitem[Fu92]{Futorny1992}Futorny, V.M., \emph{The parabolic subsets
of root system and corresponding representations of affine Lie algebras},
Contemp. Math. \textbf{131}, 45\textendash 52, Providence, RI, AMS
(1992)

\bibitem[Fu94]{Futorny1994}Futorny, V.M., \emph{Imaginary Verma modules
for affine Lie algebras}, Canad. Math. Bull. \textbf{37}, 213\textendash 218
(1994)

\bibitem[Fu96]{Futorny1996}Futorny, V.M., \emph{Irreducible non-dense
$A_{1}^{\left(1\right)}$-modules}, Pacific J. of Math. \textbf{172},
No. 1, 83\textendash 97 (1996)

\bibitem[Fu97]{Futorny1997}Futorny, V.M., \emph{Representations of
Affine Lie algebras}, Queen's Papers in Pure and Applied Mathematics
\textbf{106}, v. 1, Kingston (1997)

\bibitem[FuT01]{FutornyTsylke}Futorny, V.M., Tsylke, A., \emph{Irreducible
non-zero level modules with finite- dimensional weight spaces for
Affine Lie algebras}, J. of Algebra \textbf{238}, 426\textendash 441
(2001)

\bibitem[FK09]{FutornyKashuba2009}Futorny, V.M., Kashuba, I., \emph{Induced
Modules for Affine Lie Algebras}, SIGMA \textbf{5} 026, 14 pp.  (2009)

\bibitem[JK85]{JacobsonKac1985} Jakobsen, H.P. , Kac, V.G., \emph{A
new class of unitarizable highest weight representations of infinite-dimensional
Lie algebras}, Lecture Notes in Phys., \textbf{226}, 1\textendash 20,
Springer, Berlin (1985)

\bibitem[Kac]{Kac1983}Kac, V., \emph{Infinite-dimensional Lie algebras,
3rd Ed.}, Cambridge (1990)

\bibitem[Ma00]{Mathieu2000} Mathieu, O., \emph{Classification of
irreducible weight modules}, Ann. Inst. Fourier, Grenoble, \textbf{50},
2, 537\textendash 592 (2000)

\bibitem[MP95]{MoodyPianzola}Moody, R.V., Pianzola, A., \emph{Lie
algebras with triangular decompositions}, Can. Math. Soc. series of
monographs and advanced texts, John Wiley (1995)

\bibitem[Na84]{Nation1984} Nation, J.B., \emph{On partially ordered
sets embeddable in a free lattice}, Algebra Universalis \textbf{18},
327\textendash 333 (1984)

\bibitem[SpSt09]{SpeyerSturmfels2009} Speyer, D., Sturmfels, B.,
\emph{Tropical Mathematics}, Math. Mag. \textbf{82}, 163-173 (2009)

\bibitem[Ta71]{Takiff1971}Takiff, S.J., \emph{Rings of invariant
polynomials for a class of Lie algebras}, Trans. Amer. Math. Soc.
\textbf{160}, 249 \textendash 262 (1971)

\bibitem[Wa95]{Wallace1995} Wallace, R., \emph{Directed arc consistency
preprocessing}, Selected papers from the ECAI-94 Workshop on Constraint
Processing, M. Meyer, Ed., LNCS \textbf{923}, Springer, Berlin (1995)

\bibitem[Wi11]{Wilson2011}Wilson, B.J.,\emph{ Highest-Weight Theory
for Truncated Current Lie Algebras}, J. Algebra \textbf{336}, 1-27
(2011)
\end{thebibliography}
\end{document}